\documentclass[reqno,english]{amsart}
\usepackage{amsfonts,amsmath,latexsym,verbatim,amscd,mathrsfs,color,array}
\usepackage[colorlinks=true]{hyperref}
\usepackage{amssymb,amsthm,graphicx,color}
\usepackage[hmargin=3cm, vmargin=3cm]{geometry}
\usepackage{bbm}
\usepackage{amssymb}
\usepackage{pdfsync}
\usepackage{epstopdf}
\usepackage{cite}
\usepackage{graphicx}
\usepackage{multirow}
\usepackage[font=bf,aboveskip=15pt]{caption}
\usepackage[toc,page]{appendix}
\usepackage[colorlinks=true]{hyperref}
\usepackage{bookmark}
\DeclareFontShape{OT1}{cmr}{bx}{sc}{<-> cmbcsc10}{}
\newcommand{\la}{\lambda}

\newcommand{\R}{\mathbb R}

\newcommand{\bt}{\begin{theorem}}
\newcommand{\et}{\end{theorem}}
\newcommand{\bl}{\begin{lemma}}
\newcommand{\el}{\end{lemma}}
\newcommand{\bd}{\begin{definition}}
\newcommand{\ed}{\end{definition}}
\newcommand{\bc}{\begin{corollary}}
\newcommand{\ec}{\end{corollary}}
\newcommand{\bp}{\begin{proof}}
\newcommand{\ep}{\end{proof}}
\newcommand{\bx}{\begin{example}}
\newcommand{\ex}{\end{example}}
\newcommand{\bi}{\begin{exercise}}
\newcommand{\ei}{\end{exercise}}
\newcommand{\bo}{\begin{prop}}
\newcommand{\eo}{\end{prop}}
\newcommand{\br}{\begin{remark}}
\newcommand{\er}{\end{remark}}
\newcommand{\be}{\begin{equation}}
\newcommand{\ee}{\end{equation}}
\newcommand{\ba}{\begin{align}}
\newcommand{\ea}{\end{align}}
\newcommand{\bn}{\begin{enumerate}}
\newcommand{\en}{\end{enumerate}}
\newcommand{\bg}{\begin{align*}}
\newcommand{\bcs}{\begin{cases}}
\newcommand{\ecs}{\end{cases}}

\newcommand{\bean}{\begin{eqnarray*}}
\newcommand{\eean}{\end{eqnarray*}}

\newtheorem{definition}{Definition}[section]
\newtheorem{theorem}{Theorem}[section]
\newtheorem{lemma}{Lemma}[section]

\newtheorem{prop}{Proposition}[section]
\newtheorem{remark}{Remark}[section]

\numberwithin{equation}{section}

\begin{document}
\title[Blow-up for supercritical heat equation]{New type II Finite time blow-up for the energy supercritical heat equation}

\author[M. del Pino]{Manuel del Pino}
\address{\noindent
Department of Mathematical Sciences,
University of Bath, Bath BA2 7AY, United Kingdom}
\email{m.delpino@bath.ac.uk}

\author[C. Lai]{Chen-Chih Lai}
\address{\noindent
Department of Mathematics,
University of British Columbia, Vancouver, B.C., V6T 1Z2, Canada}
\email{chenchih@math.ubc.ca}

\author[M. Musso]{Monica Musso}
\address{\noindent
Department of Mathematical Sciences,
University of Bath, Bath BA2 7AY, United Kingdom}
\email{m.musso@bath.ac.uk}

\author[J. Wei]{Juncheng Wei}
\address{\noindent
Department of Mathematics,
University of British Columbia, Vancouver, B.C., V6T 1Z2, Canada}
\email{jcwei@math.ubc.ca}

\author[Y. Zhou]{Yifu Zhou}
\address{\noindent
Department of Mathematics,
University of British Columbia, Vancouver, B.C., V6T 1Z2, Canada}
\email{yfzhou@math.ubc.ca}

\begin{abstract}
We consider the energy supercritical heat equation with the $(n-3)$-th Sobolev exponent
\begin{equation*}
\begin{cases}
u_t=\Delta u+u^{3},~&\mbox{ in } \Omega\times (0,T),\\
u(x,t)=u|_{\partial\Omega},~&\mbox{ on } \partial\Omega\times (0,T),\\
u(x,0)=u_0(x),~&\mbox{ in } \Omega,
\end{cases}
\end{equation*}
where $5\leq n\leq 7$, $\Omega=\R^n$ or $\Omega \subset \R^n$ is a smooth, bounded domain enjoying special symmetries. We construct type II finite time blow-up solution $u(x,t)$ with the singularity taking place along an $(n-4)$-dimensional {\em shrinking sphere} in $\Omega$. More precisely, at leading order, the solution $u(x,t)$ is of the sharply scaled form
$$u(x,t)\approx \la^{-1}(t)\frac{2\sqrt{2}}{1+\left|\frac{(r,z)-(\xi_r(t),\xi_z(t))}{\la(t)}\right|^2}$$
where $r=\sqrt{x_1^2+\cdots+x_{n-3}^2}$, $z=(x_{n-2},x_{n-1},x_n)$ with $x=(x_1,\cdots,x_n)\in\Omega$. Moreover, the singularity location
$$(\xi_r(t),\xi_z(t))\sim (\sqrt{2(n-4)(T-t)},z_0)~\mbox{ as }~t\nearrow T,$$
for some fixed $z_0$, and the blow-up rate
$$\la(t)\sim \frac{T-t}{|\log(T-t)|^2}~\mbox{ as }~t\nearrow T.$$
This is a completely new phenomenon in the parabolic setting.
\end{abstract}
\maketitle

{
  \hypersetup{linkcolor=black}
  \tableofcontents
}

%%%%%%%%%%%%%%%%%%%%%%%%%%%%%%%%%%%%%%%%%

\section{Introduction}

\medskip

Semilinear heat equation
\begin{equation}\label{eqn-semilinear}
\begin{cases}
u_t=\Delta u+u^{p},~&\mbox{ in } \Omega\times (0,T),\\
u(x,t)=u|_{\partial\Omega},~&\mbox{ on } \partial\Omega\times (0,T),\\
u(x,0)=u_0(x),~&\mbox{ in } \Omega,
\end{cases}
\end{equation}
with $p>1$ has been widely studied since Fujita's celebrated work \cite{Fujita66}. Here $\Omega$ is the entire space $\R^n$ or a smooth bounded domain in $\R^n$ and $0<T\leq+\infty$. Many literatures have been devoted to studying this problem about the singularity formation, especially the blow-up rates, profiles and sets. See, for instance, \cite{Giga85CPAM,Giga86HMJ,Giga86CMP,Giga87Indiana,Giga89CPAM,Giga04Indiana,Merle04CPAM,Suzuki08Indiana,Merle09JFA,Merle11JFA,Collottype1,Collot17strong,Souplet2018optimal} and references therein. Also, for a comprehensive survey in the literature up to 2019, we refer the readers to the book of Quittner and Souplet \cite{Souplet07book}.

\medskip

For the finite time blow-up, it is said to be of
\medskip
\begin{itemize}
\item {\em type I~} if
$${\lim \sup}_{t\to T}(T-t)^{\frac{1}{p-1}}\|u(\cdot,t)\|_{\infty}<+\infty$$
\item {\em type II~} if
$${\lim \sup}_{t\to T}(T-t)^{\frac{1}{p-1}}\|u(\cdot,t)\|_{\infty}=+\infty.$$
\end{itemize}
\medskip
Type I blow-up is at most like that of the ODE $u_t=u^p$, while type II blow-up is much more difficult to detect. In particular, two different types blow-up phenomena in problem \eqref{eqn-semilinear} depend sensitively on the nonlinearity, namely the values of the exponent $p$. In \cite{Giga85CPAM}, Giga and Kohn first proved that for $1<p<\frac{n+2}{n-2}$, only type I blow-up can occur in the case of convex domain. This was generalized to radial case \cite{FHV00} and the case that the domain $\Omega$ is star-shaped \cite{cheng2013behavior}  for the energy critical case $p=\frac{n+2}{n-2}$. For the subcritical case $p<\frac{n+2}{n-2}$, multiple-point, finite time type I blow-up solution was found and its stability was further studied in \cite{Merle97Duke}. The critical exponent $p=\frac{n+2}{n-2}$ in this setting is special in various ways. For the critical case $p=\frac{n+2}{n-2}$, solutions were classified near the ground state of the energy critical heat equation in $\mathbb{R}^n$ with $n\geq 7$ in \cite{Collot17CMP}. For dimensions $ n\leq 6$ finite time blow-ups were constructed in \cite{FHV00} by matched asymptotics method.  (It is expected that no finite time blow-ups exist in dimensions $ n\geq 7$.) For the four dimensional energy critical heat equation, radial and sign-changing type II finite time blow-up solution was rigorously  constructed  in \cite{Schweyer12JFA} (See also \cite{ni4d}). In \cite{type25D}, the authors constructed type II finite time blow-up solution for the energy critical heat equation in $\mathbb{R}^5$ by the inner--outer gluing method. See also \cite{harada2019higher} for the construction of a higher speed type II blow-up in $\mathbb{R}^5$. In dimensions $n=3$ or $6$, recently finite time blow-ups are also found in \cite{type23D,harada6D}.

\medskip

In the aspect of infinite time blow-up, Cort\'azar, del Pino and Musso \cite{Green16JEMS} constructed positive and non-radial infinite time blow-up solution for problem \eqref{eqn-semilinear} with Dirichlet boundary condition and $n\geq 5$. The solution they constructed takes the profile of sharply scaled Aubin-Talenti bubbles
\begin{equation}\label{bubbles}
U_{\la,\xi}(x)=\la^{-\frac{n-2}{2}}U\left(\frac{x-\xi}{\la}\right)=\varpi_n \left(\frac{\la}{\la^2+|x-\xi|^2}\right)^{\frac{n-2}{2}},~\varpi_n=(n(n-2))^{\frac{n-2}{4}}
\end{equation}
which solve the Yamabe problem
$$\Delta U+ U^{\frac{n+2}{n-2}}=0~\mbox{ in }~\R^n.$$
Moreover, the blow-up position for the solution is determined by the Green's function of $-\Delta$ in $\Omega$, while the role of the Green's function in bubbling phenomena for elliptic problems has been known for a long time since the pioneering works \cite{BC88CPAM,Bahri95CVPDE}. In \cite{FilaKing12}, Fila and King studied problem \eqref{eqn-semilinear} with the critical exponent $p=\frac{n+2}{n-2}$ and gave insight on the infinite time blow-up in the case of a
radially symmetric, positive initial condition with an exact power decay rate. By using formal matched
asymptotic analysis, they demonstrated that the blow-up rate is  determined by the power decay in a precise
manner. Intriguingly enough, their analysis leads them to conjecture that infinite time blow-up should
only happen in low dimensions 3 and 4, see Conjecture 1.1 in \cite{FilaKing12}. Recently, this is confirmed and rigorously proved in \cite{173D} for $n=3$. In a recent work \cite{del2018sign}, the authors constructed non-radial and sign-changing solution which blows up at infinite time, where at leading order, the solution takes the form of the sign-changing ``cell'' constructed in \cite{cell1,cell2} instead of the Aubin-Talenti bubble.

\medskip

The singularity formation for supercritical case $p>\frac{n+2}{n-2}$ is much more intricate. Herrero and Vel\'azquez \cite{herrero1992blow,Velazquez94pJL} found type II blow-up solution in the radial class for $n\geq 11$ and $p>p_{JL}(n)$. Here $p_{JL}(n)$ is the Joseph-Lundgren exponent
$$
p_{JL}(n):=\begin{cases}
\infty~&\mbox{ if }~3\leq n\leq 10,\\
1+\frac{4}{n-4-2\sqrt{n-1}}~&\mbox{ if }~n\geq11.
\end{cases}
$$
The solution locally resembles a asymptotically singular scaling of a radial and positive solution to the stationary problem
$$\Delta u+u^p=0~\mbox{ in }~\R^n.$$
See also \cite{Mizoguchi05Indiana} for the case that $\Omega$ is a ball, and \cite{Collot17APDE} for the case of general domain with the restriction that $p$ is an odd integer. For the borderline case $p=p_{JL}(n)$ and $n\geq 11$, the existence of type II blow-up was shown in \cite{Seki18JFA}. In the Matano-Merle range $\frac{n+2}{n-2}<p< p_{JL}(n)$ which complements the Herrero-Vel\'azquez range, no type II blow-up can occur in radially symmetric class in the case of a ball or in entire space under additional assumptions \cite{Merle04CPAM,Merle09JFA,Mizoguchi11JDE}. In \cite{17type2}, the authors successfully constructed non-radial type II blow-up solution to \eqref{eqn-semilinear} in the Matano-Merle range $p=\frac{n+1}{n-3}\in\left(\frac{n+2}{n-2},p_{JL}(n)\right)$. More precisely, the solution constructed in \cite{17type2} blows up along a certain curve with axial symmetry in the sense that the energy density approaches to the Dirac measure along the curve as $t\nearrow T$. See also \cite{Collot17strong} for another kind of anisotropic blow-up for the case $n\geq 12$ and $p>p_{JL}(n)$.

\medskip

Singularity formation triggered by criticality and super-criticality in many other literatures has also been widely studied, such as the Schr\"odinger map, wave equations, Yang-Mills problems, geometric flows such as harmonic map heat flows, mean curvature flows and Yamabe flows. We refer the readers for instance to \cite{CDY92JDG,Velazquez97JRAM,18type2Yamabeflow,Collot17CMP,Kenig11JEMS,King03JDE,Jendrej17JFA,Merle08Acta,Krieger09Duke,Merle13Inventiones,Raphael12wavemaps,Krieger08Inventiones,Krieger09Advance,Biernat17IMRN,Ghoul19APDE,Naito94,Grotowski01MZ}, and the references therein.

\medskip

In this paper, we are concerned with the energy supercritical heat equation with the $(n-3)$-th Sobolev exponent $p=3$
\begin{equation}\label{eqn}
\begin{cases}
u_t=\Delta u+u^{3},~&\mbox{ in } \Omega\times (0,T),\\
u(x,t)=u|_{\partial\Omega},~&\mbox{ on } \partial\Omega\times (0,T),\\
u(x,0)=u_0(x),~&\mbox{ in } \Omega,
\end{cases}
\end{equation}
where $5\leq n\leq 7$, $\Omega$ is $\R^n$ or $\Omega$ is  a smooth, bounded domain in $\R^n$ enjoying the symmetry that $\Omega$ is invariant under the orthogonal transformations
\begin{itemize}
\item $Q(x_1,\cdots,x_n)=\big(R(x_1,\cdots,x_{n-3}),x_{n-2},x_{n-1},x_n\big),$
with
$$x=(x_1,\cdots,x_n)\in\Omega,~~R\in SO(n-3),$$
where $SO(n-3)$ is the classical rotation group

\item $\pi_i(x_1,\cdots,x_i,\cdots,x_n)=(x_1,\cdots,-x_i,\cdots,x_n)$ \,with\, $i=n-2,~n-1,~n$
\end{itemize}
namely
\begin{equation}\label{domain-symmetry}
Q(\Omega)=\Omega,~\pi_i(\Omega)=\Omega~\mbox{ for }~i=n-2,~n-1,~n.
\end{equation}
In other words, $\Omega$ is a radial domain in the first $n-3$ coordinates and even in the remaining $3$ coordinates.

\medskip

The aim of this paper is to construct {\em type II finite time blow-up} solution which blows up along an {\em $(n-4)$-dimensional sphere with shrinking size $\sqrt{T-t}$}. To state the main result, we write $x=(x^*,x^{**})\in\Omega$ with $x^*\in \R^{n-3}$ and $x^{**}\in \R^{3}$, and denote $|x^{*}|=r$, $x^{**}=z$. We look for solution $u(x,t)$ with the same symmetry as $\Omega$'s
\begin{equation}\label{symmetryclass}
u(x,t)=\tilde u(r,z,t)
\end{equation}
for a function $\tilde u$ defined in $\mathcal D\times(0,T)$, where
\begin{equation}\label{def-domain}
\mathcal D:=\left\{(r,z):r=\sqrt{x_1^2+\cdots+x_{n-3}^2},z=(x_{n-2},x_{n-1},x_n)~\mbox{ such that }~x=(x_1,\cdots,x_n)\in\Omega\right\}.
\end{equation}
Our main result is stated as follows.

\medskip

\begin{theorem}\label{thm}
Let $\Omega$ be $\mathbb R^n$ or a smooth, bounded domain in $\mathbb R^n$ in the symmetry class \eqref{domain-symmetry} with $5\leq n\leq 7$. Then for $T>0$ sufficiently small, there exist initial and boundary conditions such that the solution $u(x,t)$ to problem \eqref{eqn} blows up along a shrinking sphere, with the profile of the form
$$u(x,t)=\la^{-1}(t)U\left(\frac{(r,z)-(\xi_r(t),\xi_z(t))}{\la(t)}\right)(1+o(1))~\mbox{ as }~t\nearrow T,$$
where $U$ is the standard bubble \eqref{bubbles} in $\mathbb R^4$, and for some $\kappa_*>0$, $z_0\in\mathbb R^3$,
$$\la(t)=\kappa_* \frac{T-t}{|\log(T-t)|^2}(1+o(1))~\mbox{ as }~t\nearrow T,$$
$$\xi(t)=\left(\xi_r(t),\xi_z(t)\right)=\left(\sqrt{2(n-4)(T-t)}(1+o(1)),z_0(1+o(1))\right)~\mbox{ as }~t\nearrow T.$$
\end{theorem}

\medskip

Theorem \ref{thm} exhibits a completely new blow-up phenomenon where the type II blow-up takes place along a {\it higher dimensional manifold with shrinking size}. More precisely, from the characterization of $\xi(t)=\left(\xi_r(t),\xi_z(t)\right)$, the blow-up position of the solution $u(x,t)$ we construct is a copy of $\mathbb S^{n-4}$, and the $(n-4)$-dimensional sphere shrinks with self-similar size $\sqrt{T-t}$ and asymptotically collapses to a point $(0,z_0)$ in $\Omega$ as $t\nearrow T$. A schematic depiction of the evolution of the concentration set for $n=5$ is given in Figure \ref{Fig.1} below. Roughly speaking, the shape of the solution we construct looks like a ``thin tube'' in the $(r,z)$-coordinate for $n=5$. This type of concentration set with shrinking size was first conjectured to exist in \cite{3DHMF} in the context of harmonic map heat flow, where the authors considered the case that the singularity set is a fixed circle. We will come back to this intriguing question in the setting of harmonic map heat flow in a forthcoming work. Also, in the setting of energy supercritical heat equation \eqref{eqn-semilinear} with $p=\frac{n+1}{n-3}$ (the second Sobolev exponent), type II finite time blow-up solution concentrating on a fixed circle was constructed in \cite{17type2}. In another aspect, the higher dimensional blow-ups for parabolic problems can be regarded as the parallel with bubbling phenomena in the elliptic setting, see \cite{DKW07CPAM,bubbling10JEMS,DKW11Annals} for example. It is worth mentioning that the boundary bubbling driven by the geometry of the boundary in \cite{bubbling10JEMS} was also conjectured to be true for parabolic problems in \cite{17type2}.

\medskip

In fact, there are similar phenomena as shown in Theorem \ref{thm} in other literatures. In \cite{TaoEuler}, a very similar ``neck pinch'' blow-up with self-similar size $\sqrt{T-t}$ was already found for the finite time singularity formation of the generalized Euler equations in dimension three (see \cite[Theorem 1.11]{TaoEuler}). A similar relation $\dot\xi_r\sim \frac{1}{\xi_r}$, which is the key to obtain novel dynamics of the shrinking concentration set in Theorem \ref{thm}, seems to be also crucially used to produce such ``neck pinch'' blow-up in \cite[Section 7]{TaoEuler}.  In the context of the finite time blow-up for nonlinear Schr\"odinger equations (NLS), the shrinking (or collapsing) concentration sets have been found in \cite{collapsingring,contractingsphere} (see also \cite{WangRing} for the numerical simulations). Note that the solutions constructed in \cite{collapsingring,contractingsphere} are essentially radially symmetric and the shrinking concentration sets there are not of self-similar size $\sqrt{T-t}$, while the solution constructed in Theorem \ref{thm} is cylindrically symmetric and $\xi_r(t)\sim \sqrt{T-t}$.
%In other words, a function is said to be cylindrically symmetric if it is symmetric with respect to $\theta$ in the cylindrical coordinates $x=(r,z,\theta)\in[0,+\infty)\times\R\times\mathbb S^{n-2}$.
At the level of NLS with cylindrical symmetry, the standing ring (with fixed radius) blow-ups have been investigated in \cite{HolmerRing,ZwiersRing} for example.

\medskip

\begin{figure}[htbp]
\centering
\includegraphics[width=140mm]{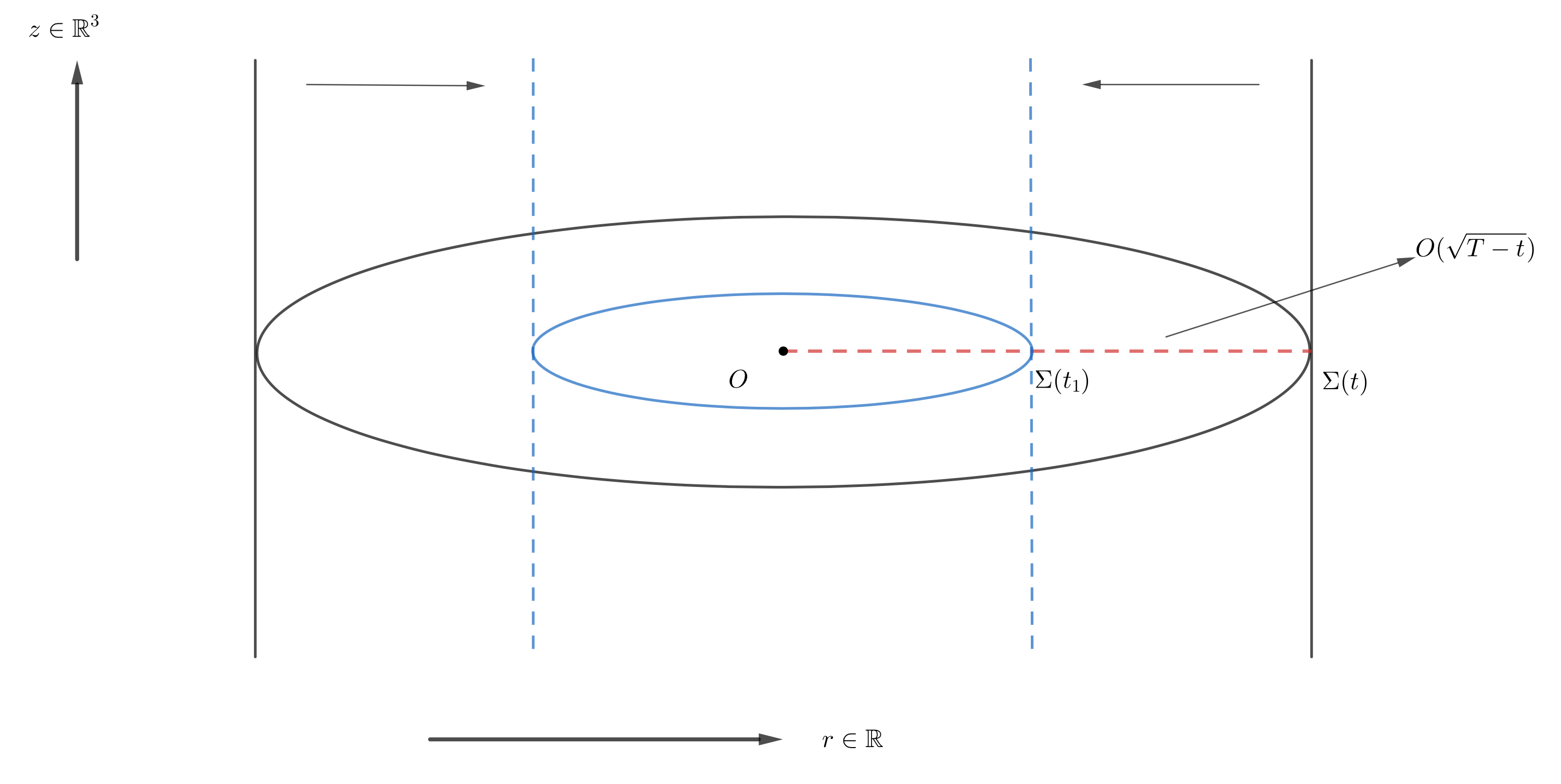}
\caption{A schematic depiction of the concentration set in Theorem \ref{thm} for $n=5$. In this case, the concentration set is a {\em shrinking circle}. The solution constructed in Theorem \ref{thm} blows up along the  shrinking circle $\Sigma(t)$. For $0<t<t_1<T$, the concentration set evolves from $\Sigma(t)$ to $\Sigma(t_1)$ (shown here as the blue curve) as $t\to t_1$, and finally {\em collapses} to a point as $t\to T$.}
\label{Fig.1}
\end{figure}

The proof of Theorem \ref{thm} is based on the {\it inner--outer gluing method}, which has been a very powerful tool in constructing solutions in many elliptic problems, see for instance \cite{DKW07CPAM,DKW11Annals,bubbling10JEMS,manifold15JMPA} and the references therein. Also, this method has been successfully applied to various parabolic problems recently, such as the infinite time and finite time blow-ups for energy critical heat equations \cite{Green16JEMS,173D,del2018sign,17type2,type25D,tower7D}, singularity formation for the harmonic map heat flows \cite{17HMF,sire2019singularity}, type II ancient solution for the Yamabe flow \cite{18type2Yamabeflow}, the vortex dynamics in Euler flows \cite{18Euler}, and the finite time singularity for the nematic liquid crystal flow \cite{LCF2D}. In the fractional setting, a new fractional gluing method was developed in
\cite{17halfHMF} to construct infinite time blow-up for the half-harmonic map flow, and also in \cite{18fractionalcritical} for the constructions of infinite time blow-up for the fractional critical heat equation. We refer the interested readers to a survey \cite{delPinosurvey} for more results in parabolic settings.

\medskip

In the symmetry class \eqref{symmetryclass}, problem \eqref{eqn} is reduced to solving
\begin{equation}\label{eqn''}
\begin{cases}
\tilde u_t=\Delta_{(r,z)} \tilde u+\frac{n-4}{r}\tilde u_r+\tilde u^{3},~&\mbox{ in } \mathcal D\times (0,T)\\
\tilde u_r=0,~&\mbox{ on } (\mathcal D\cap\{r=0\})\times (0,T)\\
\tilde u=\tilde u|_{\partial\Omega},~&\mbox{ on } (\partial \mathcal D\backslash \{r=0\})\times (0,T)\\
\tilde u(\cdot,0)=\tilde u_0,~&\mbox{ in } \mathcal D\\
\end{cases}
\end{equation}
where $\mathcal D$ is defined in \eqref{def-domain} and $\Delta_{(r,z)}:=\partial_r^2+\Delta_z$ is the Laplacian in $\R^4$. Note that $p=3$ is the critical exponent in $\R^4$. So problem \eqref{eqn''} can be viewed as the energy critical problem in $\R^4$ with a perturbation $\frac{n-4}{r}\tilde u_r$. It turns out that the term $\frac{n-4}{r}\tilde u_r$ plays a crucial role in producing the shrinking concentration set. The first step of the construction consists of choosing a suitable approximate solution with sufficiently small error and then decomposing the original problem into inner and outer problems, where the inner problem is essentially the linearization around the bubble supported in a well-chosen ball. The inner and outer problems will be solved by developing linear theories for the associated linear problems and the Schauder fixed point theorem. Since the desired blow-up solution is located exactly at the self-similar regime $r=O\left(\sqrt{T-t}\right)$, the estimates for the linear outer problem are much more delicate than the case that the concentration set is fixed ($r=O(1)$). On the other hand, it is a difficult issue to control the term of type $\frac{n-4}{r}\tilde u_r$ due to the shrinking effect $r\sim\sqrt{T-t}\to 0$ as $t\nearrow T$. As a consequence, in order to carry out the fixed point argument under suitable topology, the estimates for the inner and outer problems should be very refined. For the outer problem, we achieve this by using the Duhamel's formula in $\R^n$, while for the inner problem, motivated by \cite{Green16JEMS,17HMF}, we improve the linear theory by decomposing the linearized problem into three different modes: scaling mode, translation mode and higher modes. The most difficult modes are scaling mode and translation mode. To find inner solution with proper space-time decay, we shall carry out a new inner--outer gluing scheme for the scaling mode, while the estimates for the translation mode are obtained by the blow-up argument. As mentioned, the inner problem is supported in a ball with radius $R=R(t)$ in terms of the translated and scaled variable. To ensure the inner--outer gluing to be carried out, the radius $R(t)$ and the parameters in the norms will be very carefully chosen, which results in the dimension restriction $5\leq n\leq 7$. This seems to be reasonable since the singularity takes place exactly at the critical level $r\sim\sqrt{T-t}$, and from our computations, the estimates for the key coupling terms in the inner--outer gluing system get worse as $n$ increases. In other words, upper bound of the dimension $n$ should be present in this setting to ensure the implementation of the inner--outer gluing procedure. We believe that the blow-up on higher dimensional shrinking sphere with $n\geq 8$ should exist, presumably with more complicated blow-up rate.

\medskip

Note that the supercritical problem \eqref{eqn} is in the Matano--Merle range $\frac{n+2}{n-2}<p<p_{JL}$. It was proved that no type II blow-up is present for {\em radial} solutions in the case of a ball or in entire space under additional assumptions \cite{Merle04CPAM,Merle09JFA,Mizoguchi11JDE}, while the solution constructed in Theorem \ref{thm} is certainly not radially symmetric as the translation mode also plays a crucial role in the construction resulting in the novel dynamics of the concentration set. In this aspect, the results in Theorem \ref{thm} share a similar flavor as that of \cite{17type2}.

\medskip

We close the introduction by mentioning that
%the inner-outer gluing method can be easily carried out to recover the results in \cite{Schweyer12JFA}, where radial and sign-changing type II blow-up for the 4-dimensional energy critical heat equation was constructed by energy method. Moreover, by the inner--outer gluing method, the sign-changing type II blow-up solution is not necessarily radial. Besides, we note that
for problem \eqref{eqn}, type II finite time blow-up on a sphere with fixed instead of shrinking size also exists, and the upper bound for the dimension $n$ for this phenomenon to exist should be greater than $8$. We will not elaborate on this problem since it is a rather direct consequence of our construction here.

\medskip

Since the proof of Theorem \ref{thm} is quite long and technical, we shall give a brief sketch to illustrate the ideas in next section.

\medskip

%%%%%%%%%%%%%%%%%%%%%%%%%%%%%%%%%%%%%%%%%%%%%%%%%%%%%%%%%%%%%%%%%%%%%%%%%%%%%%%%%

\section{Sketch of the proof}

\medskip

In this section, we shall sketch the major steps in our construction.

\medskip

\noindent {\bf Step 1. Approximate solution}

\medskip

We define the error operator
$$\mathcal S (u):=-u_t+\Delta_{(r,z)} u+\frac{n-4}{r}u_r+u^3.$$
Then looking for a solution to \eqref{eqn} in the symmetry class \eqref{symmetryclass} is equivalent to finding $u$ such that
$$\mathcal S (u)=0,$$
where we drop the tilde for simplicity. The bubble
\begin{equation*}
U(y)=\frac{\alpha_{0}}{1+|y|^2}
\end{equation*}
solves the Yamabe problem
\begin{equation*}
\Delta_y U+U^3=0~\mbox{ in }~\R^4,
\end{equation*}
where $\alpha_0=2\sqrt{2}.$ Our first approximation is chosen as
$$U^*=\eta_*U_{\la(t),\xi(t)}~\mbox{ with }~U_{\la(t),\xi(t)}=\la^{-1}(t)U\left(\frac{(r,z)-\xi(t)}{\la(t)}\right)~\mbox{ and }~\eta_*:=\eta\left(\frac{|(r,z)-\xi(t)|}{\delta \sqrt{T-t}}\right)$$
for $\delta>0$ fixed small, where $\eta$ is the standard cut-off such that $\eta(s)=1$ for $s<1$ and $\eta(s)=0$ for $s>2$. The purpose of the extra cut-off $\eta_*$ in front of the bubble is to make the desired solution more concentrated near the shrinking sphere, and it also enables us to better control the terms of type $\frac{n-4}{r}\partial_r u$ as $r\to 0^+$. Next, the error of $U^*$ is
\begin{equation*}
\begin{aligned}
\mathcal S(U^*)=&~ \eta_*\left[\la^{-2}(t)\dot{\la}(t)\left(-\frac{\alpha_0}{1+|y|^2}+\frac{2\alpha_0}{(1+|y|^2)^2}\right)+\la^{-2}(t)\nabla_y U(y)\cdot \dot{\xi}(t)\right]\\
&~+U_{\la,\xi}(\Delta_{(r,z)}\eta_*-\partial_t \eta_*)+2\nabla \eta_*\cdot \nabla U_{\la,\xi}+(\eta_*^3 -\eta_* )U_{\la,\xi}^3 +\frac{n-4}r \partial_r(\eta_* U_{\la,\xi}),
\end{aligned}
\end{equation*}
where $\nabla:=(\partial_r, \nabla_z)$ and $y=\frac{(r,z)-\xi(t)}{\la(t)}$. Notice that the slow decaying error in $\mathcal S(U^*)$ is
$$\mathcal E_0=-\frac{\alpha_0\dot\la(t)}{\la^2(t)+\rho^2}\approx -\frac{\alpha_0\dot \la(t)}{\rho^2},$$
where $\rho=\la|y|.$ To reduce the size of $\mathcal E_0$, a correction $\Psi_0$ solving
$$\partial_t u=\Delta_{(r,z)} u+\mathcal E_0~\mbox{ in }~\R^4\times(0,T)$$
is added. Moreover, $\Psi_0$ takes the integral form as \eqref{def-Psi0}. So we take the second approximation as
$$u^*=U^*+\eta_*\Psi_0$$
and its error is computed in \eqref{error22}. To further reduce the size of the error $\mathcal S(u^*)$ to carry out the inner--outer gluing scheme, we add one more correction $\Psi_1$ as in \eqref{def-Psi1}. Then our corrected approximate solution becomes
$$u_c^*=U^*+\eta_*\Psi_0+\Psi_1.$$

\medskip

\noindent {\bf Step 2. The inner--outer gluing scheme}

\medskip

Now, we write
$$u=u_c^*+\mathtt P,$$
and our aim is to find the perturbation $\mathtt P$ satisfying
$$\mathcal S(u_c^*+\mathtt P)=0.$$
We decompose $\mathtt P$ into inner and outer profiles
$$\mathtt P=\mathtt P_{in}+\mathtt P_{out}$$
with
$$\mathtt P_{in}=\la^{-1}(t)\eta_R \phi(y,t)~\mbox{ and }~\mathtt P_{out}=\psi+Z^*,$$
where $Z^*$ satisfies \eqref{eqn-Z^*} and
$$\eta_R:=\eta\left(\frac{|(r,z)-\xi(t)|}{\la(t)R(t)}\right).$$
Write $y=\frac{(r,z)-\xi(t)}{\la(t)}$, $\xi(t)=(\xi_r(t),\xi_z(t))$ and $\Psi^*=\psi+Z^*$. Then a desired solution $u=u_c^*+\mathtt P$ is found if $\phi$ and $\psi$ solve essentially the inner--outer gluing system
\begin{equation}\label{sketch-inout}
\left\{
\begin{aligned}
&\la^2 \phi_t = \Delta_y \phi +3U^2(y)\phi+\mathcal H(\phi,\psi,\la,\xi)~~\mbox{ in }\mathcal D_{2R}\times (0,T)\\
&\psi_t =\Delta_{(r,z)} \psi+\frac{n-4}{r}\partial_r \psi  + \mathcal G(\phi,\psi,\la,\xi)~~\mbox{ in }\mathcal D\times (0,T)\\
\end{aligned}
\right.
\end{equation}
where $\mathcal H,~\mathcal G,~\mathcal D$ are defined in \eqref{def-mH}, \eqref{def-mG}, \eqref{def-domain}, and
$$\mathcal D_{2R}=\left\{(r,z)\in \R^4:|(r,z)-(\xi_r,\xi_z)|\leq 2\la(t) R(t)\right\}.$$
The inner-outer gluing system \eqref{sketch-inout} shall be solved by developing linear theories for the associated linear problems and then applying the fixed point argument. See Section \ref{sec-inneroutergluingscheme} for more detailed strategies for solving this system.

\medskip

\noindent {\bf Step 3. Choosing parameter functions $\la(t)$ and $\xi(t)$ at leading order}

\medskip

For the inner probelm
$$\la^2 \phi_t = \Delta_y \phi +3U^2(y)\phi+\mathcal H(\phi,\psi,\la,\xi)~~\mbox{ in }\mathcal D_{2R}\times (0,T),$$
we notice that the parabolic operator $-\la^2\partial_t+L_0$ is certainly not invertible since all the time independent elements in the 5 dimensional kernel of $L_0$ (see \eqref{kernels}) also belong to the kernel of $-\la^2\partial_t+L_0$. Here $L_0=\Delta_y +3U^2$ is the linearized operator around the bubble. Therefore, certain orthogonality conditions are expected to guarantee the existence of inner solution $\phi$ with suitable space-time decay. By the linear theory which will be developed in Section \ref{sec-linearinner}, approximately, orthogonality conditions
$$
\int_{\R^4} \mathcal H(\phi,\psi,\la,\xi) Z_j (y) dy\approx 0~~\mbox{ for all } j=1,\cdots,5,~  t\in (0,T)
$$
should be satisfied, where $Z_j$ are kernel functions defined in \eqref{kernels}. By singling out the main term $\mathcal H_*[\la,\xi,\Psi^*]$ in $\mathcal H(\phi,\psi,\la,\xi)$,
 it is thus reasonable to find the leading orders $\la_*(t),~\xi_*(t)$ of the scaling and translation parameters $\la(t),~\xi(t)$ such that
$$
\int_{\R^4} \mathcal H_*(\la_*,\xi_*,\Psi^*) Z_j (y) dy= 0~~\mbox{ for all } j=1,\cdots,5,~  t\in (0,T).
$$
The orthogonality conditions above imply that
\begin{equation*}
\left\{
\begin{aligned}
&\dot{\xi}_r(t) +\frac{n-4}{\xi_r(t)}=o(1)\\
&\dot\xi_{z_j}(t)=o(1),\quad j=2,3,4,\\
&\int_{-T}^{t-\la^2(t)}\frac{\dot{\la}(s)}{t-s}ds=-c+o(1)\\
\end{aligned}
\right.
\end{equation*}
where $o(1)\to 0$ as $t\nearrow T$ and $c>0$. Note that the operator $\frac{n-4}{r}\partial_r$ resulting from the symmetry \eqref{symmetryclass} plays a crucial role in generating the novel dynamics for $\xi_r(t)$, and the integro-differential equation for $\la(t)$ is due to the non-local correction $\Psi_0$. After some effort, the integral equation can be approximated by
$$\dot{\la}(t)=-\frac{c}{|\log(T-t)|^2}.$$
Therefore, to make $\mathcal H$ as orthogonal as possible to the kernel functions $Z_j$ with $j=1,\cdots,5$, good choices for the leading orders $\la_*(t)$ and $\xi_*(t)=(\xi_{r,*}(t),\xi_{z,*}(t))$ are
\begin{equation*}
\left\{
\begin{aligned}
&\xi_{r,*}(t)=\sqrt{2(n-4)(T-t)}\\
&\xi_{z,*}(t)=z_0\\
&\la_*(t)=\frac{|\log T|(T-t)}{|\log(T-t)|^2}\\
\end{aligned}
\right.
\end{equation*}
where $z_0$ is a given point in $\R^3$. See Section \ref{sec-laxi} for detailed derivation of $\la_*(t)$ and $\xi_*(t)$. The remainders for $\la(t)$ and $\xi(t)$ will be completely solved in Section \ref{sec-innerouter}.

\medskip

\noindent {\bf Step 4. Linear theory for the outer problem}

\medskip

To solve the outer problem, we shall develop a linear theory for the associated linear outer problem in Section \ref{sec-linearouter}.

Since the desired blow-up solution concentrates on an $(n-4)$-dimensional sphere with shrinking size $\sqrt{T-t}$, suitable estimates for the outer solution $\psi$ are affected by the shrinking effect and thus very delicate to find. To achieve this, we find solutions in the symmetry class \eqref{symmetryclass} by using the Duhamel's formula in $\R^n$. We consider the model linear problem
\begin{equation*}
\begin{cases}
\psi_t=\Delta_{(r,z)} \psi+\frac{n-4}{r}\partial_r \psi +\mathbf {f}_{\rm out},~&\mbox{ in }\mathcal D\times (0,T)\\
\psi=0,~&\mbox{ on }(\partial \mathcal D\backslash \{r=0\})\times (0,T)\\
\psi_r=0,~&\mbox{ on } (\mathcal D\cap\{r=0\})\times (0,T)\\
\psi(r,z,0)=0,~&\mbox{ in }\mathcal D\\
\end{cases}
\end{equation*}
where the non-homogeneous term $\mathbf f_{{\rm out}}$ is assumed to be bounded with respect to the weights appearing in $\mathcal G$ defined by \eqref{def-mG} in the outer problem. By using the Duhamel's formula, we obtain the following estimate
$$\|\psi\|_*\lesssim \|\mathbf {f}_{\rm out}\|_{**},$$
where the weighted norms $\|\cdot\|_*$ and $\|\cdot\|_{**}$ are defined in \eqref{def-norm*} and \eqref{def-norm**}. The proofs of these technical estimates will be postponed to the Appendix.

\medskip

\noindent {\bf Step 5. Linear theory for the inner problem}

\medskip

We then want to develop a linear theory for the associated linear inner problem.

Our strategy is to decompose the inner problem into three different Fourier modes
\medskip
\begin{itemize}
\item Scaling mode: $\la^2\phi^0_{t}=\Delta_y\phi^0 + 3U^{2}(y)\phi^0 + h^0(y,t)+c^0(t)Z_5(y)~\mbox{ in }~\mathcal D_{2R}\times(0,T)$

\medskip

\item Translation mode: $\la^2\phi^1_{t}=\Delta_y\phi^1 + 3U^{2}(y)\phi^1 + h^1(y,t)+\sum_{\ell=1}^4 c^{\ell}(t)Z_{\ell}(y)~\mbox{ in }~\mathcal D_{2R}\times(0,T)$

\medskip

\item Higher modes: $\la^2\phi^{\perp}_{t}=\Delta_y\phi^{\perp} + 3U^{2}(y)\phi^{\perp} + h^{\perp}(y,t)~\mbox{ in }~\mathcal D_{2R}\times(0,T)$
\end{itemize}
\medskip
and we shall construct inner solution in each mode. For the scaling mode, we further decompose $\phi^0$ into inner and outer profiles
$$\phi^0=\phi^0_{out}+\eta_{R_*}\phi^0_{in},$$
where $\eta_{R_*}:=\eta\left(\frac{|y|}{R_*}\right)$ with the standard cut-off $\eta$ defined in \eqref{def-cutoff} and $R_*=R^{\sigma}~\mbox{ for }~\sigma\in(0,1).$ As mentioned, due to the shrinking effect of the concentration set, the estimates required for the inner--outer gluing to work should be very refined. This is the reason why we shall carry out a new inner--outer gluing scheme at scaling mode. As a result, the estimates we get are only deteriorated inside the inner ball, and adjusting $\sigma$ enables us to better control the estimates. On the other hand, we observe that the kernel functions $Z_{\ell}(y)$ ($\ell=1,\cdots,4$) corresponding to the translation mode are of fast decay
$$Z^{\ell}(y)\sim\frac{1}{|y|^3}~\mbox{ as }~|y|\to+\infty,$$
which enables us to get the estimates for $\phi^1$ by the blow-up argument inspired by \cite{17HMF}. Here we need a technical restriction that $a_1\in(1,2)$ to ensure the integrability in the argument, where $a_1$ is one of the parameters in the $\|\cdot\|_{\nu_1,a_1}$-norm measuring the translation mode. See Section \ref{sec-linearinner} for more details.

\medskip

\noindent {\bf Step 6. Solving the inner--outer gluing system}

\medskip

Our aim is to solve the inner--outer gluing system \eqref{sketch-inout} by applying the linear theories developed for inner and outer problems and also the Schauder fixed point theorem. This is the context of Section \ref{sec-innerouter}.

For the outer problem, we estimate $\|\mathcal G\|_{**}$ where the $\|\cdot\|_{**}$-norm is the norm we find in the linear theory for the outer problem in Section \ref{sec-linearouter}. For the inner problem, we estimate $\mathcal H$ in the norms we find for different modes in the linear theory in Section \ref{sec-linearinner}, namely, $\|\mathcal H^0\|_{\nu,2+a},~\|\mathcal H^1\|_{\nu_1,2+a_1},~\|\mathcal H^{\perp}\|_{\nu,2+a}$. To gain smallness in the contraction, restrictions for all the parameters in the norms are required, from which we obtain the dimension restriction $5\leq n\leq 7$. We then need to adjust $\la(t)$ and $\xi(t)$ such that the orthogonality conditions hold. The reduced equation for $\xi(t)$ is direct to solve, while the reduced equation of $\la(t)$ turns out to be an integro-differential equation due to the non-local correction $\Psi_0$, which shall be solved by the similar argument as in \cite[Section 8]{17HMF}. Finally, using the Schauder fixed point theorem, we prove the existence of desired blow-up solution.

\medskip

%%%%%%%%%%%%%%%%%%%%%%%%%%%%%%%%%%%%%%%%%%%%%%%%%%%%%%%%%%%%%%%%%%%%%%%%%%%%%%%%%

\medskip

\section{Approximate solutions and error estimates}

\medskip

In the symmetry class \eqref{symmetryclass}, problem \eqref{eqn} becomes
$$u_t=u_{rr}+\frac{n-4}{r}u_r+\Delta_{z} u+u^3,$$
where $(r,z) \in \mathcal D$ with $\mathcal D$ defined in \eqref{def-domain}. Define the error operator
\begin{equation*}
\mathcal S (u):=-u_t+\Delta_{(r,z)} u+\frac{n-4}{r}u_r+u^3,
\end{equation*}
where $\Delta_{(r,z)}:=\partial_r^2+\Delta_z$ is the Laplacian in $\R^4$. Our first approximate solution is based on the Aubin--Talenti bubble (see \cite{Aubin76,Talenti76})
\begin{equation}\label{def-U}
U(y)=\frac{\alpha_{0}}{1+|y|^2}
\end{equation}
which solves the Yamabe problem
\begin{equation*}
\Delta_y U+U^3=0~\mbox{ in }~\R^4.
\end{equation*}
Here $\alpha_0=2\sqrt{2}.$ It is well-known that the linearized operator around the bubble
\begin{equation}\label{def-L0}
L_0(\phi):=\Delta \phi+3U^2\phi
\end{equation}
is non-degenerate in the sense that all bounded solutions to $L_0(\phi)=0$ are the linear combination of
\begin{equation}\label{kernels}
Z_i(y):=\partial_{y_i} U(y),~~i=1,2,3,4,~Z_5(y):=U(y)+\nabla U(y)\cdot y.
\end{equation}

\subsection{First approximate solution}

We define
\begin{equation*}
U_{\la(t),\xi(t)}(r,z)=\la^{-1}(t)U\left(\frac{(r,z)-\xi(t)}{\la(t)}\right),
\end{equation*}
where
$$\xi(t)=(\xi_r(t),\xi_z(t))$$
with
$$\xi_r(t)= \xi_{r,*}(t)+\xi_{r,1}(t),~~\xi_{r,1}(t)=o(\xi_{r,*}(t)),$$
$$\xi_z(t)=\xi_{z,*}(t)+\xi_{z,1}(t),~~|\xi_{z,1}(t)|=o(|\xi_{z,*}(t)|).$$
In the sequel, we denote
$$y=\frac{(r,z)-\xi(t)}{\la(t)}.$$
Now we choose the first approximate solution as
\begin{equation*}
U^*=\eta_*U_{\la(t),\xi(t)},
\end{equation*}
where the cut-off function $\eta_*$ is defined by
\begin{equation}\label{def-eta*}
\eta_*:=\eta\left(\frac{|(r,z)-\xi(t)|}{\delta \sqrt{T-t}}\right)
\end{equation}
with the positive constant $\delta$ fixed sufficiently small. Here the smooth cut-off function $\eta$ is defined by
\begin{equation}\label{def-cutoff}
\eta(s)=\begin{cases}
1,~&s<1,\\
0,~&s>2.\\
\end{cases}
\end{equation}
Then the first error of $U^*$ is
\begin{equation}\label{firsterror}
\begin{aligned}
\mathcal S(U^*)=&~-\eta_*\partial_t U_{\la,\xi}+U_{\la,\xi}(\Delta_{(r,z)}\eta_*-\partial_t \eta_*)+2\nabla \eta_*\cdot \nabla U_{\la,\xi}+(\eta_*^3 -\eta_* )U_{\la,\xi}^3 +\frac{n-4}{r} \partial_r(\eta_* U_{\la,\xi})\\
=&~\eta_*\left[\la^{-2}(t)\dot{\la}(t)Z_5(y)+\la^{-2}(t)\nabla U(y)\cdot \dot{\xi}(t)\right]+U_{\la,\xi}(\Delta_{(r,z)}\eta_*-\partial_t \eta_*)+2\nabla \eta_*\cdot \nabla U_{\la,\xi}\\
&~+(\eta_*^3 -\eta_* )U_{\la,\xi}^3 +\frac{n-4}r \partial_r(\eta_* U_{\la,\xi})\\
=&~ \eta_*\left[\la^{-2}(t)\dot{\la}(t)\left(-\frac{\alpha_0}{1+|y|^2}+\frac{2\alpha_0}{(1+|y|^2)^2}\right)+\la^{-2}(t)\nabla U(y)\cdot \dot{\xi}(t)\right]\\
&~+U_{\la,\xi}(\Delta_{(r,z)}\eta_*-\partial_t \eta_*)+2\nabla \eta_*\cdot \nabla U_{\la,\xi}+(\eta_*^3 -\eta_* )U_{\la,\xi}^3 +\frac{n-4}r \partial_r(\eta_* U_{\la,\xi}),
\end{aligned}
\end{equation}
where $\nabla:=(\partial_r, \nabla_z)$.

\subsection{Corrected approximate solution}

Observe that the slow decaying error in \eqref{firsterror} is
\begin{equation}\label{def-mE0}
\mathcal E_0=-\frac{\alpha_0\dot\la(t)}{\la^2(t)+\rho^2}\approx -\frac{\alpha_0\dot \la(t)}{\rho^2},
\end{equation}
where $\rho=|(r,z)-\xi(t)|.$
In order to reduce the size of the first error, we shall choose $\Psi^0$ to be an approximate solution of
$$\partial_t u=\Delta_{(r,z)} u+\mathcal E_0~\mbox{ in }~\R^4\times(0,T).$$
To achieve this, we consider
$$\partial_t \psi^0=\psi^0_{\rho\rho}+\frac{3}{\rho}\psi^0_{\rho}-\frac{\alpha_0\dot \la(t)}{\rho^2}.$$
We perform the change of variable $\phi_0=\rho \psi^0$ and get
\begin{equation}\label{correctHMF}
\partial_t \phi_0=(\phi_0)_{\rho\rho}+\frac{1}{\rho}(\phi_0)_{\rho}-\frac{1}{\rho^2}\phi_0-\frac{\alpha_0\dot \la(t)}{\rho^2}.
\end{equation}
From the same computations as in \cite[Section 4]{17HMF}, an explicit solution to problem \eqref{correctHMF} is given by the Duhamel's formula
$$\phi_0=-\alpha_0\rho \int_{-T}^t \dot\la(s) k(\rho,t-s) ds$$
with
\begin{equation}\label{def-k}
k(\rho,t):=\frac{1-e^{-\frac{\rho^2}{4t}}}{\rho^2}.
\end{equation}
Therefore, by $\psi^0=\rho^{-1}\phi_0$, we get
$$\psi^0=-\alpha_0 \int_{-T}^t \dot\la(s) k(\rho,t-s) ds.$$
We regularize the above $\psi^0$ and choose a good approximation $\Psi_0$ to be
\begin{equation}\begin{aligned}\label{def-Psi0}
\Psi_0(r,z,t)=&~\Psi_0(|(r,z)-\xi(t)|,t)=\Psi_0(\rho,t)\\=&~-\alpha_0\int_{-T}^t \dot\la(s) k(\zeta(\rho,t),t-s) ds,\end{aligned}
\end{equation}
where
\begin{equation}\label{def-zeta}
\zeta(\rho,t)=\sqrt{\rho^2+\la^2(t)}.
\end{equation}

Now we compute the new error produced by  $\Psi_0$ and get
\begin{equation*}
\begin{aligned}
&~\partial_t \Psi_0-\Delta_{(r,z)}\Psi_0-\mathcal E_0=\partial_t \Psi_0-\Delta_{(r-\xi_r(t),z-\xi_z(t))}\Psi_0-\mathcal E_0\\=&~\partial_t \Psi_0-\partial_\rho^2\Psi_0-\frac3\rho\,\partial_\rho\Psi_0-\mathcal E_0
\\=&~\alpha_0\left[\frac{(r-\xi_r)\dot{\xi}_r+(z-\xi_z)\dot{\xi}_z-\la(t)\dot{\la}(t)}{\zeta}\right]\int_{-T}^t \dot{\la}(s) k_{\zeta}(\zeta,t-s)ds\\
&~+\alpha_0\int_{-T}^t \dot{\la}(s)\left[-k_t(\zeta,t-s)+\frac{\rho^2}{\zeta^2}\,k_{\zeta\zeta}(\zeta,t-s)+\frac3\zeta\,k_\zeta(\zeta,t-s)+\frac{\la^2(t)}{\zeta^3}\,k_\zeta(\zeta,t-s)\right]ds.
\end{aligned}
\end{equation*}
Observe from \eqref{def-k} that $k(\zeta,t)$ satisfies
$$-k_t+k_{\zeta\zeta}+\frac3\zeta\,k_\zeta=0.$$
Therefore, we obtain
\begin{equation}\label{def-mR}
\begin{aligned}
\partial_t \Psi_0-\Delta_{(r,z)}\Psi_0-\mathcal E_0=&~\alpha_0\left[\frac{(r-\xi_r)\dot{\xi}_r+(z-\xi_z)\dot{\xi}_z-\la(t)\dot{\la}(t)}{\zeta}\right]\int_{-T}^t \dot{\la}(s) k_{\zeta}(\zeta,t-s)ds\\
&~+\alpha_0\int_{-T}^t \dot{\la}(s)\left[-\frac{\la^2(t)}{\zeta^2}\,k_{\zeta\zeta}(\zeta,t-s)+\frac{\la^2(t)}{\zeta^3}\,k_\zeta(\zeta,t-s)\right]ds\\
=&~\alpha_0\left[\frac{(r-\xi_r)\dot{\xi}_r+(z-\xi_z)\dot{\xi}_z-\la(t)\dot{\la}(t)}{\la(t)(1+|y|^2)^{1/2}}\right]\int_{-T}^t \dot{\la}(s) k_{\zeta}(\zeta,t-s)ds\\
&~+\frac{\alpha_0}{\la(t)(1+|y|^2)^{3/2}}\int_{-T}^t \dot{\la}(s)\left[-\zeta k_{\zeta\zeta}(\zeta,t-s)+k_\zeta(\zeta,t-s)\right]ds\\
:=&~\mathcal R[\la].
\end{aligned}
\end{equation}

Now we choose the corrected approximation as
$$u^*=U^*+\eta_*\Psi_0$$
and its error is given by
\begin{equation}\label{error22}
\begin{aligned}
\mathcal S(u^*)=&~\mathcal S(U^*)+\eta_*(\Delta_{(r,z)} \Psi_0-\partial_t \Psi_0)-\Psi_0\partial_t \eta_*+\Psi_0\Delta_{(r,z)}\eta_*+2\nabla \eta_*\cdot\nabla \Psi_0\\
&~+\frac{n-4}{r}\partial_r (\eta_*\Psi_0)+(U^*+\eta_*\Psi_0)^3-(U^*)^3\\
=&~\mathcal S(U^*)-\eta_*(\mathcal E_0+\mathcal R[\la])-\Psi_0\partial_t \eta_*+\Psi_0\Delta_{(r,z)}\eta_*+2\nabla \eta_*\cdot\nabla \Psi_0 +\frac{n-4}{r}\partial_r (\eta_*\Psi_0)\\
&~+(U^*+\eta_*\Psi_0)^3-(U^*)^3\\
=&~ \eta_*\left[\frac{2\alpha_0\la^{-2}(t)\dot{\la}(t)}{(1+|y|^2)^2}+\la^{-2}(t)\nabla U(y)\cdot \dot{\xi}(t)\right]-\eta_*\mathcal R[\la]+U_{\la,\xi}(\Delta_{(r,z)}\eta_*-\partial_t \eta_*)\\
&~+2\nabla \eta_*\cdot \nabla U_{\la,\xi}+(\eta_*^3 -\eta_* )U_{\la,\xi}^3 +\frac{n-4}r \partial_r(\eta_* U_{\la,\xi})-\Psi_0\partial_t \eta_*+\Psi_0\Delta_{(r,z)}\eta_*\\
&~+2\nabla \eta_*\cdot\nabla \Psi_0 +\frac{n-4}{r}\partial_r (\eta_*\Psi_0)+(U^*+\eta_*\Psi_0)^3-(U^*)^3\\
=&~\mathcal K[\la,\xi]+U_{\la,\xi}(\Delta_{(r,z)}\eta_*-\partial_t \eta_*)+2\nabla \eta_*\cdot \nabla U_{\la,\xi}+(\eta_*^3 -\eta_* )U_{\la,\xi}^3 +U_{\la,\xi}\frac{n-4}r \partial_r \eta_*\\
&~ -\Psi_0\partial_t \eta_*+\Psi_0\Delta_{(r,z)}\eta_*+2\nabla \eta_*\cdot\nabla \Psi_0 +\frac{n-4}{r}\partial_r (\eta_*\Psi_0)+(U^*+\eta_*\Psi_0)^3-(U^*)^3,\\
\end{aligned}
\end{equation}
where $\mathcal K[\la,\xi]$ is defined as
\begin{equation}\label{def-mK}
\mathcal K[\la,\xi]:=\eta_*\left[\frac{2\alpha_0 \la^{-2}(t)\dot{\la}(t)}{(1+|y|^2)^2}+\la^{-2}(t)\nabla U(y)\cdot \dot{\xi}(t)+\frac{n-4}{\la(t)y_1+\xi_r(t)}\la^{-2}(t)\partial_{y_1}U(y)-\mathcal R[\la]\right].
\end{equation}
We define
\begin{equation}\label{def-Sout}
\begin{aligned}
\mathcal S_{\rm out}[\la,\xi]:=&~U_{\la,\xi}(\Delta_{(r,z)}\eta_*-\partial_t \eta_*)+2\nabla \eta_*\cdot \nabla U_{\la,\xi}+(\eta_*^3 -\eta_* )U_{\la,\xi}^3 +U_{\la,\xi}\frac{n-4}r \partial_r \eta_*\\
&~+\frac{n-4}{r}\Psi_0\partial_r \eta_*+\Psi_0\Delta_{(r,z)}\eta_*+2\nabla \eta_*\cdot\nabla \Psi_0 -\Psi_0\partial_t \eta_*+\frac{n-4}{r}\eta_*(1-\eta_R)\partial_r\Psi_0\\
&~+\eta_*(1-\eta_R)\left(\la^{-2}\nabla U(y)\cdot \dot{\xi}+\frac{n-4}{\la y_1+\xi_r }\la^{-2} \partial_{y_1}U(y)\right).
\end{aligned}
\end{equation}
To further reduce the size of the error, we introduce the leading orders of $\la$ and $\xi$
$$\la_*=\frac{|\log T|(T-t)}{|\log(T-t)|^2},~~\xi_*=(\xi_{r,*},\xi_{z,*})=(\sqrt{2(n-4)(T-t)},z_0),$$
which will be derived in Section \ref{sec-laxi}. Here $z_0\in\R^3$. Notice that for $T\ll 1$, the error $\mathcal S_{\rm out}[\la,\xi]$ is supported in $\mathcal D'\times (0,T)$ with
$$\mathcal D'=\left\{(r,z)\in \mathcal D: \la_* R\leq |(r,z)-\xi(t)|\leq 2\delta \sqrt{T-t}\right\},$$
where $R=\la_*^{-\beta}$ with $\beta\in(0,1/2)$ to be determined later. Let $\Psi_1$ be the solution solving
\begin{equation}\label{def-Psi1}
\begin{cases}
\partial_t\Psi_1=\Delta_{(r,z)} \Psi_1 +\frac{n-4}{r}\partial_r\Psi_1+\mathcal S_{\rm out}[\la_*,\xi_*],~&\mbox{ in }~\mathcal D\times (0,T)\\
\Psi_1=0,~&\mbox{ on }~(\partial\mathcal D\backslash \{r=0\})\times (0,T)\\
\partial_r \Psi_1=0,~&\mbox{ on }~(\mathcal D\cap \{r=0\})\times (0,T)\\
\Psi_1(r,z,0)=0,~&\mbox{ in }~\mathcal D\\
\end{cases}
\end{equation}
where $\mathcal S_{\rm out}[\la_*,\xi_*]$ is defined by replacing $\la$, $\xi$ in $\mathcal S_{\rm out}[\la,\xi]$
by $\la_*$ and $\xi_*$, respectively. Note that the new error produced by $\Psi_1$ turns out to be smaller order and will not change the leading orders $\la_*$, $\xi_*$. We shall show this in Section \ref{sec-innerouter}.

In conclusion, the corrected approximation we finally choose is
$$u_c^*=U^*+\eta_*\Psi_0+\Psi_1.$$
In the sequel, we shall find a perturbation $\mathtt P$ such that $u=u_c^*+\mathtt P$ is the desired solution, namely,
$$\mathcal S(u_c^*+\mathtt P)=0.$$

%%%%%%%%%%%%%%%%%%%%%%%%%%%%%%%%%%%%%%%%%%%%%%%%%%%%%%%%

\medskip

\section{The inner--outer gluing scheme}\label{sec-inneroutergluingscheme}

\medskip

We look for solution of the form
$$u=U^*+\mathtt{w},$$
where $\mathtt{w}$ is a small perturbation consisting of inner and outer parts
$$\mathtt{w}=\varphi_{\rm in}+\varphi_{\rm out},~~\varphi_{\rm in}=\la^{-1}(t)\eta_R \phi(y,t),~~\varphi_{\rm out}=\eta_*\Psi_0(r,z,t)+\Psi_1(r,z,t)+\psi(r,z,t)+Z^*(x,t).$$
Here
\begin{equation}\label{def-etaR}
\eta_R=\eta_{R(t)}(r,z,t)=\eta\left(\frac{|(r,z)-\xi(t)|}{\la(t)R(t)}\right),
\end{equation}
the smooth cut-off function $\eta$ is defined by \eqref{def-cutoff},
and $Z^*$ satisfies
\begin{equation}\label{eqn-Z^*}
\begin{cases}
Z_t^*=\Delta_x Z^*,~&\mbox{ in }\Omega\times (0,T)\\
Z^*(\cdot,t)=0,~&\mbox{ on }\partial \Omega\times (0,T)\\
Z^*(\cdot,0)=Z_0^*,~&\mbox{ in }\Omega
\end{cases}
\end{equation}
in the original variables $x\in\R^n$. Throughout the paper, we choose $R(t)$ such that $\la(t)R(t)\ll \sqrt{T-t}$ for $T\ll 1$.
Denote
$$\mathcal D_{2R}=\left\{(r,z)\in \R^4:|(r,z)-(\xi_r,\xi_z)|\leq 2\la R\right\}$$
and $\Psi^*=\psi+Z^*$. Then $u$ is a solution to the original problem \eqref{eqn} if

\medskip

\begin{itemize}
\item $\phi$ solves the {\em inner problem}
\begin{equation}\label{inner}
\begin{aligned}
\la^2 \phi_t = \Delta_y \phi +3U^2(y)\phi+\mathcal H(\phi,\psi,\la,\xi)~~\mbox{ in }\mathcal D_{2R}\times (0,T)
\end{aligned}
\end{equation}
where
\begin{equation}\label{def-mH}
\begin{aligned}
\mathcal H(\phi,\psi,\la,\xi)(y,t):=&~3\la U^{2}(y)[\eta_*\Psi_0+\psi+Z^*](\la y+\xi,t)+\frac{(n-4)\la}{\la y_1+\xi_r}\phi_{y_1}\\
&~+\la\left[\dot\la (\nabla_y\phi\cdot y+\phi)+\nabla_y \phi\cdot\dot\xi\right]+\frac{(n-4)\la^3}{r}\eta_*\partial_r\Psi_0\\
&~+\la^3 \mathcal N(\mathtt w)+\la^3 \mathcal K[\la,\xi]+3\la U^{2}(y)\Psi_1
\end{aligned}
\end{equation}
with $\mathcal K[\la,\xi]$ defined in \eqref{def-mK}.

\medskip

\item $\psi$ solves the {\em outer problem}
\begin{equation}\label{outer}
\psi_t =\Delta_{(r,z)} \psi+\frac{n-4}{r}\partial_r \psi  + \mathcal G(\phi,\psi,\la,\xi)~~\mbox{ in }\mathcal D\times (0,T)
\end{equation}
with
\begin{equation}\label{def-mG}
\begin{aligned}
\mathcal G(\phi,\psi,\la,\xi):=&~3\la^{-2}(1-\eta_R)U^2(y)(\psi+Z^*+\eta_*\Psi_0+\Psi_1)\\
&~+\la^{-3}\left[(\Delta_y \eta_R) \phi+2\nabla_y\eta_R\cdot\nabla_y \phi-\la^2\phi\partial_t \eta_R\right]+\frac{(n-4)\la^{-1}}{r}\phi\partial_r \eta_R\\
&~+(1-\eta_R)\mathcal K_1[\la,\xi]+\mathcal S_{\rm out}[\la,\xi]-\mathcal S_{\rm out}[\la_*,\xi_*]+(1-\eta_R)\mathcal N(\mathtt w),
\end{aligned}
\end{equation}
where
\begin{equation}\label{def-mK1}
\mathcal K_1[\la,\xi]:=\eta_*\left[\frac{2\alpha_0 \la^{-2}(t)\dot{\la}(t)}{(1+|y|^2)^2}-\mathcal R[\la]\right],
\end{equation}
$\mathcal S_{\rm out}$ is defined in \eqref{def-Sout} and
\begin{equation}\label{def-mN}
\mathcal N(\mathtt w):=(U^*+\mathtt w)^3-(U^*)^3-3U_{\la,\xi}^2\mathtt w.
\end{equation}
\end{itemize}
\medskip
We now describe our strategy to solve the inner and outer problems. We shall first develop linear theories for the associated linear problems of \eqref{outer} and \eqref{inner}. Since the solution we want to construct concentrates on an $(n-4)$-dimensional sphere with shrinking size $\sqrt{T-t}$, suitable estimates for the outer solution $\psi$ are very delicate to find. To achieve this, we find solutions in the symmetry class \eqref{symmetryclass} by using the Duhamel's formula in $\R^n$. For the linear inner problem, we want to find inner solution with proper space-time decay.
 Since the inner--outer gluing relies on delicate analysis of the space-time decay of solutions, we shall further decompose the inner problem \eqref{inner} into three different spherical harmonic modes and construct solution in each mode. To get more refined estimates for the gluing to work, we carry out a new inner--outer gluing scheme for the linear inner problem, where certain orthogonality conditions are of course needed due to the existence of the nontrivial kernels (see \eqref{kernels}) of the linearized operator $L_0$ in \eqref{def-L0}. This will give us the reduced equations for the parameter functions $\la(t)$ and $\xi(t)$. The reduced equation for $\xi(t)$ will be easy to solve. However, the reduced equation for $\la(t)$ turns out to be an integro-differential equation due to the non-local correction $\Psi_0$ in \eqref{def-Psi0}, and it is more involved. Thanks to \cite{17HMF}, we solve $\la(t)$ by following a similar procedure since the integro-differential equation for $\la(t)$ is close in spirit to that of \cite{17HMF}. Finally, by using the Schauder fixed point theorem, we solve the inner--outer gluing system and prove the existence of the desired blow-up solution.

\medskip

The rest of the paper is organized as follows. In Section \ref{sec-laxi}, we derive the leading orders for the parameter functions $\la(t)$ and $\xi(t)$. In Section \ref{sec-linearouter}, we establish the estimates for the linear outer problem with different right hand sides which appear in $\mathcal G$ defined in \eqref{def-mG}. The proof is postponed to the Appendix. In Section \ref{sec-linearinner}, we develop the linear theory for the inner problem by spherical harmonic decomposition. In Section \ref{sec-innerouter}, the inner--outer gluing system is formulated, and we shall solve $(\phi,\psi,\la,\xi)$ from the full system by the linear theories developed in Section \ref{sec-linearouter}, Section \ref{sec-linearinner} and the Schauder fixed point theorem.

\bigskip
\noindent{{\bf Notation}}. \, Throughout the paper, we shall use the symbol  $``\,\lesssim\,"$ to denote $``\,\leq\, C\,"$ for a positive constant $C$ independent of $t$ and $T$. Here $C$ might be different from line to line.

%%%%%%%%%%%%%%%%%%%%%%%%%%%%%%%%%%%%%%%%%%%%%%%%%%%%%%%%%%%%

\medskip

\section{The choices of \texorpdfstring{$\lambda_*$}{lambda} and \texorpdfstring{ $\xi_*$}{xi}}\label{sec-laxi}

\medskip

In this section, we shall choose the leading orders $\la_*(t)$, $\xi_*(t)=(\xi_{r,*}(t),\xi_{z,*}(t))$ of the parameter functions $\la(t)$ and $\xi(t)$. In Section \ref{sec-linearinner}, a linear theory for inner problem concerning the solvability and estimates of the associated linear problem will be developed, where approximately the following orthogonality conditions
\begin{equation}\label{oc}
\int_{\R^4} \mathcal H(\phi,\psi,\la,\xi) Z_j (y) dy=0~~\mbox{ for all } j=1,\cdots,5,~  t\in (0,T)
\end{equation}
are needed to guarantee the existence of inner solution $\phi$ with desired space-time decay. Here $Z_j$ are the kernel functions (c.f. \eqref{kernels}) of the linearized operator $L_0$. Basically,  we will derive the scaling and translation parameters $\la(t)$ and $\xi(t)$ at main order from the orthogonality conditions \eqref{oc}.

Recall that
\begin{equation*}
\begin{aligned}
\mathcal H(\phi,\psi,\la,\xi)(y,t):=&~3\la U^{2}(y)[\eta_*\Psi_0+\psi+Z^*](\la y+\xi,t)+\frac{(n-4)\la}{\la y_1+\xi_r}\phi_{y_1}\\
&~+\la\left[\dot\la (\nabla_y\phi\cdot y+\phi)+\nabla_y \phi\cdot\dot\xi\right]+\frac{(n-4)\la^3}{r}\eta_*\partial_r\Psi_0\\
&~+\la^3 \mathcal N(\mathtt w)+\la^3 \mathcal K[\la,\xi]+3\la U^{2}(y)\Psi_1.\\
\end{aligned}
\end{equation*}
In this section, we shall single out the leading term $\mathcal H_*$ in $\mathcal H$ to derive $\la_*$ and $\xi_*$. We define
\begin{equation*}
\begin{aligned}
\mathcal H_*[\la,\xi,\Psi^*]:=&~3\la U^{2}(y)[\eta_*\Psi_0+\Psi^*](\la y+\xi,t)+\la^3 \mathcal K[\la,\xi]\\
=&~3\la U^{2}(y)[\eta_*\Psi_0+\Psi^*](\la y+\xi,t)+\frac{2\alpha_0\eta_* \la(t)\dot{\la}(t)}{(1+|y|^2)^2}\\
&~+\la(t)\eta_*\nabla U(y)\cdot \dot{\xi}(t)+\frac{(n-4)\la(t)\eta_*}{\la(t)y_1+\xi_r(t)}\partial_{y_1}U(y)\\
&~-\frac{\alpha_0\la^2(t)\eta_*}{(1+|y|^2)^{3/2}}\int_{-T}^t \dot{\la}(s)\left[-\zeta k_{\zeta\zeta}(\zeta,t-s)+k_\zeta(\zeta,t-s)\right]ds\\
&~-\alpha_0\eta_*\la^2(t)\left[\frac{(r-\xi_r)\dot{\xi}_r+(z-\xi_z)\dot{\xi}_z-\la(t)\dot{\la}(t)}{(1+|y|^2)^{1/2}}\right]\int_{-T}^t \dot{\la}(s) k_{\zeta}(\zeta,t-s)ds,\\
\end{aligned}
\end{equation*}
where $\Psi^*=\psi+Z^*$. The contribution of the rest terms $\mathcal H-\mathcal H_*$ in the orthogonality conditions turns out to be negligible compared to the leading term $\mathcal H_*$. We shall deal with this in Section \ref{sec-innerouter} when we finally solve the inner--outer gluing system.

Since $R(t)$ is chosen such that $\la(t)R(t)\ll \sqrt{T-t}$, we see that
$${\rm supp} (\eta_R) \subset {\rm supp} (\eta_*)$$
for $T\ll 1$. So in the inner region $\mathcal D_{2R}$, $\eta_*\equiv 1$. Then
$$\int_{\mathbb R^4} \mathcal H_*[\la,\xi,\Psi^*] Z_1(y)dy=0$$
implies that
\begin{equation*}
\dot{\xi}_r(t) +\frac{n-4}{\xi_r(t)}=o(1),
\end{equation*}
where $o(1)\to 0$ as $t\nearrow T.$ So the choice of $\xi_r(t)$ at main order is
\begin{equation*}
\xi_{r,*}(t)=\sqrt{2(n-4)(T-t)}.
\end{equation*}
Similarly, the orthogonality conditions
$$\int_{\R^4} \mathcal H_*[\la,\xi,\Psi^*] Z_j(y) dy=0,~~j=2,3,4$$
imply that
\begin{equation*}
\dot\xi_{z_j}(t)=o(1),~~j=2,3,4,
\end{equation*}
where $o(1)\to 0$ as $t\nearrow T.$ So at main order, the choice of $\xi_z(t)$ is
\begin{equation*}
\xi_{z,*}(t)=z_0
\end{equation*}
where $z_0$ is a given point in $\R^3$. For convenience, we take $z_0\equiv (0,0,0)$.

In order to get the reduced equation for $\la(t)$ from
$$\int_{\R^4} \mathcal H_*[\la,\xi,\Psi^*] Z_5(y)dy=0,$$
we first evaluate
\begin{equation*}
\begin{aligned}
\int_{\R^4} \mathcal R[\la] Z_5(y)dy=&~\frac{\alpha_0}{\la(t)}\int_{\R^4} \frac{Z_5(y)}{(1+|y|^2)^{3/2}}\left(\int_{-T}^t \dot{\la}(s)[k_{\zeta}(\zeta,t-s)-\zeta k_{\zeta\zeta}(\zeta,t-s)]ds\right) dy\\
&~-\alpha_0 \dot{\la}(t)\int_{\R^4} \frac{Z_5(y)}{(1+|y|^2)^{1/2}}\left(\int_{-T}^t \dot{\la}(s)k_{\zeta}(\zeta,t-s)ds\right)dy.
\end{aligned}
\end{equation*}
Let
\begin{equation}\label{def-Upsilon}
\Upsilon=\frac{\zeta^2}{t-s}=\frac{\la^2(t)(1+|y|^2)}{t-s},~~\tau=\frac{\la^2(t)}{t-s}
\end{equation}
and
\begin{equation*}
K(\Upsilon)=\frac{1-e^{-\frac{\Upsilon}{4}}}{\Upsilon}.
\end{equation*}
Then, recalling from \eqref{def-k}, we get
\begin{equation*}
\begin{aligned}
k_{\zeta}(\zeta,t-s)-\zeta k_{\zeta\zeta}(\zeta,t-s)=&~-\frac{8(1-e^{-\frac{\zeta^2}{4(t-s)}})}{\zeta^3}+\frac{2e^{-\frac{\zeta^2}{4(t-s)}}}{\zeta (t-s)}+\frac{\zeta e^{-\frac{\zeta^2}{4(t-s)}}}{4(t-s)^2}\\
=&~-4\left(\frac{\Upsilon}{t-s}\right)^{3/2}K_{\Upsilon\Upsilon}(\Upsilon)
\end{aligned}
\end{equation*}
and also
\begin{equation*}
k_{\zeta}(\zeta,t-s)=-\frac{2}{\zeta^3}+\frac{e^{-\frac{\zeta^2}{4(t-s)}}}{2\zeta(t-s)}+\frac{2e^{-\frac{\zeta^2}{4(t-s)}}}{\zeta^3}=\frac{2\sqrt{\Upsilon}}{(t-s)^{3/2}}K_{\Upsilon}(\Upsilon).
\end{equation*}
Therefore, we obtain
\begin{equation}\label{oc5R}
\begin{aligned}
\int_{\R^4} \mathcal R[\la] Z_5(y)dy=&~-\frac{4\alpha_0}{\la^2(t)}\int_{\R^4} \frac{Z_5(y)}{(1+|y|^2)^{2}}\left(\int_{-T}^t \frac{\dot{\la}(s)}{t-s}\Upsilon^2 K_{\Upsilon\Upsilon}(\Upsilon) ds\right) dy\\
&~-\frac{2\alpha_0 \dot{\la}(t)}{\la(t)}\int_{\R^4} \frac{Z_5(y)}{1+|y|^2}\left(\int_{-T}^t \frac{\dot{\la}(s)}{t-s} \Upsilon K_{\Upsilon}(\Upsilon) ds\right)dy.
\end{aligned}
\end{equation}

We expand $Z^*(\la y+\xi,t)$ and $\psi(\la y+\xi,t)$ at $q=(0,z_0)$
$$Z^*(\la y+\xi,t)=Z_0^*(q)+o(1),~~\psi(\la y+\xi,t)=\psi(q,0)+o(1).$$
On the other hand, from \eqref{def-Psi0} and \eqref{def-Upsilon}, we have
\begin{equation}\label{oc5Psi_0}
\begin{aligned}
\int_{\R^4} 3\la(t) U^2(y) Z_5(y) \Psi_0(\rho,t) dy=&-3\alpha_0\la(t)\int_{\R^4}  U^2(y) Z_5(y) \left(\int_{-T}^t \dot\la(s) k(\zeta(\rho,t),t-s) ds\right) dy\\
=&-3\alpha_0\la(t)\int_{\R^4}  U^2(y) Z_5(y) \left(\int_{-T}^t \frac{\dot\la(s)}{t-s} K(\Upsilon) ds\right) dy.
\end{aligned}
\end{equation}
Then, the orthogonality condition
$$\int_{\R^4} \mathcal H_*[\la,\xi,\Psi^*] Z_5(y)dy=0$$
gives
\begin{equation}\label{oc5''}
\int_{\R^4} \bigg(3 U^{2}(y)[\Psi_0(\rho,t)+\psi(q,0)+Z_0^*(q)]+\la^2(t)\mathcal K[\la,\xi]\bigg) Z_5(y) dy+o(1)=0.
\end{equation}
By \eqref{oc5''}, \eqref{def-mK}, \eqref{oc5R}, \eqref{oc5Psi_0} and some computations, we obtain
\begin{equation}\label{oc5all}
\begin{aligned}
&~4\alpha_0 \int_{\R^4} \frac{Z_5(y)}{(1+|y|^2)^{2}}\left(\int_{-T}^t \frac{\dot{\la}(s)}{t-s}\Upsilon^2 K_{\Upsilon\Upsilon} ds\right) dy-3\alpha_0 \int_{\R^4}  U^2(y) Z_5(y) \left(\int_{-T}^t \frac{\dot\la(s)}{t-s} K(\Upsilon) ds\right) dy\\
&+2\alpha_0 \dot{\la}(t)\la(t)\int_{\R^4} \frac{Z_5(y)}{1+|y|^2}\left(\int_{-T}^t \frac{\dot{\la}(s)}{t-s} \Upsilon K_{\Upsilon}(\Upsilon) ds\right)dy+2\alpha_0\dot{\la}(t)\int_{\R^4} \frac{Z_5(y)}{(1+|y|^2)^2}dy\\
&+3[Z_0^*(q)+\psi(q,0)]\int_{\R^4} U^2(y) Z_5(y) dy+o(1)=0.
\end{aligned}
\end{equation}
The scaling parameter $\la(t)$ should be decreasing to $0$ as $t\nearrow T$ so that a blow-up solution can be constructed. So we may  impose
$$\dot{\la}(t)=o(1)~\mbox{ as } t\nearrow T.$$
Then \eqref{oc5all} becomes
\begin{equation}\label{oc5'''}
\begin{aligned}
&4\alpha_0 \int_{\R^4} \frac{Z_5(y)}{(1+|y|^2)^{2}}\left(\int_{-T}^t \frac{\dot{\la}(s)}{t-s}\Upsilon^2 K_{\Upsilon\Upsilon} ds\right) dy-3\alpha_0 \int_{\R^4}  U^2(y) Z_5(y) \left(\int_{-T}^t \frac{\dot\la(s)}{t-s} K(\Upsilon) ds\right) dy\\
&=-3[Z_0^*(q)+\psi(q,0)]\int_{\R^4} U^2(y) Z_5(y) dy+o(1).
\end{aligned}
\end{equation}
We define
\begin{equation*}
\begin{aligned}
&4\alpha_0 \int_{\R^4} \frac{Z_5(y)}{(1+|y|^2)^{2}}\left(\int_{-T}^t \frac{\dot{\la}(s)}{t-s}\Upsilon^2 K_{\Upsilon\Upsilon} ds\right) dy-3\alpha_0 \int_{\R^4}  U^2(y) Z_5(y) \left(\int_{-T}^t \frac{\dot\la(s)}{t-s} K(\Upsilon) ds\right) dy\\
&:=\int_{-T}^t \frac{\dot{\la}(s)}{t-s}\Gamma\left(\frac{\la^2(t)}{t-s}\right) ds
\end{aligned}
\end{equation*}
with
\begin{equation}\label{def-Gamma}
\Gamma(\tau):=\alpha_0 |\mathbb{S}^3|\int_0^{\infty} \left(\frac{4Z_5(y)|y|^3}{(1+|y|^2)^2}\Upsilon^2 K_{\Upsilon\Upsilon}(\Upsilon)-3 U^2(y) Z_5(y)|y|^3K(\Upsilon) \right)\bigg|_{\Upsilon=\tau(1+|y|^2)} d|y|,
\end{equation}
where $|\mathbb{S}^3|$ is the area of the unit sphere $\mathbb{S}^3$.
Now we need to analyze the behavior of $\Gamma(\tau)$ as $\tau\ll 1$ and $\tau\gg 1$. By the definition of $U(y)$ and $Z_5(y)$ as in \eqref{def-U} and \eqref{kernels} respectively, we write $\Gamma(\tau)$ explicitly as
\begin{equation*}
\begin{aligned}
\Gamma(\tau)=&~\alpha_0 |\mathbb{S}^3|\int_0^{\infty} \left(\frac{4Z_5(y)|y|^3}{(1+|y|^2)^2}\Upsilon^2 K_{\Upsilon\Upsilon}(\Upsilon)-3 U^2(y) Z_5(y)|y|^3K(\Upsilon) \right)\bigg|_{\Upsilon=\tau(1+|y|^2)} d|y|\\
=&~\alpha_0^2 |\mathbb{S}^3|\int_0^{\infty} \left(\frac{(1-|y|^2)|y|^3}{(1+|y|^2)^4}\left[4\Upsilon^2 K_{\Upsilon\Upsilon}(\Upsilon)-3 \alpha_0^2 K(\Upsilon)\right] \right)\bigg|_{\Upsilon=\tau(1+|y|^2)} d|y|.\\
\end{aligned}
\end{equation*}
Since
\begin{equation*}
4\Upsilon^2 K_{\Upsilon\Upsilon}(\Upsilon)-3 \alpha_0^2 K(\Upsilon)=-\frac{\Upsilon e^{-\frac{\Upsilon}{4}}}{4}-2e^{-\frac{\Upsilon}{4}}+(8-3\alpha_0^2)\frac{1-e^{-\frac{\Upsilon}{4}}}{\Upsilon}
\end{equation*}
with $\alpha_0=2\sqrt{2}$,
we obtain
\begin{equation*}
\Gamma(\tau)=\begin{cases}
c_*+O(\tau),~&\mbox{ for } \tau<1\\
O\left(\frac{1}{\tau}\right),~&\mbox{ for } \tau>1\\
\end{cases}
\end{equation*}
where $c_*>0$ is a constant. Therefore, \eqref{oc5'''} is reduced to
\begin{equation}\label{oc5''''}
c_*\int_{-T}^{t-\la^2(t)}\frac{\dot{\la}(s)}{t-s}ds=-3c_0[Z_0^*(q)+\psi(q,0)]+o(1),
\end{equation}
where
$$c_0:=\int_{\R^4} U^2(y) Z_5(y) dy<0.$$
Since $\la(t)$ decreases to $0$ as $t\nearrow T,$
we impose
\begin{equation*}
a_*:=Z_0^*(q)+\psi(q,0)<0.
\end{equation*}
Now we claim that a good choice of $\la(t)$ at main order is
\begin{equation}\label{approxla}
\dot{\la}(t)=-\frac{c}{|\log(T-t)|^2},
\end{equation}
where $c>0$ is a constant to be determined later.
Indeed, by substituting, we get
\begin{equation*}
\begin{aligned}
\int_{-T}^{t-\la^2(t)}\frac{\dot{\la}(s)}{t-s}ds=&\int_{-T}^{t-(T-t)}\frac{\dot{\la}(s)}{t-s}ds+\int_{t-(T-t)}^{t-\la^2(t)} \frac{\dot{\la}(t)}{t-s}ds-\int_{t-(T-t)}^{t-\la^2(t)} \frac{\dot{\la}(t)-\dot{\la}(s)}{t-s}ds\\
=&\int_{-T}^{t-(T-t)}\frac{\dot{\la}(s)}{t-s}ds+\dot{\la}(t)(\log(T-t)-2\log\la(t))-\int_{t-(T-t)}^{t-\la^2(t)} \frac{\dot{\la}(t)-\dot{\la}(s)}{t-s}ds\\
\approx&\int_{-T}^{t}\frac{\dot{\la}(s)}{T-s}ds-\dot{\la}(t)\log(T-t):=\beta(t)
\end{aligned}
\end{equation*}
as $t\nearrow T$. By \eqref{approxla}, we then see that
$$\log(T-t)\frac{d\beta}{dt}(t)=\frac{d}{dt}\left(-\log^2(T-t)\dot{\la}(t)\right)=0,$$
which means $\beta(t)$ is a constant. Thus, equation \eqref{oc5''''} can be approximately solved for
$$\dot{\la}(t)=-\frac{c}{|\log(T-t)|^2}$$
with the constant $c$ chosen as
$$-c\int_{-T}^T \frac{ds}{(T-s)|\log(T-s)|^2}=\kappa_*,$$
where
$$\kappa_*:=-\frac{3c_0 a_*}{c_*}.$$
At main order, we obtain
$$\dot{\la}(t)=\kappa_*\dot{\la}_*(t)$$
with
$$\dot{\la}_*(t)=-\frac{|\log T|}{|\log(T-t)|^2}.$$
By imposing $\la_*(T)=0$, we finally get
\begin{equation*}
\la_*(t)=\frac{|\log T|(T-t)}{|\log(T-t)|^2}(1+o(1))~\mbox{ as }~t\nearrow T.
\end{equation*}

\medskip

%%%%%%%%%%%%%%%%%%%%%%%%%%%%%%%%%%%%%%%%%%%%%%%%%%%%%%%%%%%%

\medskip

\section{Linear theory for the outer problem}\label{sec-linearouter}

\medskip

In order to solve the outer problem \eqref{outer}, we need a linear theory for the associated linear problem. We consider
\begin{equation}\label{outer-linear}
\begin{cases}
\psi_t=\Delta_{(r,z)} \psi+\frac{n-4}{r}\partial_r \psi +\mathbf {f}_{\rm out},~&\mbox{ in }\mathcal D\times (0,T)\\
\psi=0,~&\mbox{ on }(\partial \mathcal D\backslash \{r=0\})\times (0,T)\\
\psi_r=0,~&\mbox{ on } (\mathcal D\cap\{r=0\})\times (0,T)\\
\psi(r,z,0)=0,~&\mbox{ in }\mathcal D\\
\end{cases}
\end{equation}
where the non-homogeneous term $\mathbf f_{{\rm out}}$ in \eqref{outer-linear} is assumed to be bounded with respect to the weights appearing in the outer problem \eqref{outer}. Define the weights
\begin{equation}\label{def-weights}
\left\{
\begin{aligned}
&\varrho_1:=\la_*^{\nu-3}(t)R^{-2-\alpha}(t)\chi_{\{|(r,z)-\xi(t)|\leq 2\la_* R\}}\\
&\varrho_2:=\frac{\la_*^{\nu_2}}{|(r,z)-\xi(t)|^2}\chi_{\{\la_* R \leq|(r,z)-\xi(t)|\leq 2\delta\sqrt{T-t}\}}\\
&\varrho_3:=1\\
\end{aligned}
\right.
\end{equation}
where we choose $R(t)=\la_*^{-\beta}(t)$ for $\beta\in(0,1/2)$.

Define the norms
\begin{equation}\label{def-norm**}
\|f\|_{**}:=\sup_{(r,z,t)\in \mathcal D\times(0,T)} \left(\sum\limits_{i=1}^3 \varrho_i(r,z,t)\right)^{-1}|f(r,z,t)|
\end{equation}

\begin{equation}\label{def-norm*}
\begin{aligned}
\|\psi\|_{*}:=&~\frac{\la_*^{2-\nu-\frac4n}(0)R^{2+\alpha-\frac8n}(0)}{|\log T|}\|\psi\|_{L^{\infty}(\Omega\times(0,T))}+\frac{\la_*^{\frac52-\nu-\frac2n}(0)R^{2+\alpha-\frac4n}(0)}{|\log T|}\|\nabla\psi\|_{L^{\infty}(\Omega\times(0,T))}\\
&~+\sup_{(r,z,t)\in \mathcal D\times(0,T)}\left[ \frac{\la_*^{2-\nu-\frac4n}(t)R^{2+\alpha-\frac8n}(t)}{|\log(T-t)|}\left|\psi(r,z,t)-\psi(r,z,T)\right|\right]\\
&~+\sup_{(r,z,t)\in \mathcal D\times(0,T)}\left[\frac{\la_*^{\frac52-\nu-\frac2n}(t)R^{2+\alpha-\frac4n}(t)}{|\log (T-t)|}\left|\nabla\psi(r,z,t)-\nabla\psi(r,z,T)\right|\right]\\
&~+\sup_{\mathcal D\times I_T} \frac{\la_*^{2-\nu+\gamma}(t_2)R^{2+\alpha}(t_2)}{(t_2-t_1)^{\gamma}} |\psi(r,z,t_2)-\psi(r,z,t_1)|,
\end{aligned}
\end{equation}
where $y=\frac{(r,z)-\xi(t)}{\la(t)}$, $\nu,\alpha\in(0,1)$, $\gamma\in(0,1)$,  and the last supremum is taken over
$$\mathcal D\times I_T=\left\{(r,z,t_1,t_2): ~(r,z)\in\mathcal D, ~0\leq t_1\leq t_2\leq T,~t_2-t_1\leq \frac{1}{10}(T-t_2)\right\}.$$
For problem \eqref{outer-linear}, we have the following proposition.
\begin{prop}\label{outer-apriori}
Let $\psi$ be the solution to problem \eqref{outer-linear} with $\|\mathbf {f}_{\rm out}\|_{**}<+\infty$. Then it holds that
\begin{equation}\label{est-outer-apriori}
\|\psi\|_*\lesssim \|\mathbf {f}_{\rm out}\|_{**}.
\end{equation}
\end{prop}
In order to establish Proposition \ref{outer-apriori}, we consider
\begin{equation}\label{outer-linear''}
\begin{cases}
\psi_t=\Delta_{\R^n} \psi +f,~&\mbox{ in }\Omega\times (0,T)\\
\psi=0,~&\mbox{ on }\partial \Omega\times (0,T)\\
\psi(x,0)=0,~&\mbox{ in }\Omega\\
\end{cases}
\end{equation}
which is defined in $\R^n$ in the symmetry class \eqref{symmetryclass}. For problem \eqref{outer-linear''}, we prove the following three lemmas concerning the a priori estimates with different right hand sides.

\begin{lemma}\label{lemma-rhs1}
Let $\psi$ solve problem \eqref{outer-linear''} with right hand side
$$|f(x,t)|\lesssim \la_*^{\nu-3}(t)R^{-2-\alpha}(t)\chi_{\{|(r,z)-\xi(t)|\leq 2\la_* R\}}.$$ If
$$\nu-3+\beta(2+\alpha)<0,~\nu-\beta(2-\alpha)>0,$$
then
\begin{equation}\label{outer-rhs1}
|\psi(x,t)|\lesssim\la_*^{\nu-2+\frac4n}(0) R^{-2-\alpha+\frac8n}(0)|\log T|,
\end{equation}
\begin{equation}\label{outerT-rhs1}
|\psi(x,t)-\psi(x,T)|\lesssim \la_*^{\nu-2+\frac4n}(t) R^{-2-\alpha+\frac8n}(t)|\log(T-t)|,
\end{equation}
\begin{equation}\label{outergradient-rhs1}
|\nabla \psi(x,t)|\lesssim\la_*^{\nu-\frac52+\frac2n}(0)R^{-2-\alpha+\frac4n}(0)|\log T|,
\end{equation}
\begin{equation}\label{outergradientT-rhs1}
|\nabla \psi(x,t)-\nabla \psi(x,T)|\lesssim\la_*^{\nu-\frac52+\frac2n}(t)R^{-2-\alpha+\frac4n}(t)|\log(T-t)|,
\end{equation}
and
\begin{equation}\label{outerholder-rhs1}
|\psi(x,t_2)-\psi(x,t_1)|\lesssim  \la_*^{\nu+\frac{\mu}{2}-3}(t_2)R^{-2-\alpha}(t_2) (t_2-t_1)^{1-\mu/2},
\end{equation}
where $0\leq t_1\leq t_2\leq T$ with $t_2-t_1\leq \frac{1}{10}(T-t_2)$ and $\mu\in(0,1)$.
\end{lemma}

\begin{lemma}\label{lemma-rhs2}
Let $\psi$ solve problem \eqref{outer-linear''} with right hand side
$$|f(x,t)|\lesssim \frac{\la_*^{\nu_2}}{|(r,z)-\xi(t)|^2}\chi_{\{\la_* R\leq|(r,z)-\xi(t)|\leq 2\delta\sqrt{T-t}\}},$$
where $\nu_2\in(0,1)$, $\delta>0$.
Then
\begin{equation}\label{outer-rhs2}
|\psi(x,t)|\lesssim  \la_*^{\nu_2-1}(0)R^{-2}(0)|\log T|,
\end{equation}
\begin{equation}\label{outerT-rhs2}
|\psi(x,t)-\psi(x,T)|\lesssim  \la_*^{\nu_2-1}(t)R^{-2}(t)|\log (T-t)|,
\end{equation}
\begin{equation}\label{outergradient-rhs2}
|\nabla \psi(x,t)|\lesssim \la_*^{\nu_2-2}(t)R^{-2}(t)\sqrt{t},
\end{equation}
\begin{equation}\label{outergradientT-rhs2}
|\nabla \psi(x,t)-\nabla \psi(x,T)|\lesssim \la_*^{\nu_2-\frac32}(t)R^{-2}(t)|\log (T-t)|,
\end{equation}
and
\begin{equation}\label{outerholder-rhs2}
|\psi(x,t_2)-\psi(x,t_1)|\lesssim \la_*^{\nu_2-1-\gamma}(t_2) R^{-2}(t_2) (t_2-t_1)^{\gamma},
\end{equation}
where $0\leq t_1\leq t_2\leq T$ with $t_2-t_1\leq \frac{1}{10}(T-t_2)$ and $\gamma\in(0,1)$.
\end{lemma}

\begin{lemma}\label{lemma-rhs3}
Let $\psi$ solve problem \eqref{outer-linear''} with right hand side
$$|f(x,t)|\lesssim 1.$$
Then
\begin{equation}\label{outer-rhs3}
|\psi(x,t)|\lesssim t,
\end{equation}
\begin{equation}\label{outerT-rhs3}
|\psi(x,t)|\lesssim (T-t)|\log(T-t)|,
\end{equation}
\begin{equation}\label{outergradient-rhs3}
|\nabla\psi(x,t)|\lesssim T^{1/2},
\end{equation}
\begin{equation}\label{outergradientT-rhs3}
|\nabla\psi(x,t)-\nabla\psi(x,T)|\lesssim (T-t)^{1/2},
\end{equation}
and
\begin{equation}\label{outerholder-rhs3}
|\psi(x,t_2)-\psi(x,t_1)|\lesssim (t_2-t_1)|\log (t_2-t_1)|,
\end{equation}
where $0\leq t_1\leq t_2\leq T$ with $t_2-t_1\leq \frac{1}{10}(T-t_2)$.
\end{lemma}

\medskip

\begin{proof}[Proof of Proposition \ref{outer-apriori}]
We denote $\psi[\mathbf {f}_{\rm out}]$ by the solution to problem \eqref{outer-linear''} with the right hand side $\mathbf {f}_{\rm out}$ satisfying $\|\mathbf {f}_{\rm out}\|_{**}<+\infty$. Decompose
$$\mathbf {f}_{\rm out}=\sum\limits_{i=1}^3 f_i~\mbox{ with }~|f_i|\lesssim \varrho_i\|f_i\|_{**}.$$
Let $1-\frac{\mu}{2}=\gamma$ in Lemma \ref{lemma-rhs1}. Then by the linearity, \eqref{est-outer-apriori} follows from Lemma \ref{lemma-rhs1}, Lemma \ref{lemma-rhs2} and Lemma \ref{lemma-rhs3}.
\end{proof}

The proofs of Lemma \ref{lemma-rhs1}, Lemma \ref{lemma-rhs2} and Lemma \ref{lemma-rhs3} are postponed to the Appendix.

\medskip

%%%%%%%%%%%%%%%%%%%%%%%%%%%%%%%%%%%%%%%%%%%%%%%%%%%%%%%%

\section{Linear theory for the inner problem}\label{sec-linearinner}

\medskip

In this section, we develop a linear theory concerning the estimates for the associated linear problem of the inner problem under certain topology.

In order to solve the inner problem \eqref{inner}, we consider the associated linear problem
\begin{equation}\label{eqn-li}
\la^{2} \phi_t=\Delta_y \phi+3U^{2}(y) \phi+h(y,t)~\mbox{ in }~\mathcal D_{2R}\times (0,T).
\end{equation}
Recall that the linearized operator $L_0=\Delta +3U^{2}$
has only one positive eigenvalue $\mu_0$ such that
$$L_0(Z_0)=\mu_0 Z_0,~Z_0\in L^{\infty}(\R^4),$$
where the corresponding eigenfunction $Z_0$ is radially symmetric with the asymptotic behavior
\begin{equation}\label{decay-Z_0}
Z_0(y)\sim |y|^{-3/2}e^{-\sqrt{\mu_0}|y|}~\mbox{ as }~|y|\to+\infty.
\end{equation}
Multiplying equation \eqref{eqn-li} by $Z_0$ and integrating over $\R^4$, we obtain that
$$\la^2(t)\dot p(t)-\mu_0 p(t)=q(t),$$
where
$$p(t)=\int_{\R^4} \phi(y,t)Z_0(y) dy~\mbox{ and }~q(t)=\int_{\R^4} h(y,t)Z_0(y) dy.$$
Then we have
$$p(t)=e^{\int_0^t \mu_0\la^{-2}(r) dr}\left(p(0)+\int_0^t q(\eta)\la^{-2}(\eta)e^{-\int_0^{\eta} \mu_0\la^{-2}(r)dr}d\eta\right).$$
In order to get a decaying solution, the initial condition
$$p(0)=-\int_0^T q(\eta)\la^{-2}(\eta)e^{-\int_0^{\eta} \mu_0\la^{-2}(r)dr}d\eta$$
is required. The above formal argument suggests that a linear constraint should be imposed on the initial value $\phi(y,0)$. Therefore, we consider the associated linear Cauchy problem of the inner problem \eqref{inner}
\begin{equation}\label{linear-inner'''''}
\begin{cases}
  \la^2\phi_{t}=\Delta_y\phi + 3U^{2}(y)\phi + h(y,t), & \mbox{ in } \mathcal D_{2R}\times (0,T) \\
  \phi(y,0)=e_0 Z_0(y), & \mbox{ in } \mathcal D_{2R(0)}
\end{cases}
\end{equation}
where $R=R(t)=\la_*^{-\beta}(t)$ for $\beta\in (0,1/2)$.
On the other hand, the parabolic operator $-\la^2\partial_t+L_0$ is certainly not invertible since all the time independent elements in the 5 dimensional kernel of $L_0$ (see \eqref{kernels}) also belong to the kernel of $-\la^2\partial_t+L_0$. In order to construct solution to \eqref{linear-inner'''''} with suitable space-time decay, some orthogonality conditions are expected to hold. So we consider the projected problem
\begin{equation}\label{linear-inner}
\begin{cases}
  \la^2\phi_{t}=\Delta_y\phi + 3U^{2}(y)\phi + h(y,t)+\tilde c^0(t)Z_5(y)+\sum\limits_{\ell=1}^4 c^{\ell}(t) Z_{\ell}(y), & \mbox{ in } \mathcal D_{2R}\times (0,T), \\
  \phi(y,0)=e_0 Z_0(y), & \mbox{ in } \mathcal D_{2R(0)}.
\end{cases}
\end{equation}

Our aim is to find suitable solution to problem \eqref{linear-inner} with space-time decay of the following type
\begin{equation}\label{def-normnua}
\|\phi\|_{\nu,a}:=\sup_{\substack{y\in \mathcal D_{2R} \\ t\in(0,T)}} \la_*^{-\nu}(t)(1+|y|^a)\left[|\phi(y,t)|+(1+|y|)|\nabla \phi(y,t)|\right],
\end{equation}
and the right hand side of problem \eqref{linear-inner}
\begin{equation}
\|h\|_{\nu,a}:=\sup_{\substack{y\in \mathcal D_{2R} \\ t\in(0,T)}} \la_*^{-\nu}(t)(1+|y|^a)|h(y,t)|.
\end{equation}
The construction of such solution is carried out by decomposing the equation into different spherical harmonic modes. Let an orthonormal basis $\{\Theta_i\}_{i=0}^{\infty}$ made up of spherical harmonics in $L^2(\mathbb{S}^{3})$, i.e.
$$\Delta_{\mathbb{S}^{3}}\Theta_i+\tilde\mu_i \Theta_i=0 ~\mbox{ in } ~\mathbb{S}^{3}$$
with
$$0=\tilde\mu_0<\tilde\mu_1=\cdots=\tilde\mu_4=3<\tilde\mu_5\leq\cdots.$$
More precisely, $\Theta_0(y)=c_0,~\Theta_i(y)=c_1y_j,~j=1,\cdots,4$ for two constants $c_0$, $c_1$ and $\tilde\mu_i$ takes the general form $i(2+i)$ with multiplicity $\frac{(3+i)!}{6i!}$ for $i\geq 0$.

For $h\in L^2(\mathcal D_{2R})$, we decompose it into
\begin{equation*}
h(y,t)=\sum\limits_{j=0}^{\infty} h_j(r,t)\Theta_j(y/r),\quad r=|y|,\quad h_j(r,t)=\int_{\mathbb{S}^{3}} h(r\theta,t)\Theta_j(\theta)d\theta
\end{equation*}
and write $h=h^0+h^1+h^{\perp}$ with
$$h^0=h_0(r,t),\quad h^1=\sum\limits_{j=1}^4 h_j(r,t)\Theta_j,\quad h^{\perp}=\sum\limits_{j=5}^{\infty} h_j(r,t)\Theta_j.$$
Also, we decompose $\phi=\phi^0+\phi^1+\phi^{\perp}$ in a similar form. Then finding a solution to problem \eqref{linear-inner} is equivalent to finding the pairs $(\phi^0,h^0),~(\phi^1,h^1),~(\phi^{\perp},h^{\perp})$ in each mode.

The main result of this section is stated as follows.
\begin{prop}\label{lineartheory}
Let constants $a,\nu,\nu_1,\sigma\in(0,1)$ and $a_1\in(1,2)$. For $T>0$ sufficiently small and any $h(y,t)$ satisfying $\|h^0\|_{\nu,2+a}<+\infty$, $\|h^1\|_{\nu_1,2+a_1}<+\infty$, $\|h^{\perp}\|_{\nu,2+a}<+\infty$, there exists a solution $(\phi,\tilde c^0,c^{\ell},e_0)$ solving \eqref{linear-inner} and   $(\phi,\tilde c^0,c^{\ell},e_0)=(\phi[h],\tilde c^0[h^0],c^{\ell}[h^1],e_0[h])$ defines a linear operator of $h$ that satisfies the estimates
\begin{itemize}
\item For $|y|\leq 2R^{\sigma}$,
\begin{equation}\label{inner-apriori1}
\begin{aligned}
|\phi(y,t)|+(1+|y|)|\nabla \phi(y,t)|\lesssim&~ \Bigg[\frac{\la_*^{\nu}(t)R^{\sigma(4-a)}\log R}{1+|y|^4}\|h^0\|_{\nu,2+a}+\frac{\la_*^{\nu_1}(t)}{1+|y|^{a_1}}\|h^1\|_{\nu_1,2+a_1}\\
&~+\frac{\la_*^{\nu}(t)}{1+|y|^a}\|h^{\perp}\|_{\nu,2+a}\Bigg]
\end{aligned}
\end{equation}

\item For $2R^{\sigma}\leq |y| \leq 2R$,
\begin{equation}\label{inner-apriori2}
\begin{aligned}
|\phi(y,t)|+(1+|y|)|\nabla \phi(y,t)|\lesssim&~ \Bigg[\frac{\la_*^{\nu}(t)\log R}{1+|y|^a}\|h^0\|_{\nu,2+a}+\frac{\la_*^{\nu_1}(t)}{1+|y|^{a_1}}\|h^1\|_{\nu_1,2+a_1}\\
&~+\frac{\la_*^{\nu}(t)}{1+|y|^a}\|h^{\perp}\|_{\nu,2+a}\Bigg]
\end{aligned}
\end{equation}
\end{itemize}
\begin{equation}\label{def-cell}
\tilde c^0(t)=-\frac{\int_{\mathcal D_{2R_*}}h^0 Z_5 dy}{\int_{\mathcal D_{2R_*}}|Z_5|^2 dy}-\mathcal O[h^0],~~c^{\ell}(t)=-\frac{\int_{\mathcal D_{2R}}h^1 Z_{\ell} dy}{\int_{\mathcal D_{2R}}|Z_{\ell}|^2 dy}~\mbox{ for }~\ell=1,\cdots,4,
\end{equation}
where $R_*=R^{\sigma}$ and $\mathcal O[h^0]$ is a linear operator of $h^0$ satisfying
$$|\mathcal O[h^0]|\lesssim \la_*^{\nu} R_*^{a'-a}\log R\|h^0\|_{\nu,2+a}$$
for $a'\in(0,a)$. Moreover,
\begin{equation*}
|e_0[h]|\lesssim \la_*^{\nu}\left(\|h^0\|_{\nu,2+a}+\|h^1\|_{\nu_1,2+a_1}+\|h^{\perp}\|_{\nu,2+a}\right).
\end{equation*}
\end{prop}

We devote the rest of this section to proving Proposition \ref{lineartheory}. Our strategy is to construct $\phi=\phi^0+\phi^1+\phi^{\perp}$ mode by mode.

\medskip

\noindent {\bf 1. Construction at mode $0$.}

\medskip

We construct $\phi^0$ solving the linearized problem at mode $0$
\begin{equation}\label{eqnmode0'''}
\begin{cases}
  \la^2\phi^0_{t}=\Delta_y\phi^0 + 3U^{2}(y)\phi^0 + h^0(y,t)+\tilde c^0(t)Z_5(y), & \mbox{ in } \mathcal D_{2R}\times (0,T), \\
  \phi^0(y,0)=e_0 Z_0(y), & \mbox{ in } \mathcal D_{2R(0)}.
\end{cases}
\end{equation}

The main result for mode $0$ is the following

\begin{prop}\label{propmode0}
Let $\nu,a,\sigma\in(0,1)$. Suppose $\|h^0\|_{\nu,2+a}<+\infty$. Then there exists a solution $(\phi^0,\tilde c^0,e_0)$ to problem \eqref{eqnmode0'''}, which depends on $h^0$ linearly such that
\begin{equation*}
|\phi^0(y,t)|+(1+|y|)|\nabla \phi^0(y,t)|\lesssim \la_*^{\nu}\log R\|h^0\|_{\nu,2+a}\begin{cases}
\frac{R^{\sigma(4-a)}}{1+|y|^4},~&\mbox{ for }~|y|\leq 2R^{\sigma},\\
\frac{1}{1+|y|^a},~&\mbox{ for }~2R^{\sigma}\leq|y|\leq 2R,\\
\end{cases}
\end{equation*}
\begin{equation*}
\tilde c^0[h^0](t)=-\frac{\int_{\mathcal D_{2R_*}}h^0 Z_5 dy}{\int_{\mathcal D_{2R_*}}|Z_5|^2 dy}-\mathcal O[h^0],
\end{equation*}
where $\mathcal O[h^0]$ is a linear operator of $h^0$ satisfying
$$|\mathcal O[h^0]|\lesssim \la_*^{\nu} R_*^{a'-a}\log R\|h^0\|_{\nu,2+a}$$
for $a'\in(0,a).$ Futhermore, it holds that
$$|e_0[h^0]|\lesssim \la_*^{\nu}\|h^0\|_{\nu,2+a}.$$
\end{prop}

\begin{remark}
If we define
\begin{equation}\label{def-normm0}
\|\phi^0\|_{0,\sigma,\nu,a}:=\sup_{(y,t)\in\mathcal D_{2R}\times(0,T)} \frac{1+|y|^4}{\la_*^{\nu}(t)R^{\sigma(4-a)}(t)\log R}\left[|\phi^0(y,t)|+(1+|y|)|\nabla \phi^0(y,t)|\right],
\end{equation}
then Proposition \ref{propmode0} implies that
\begin{equation*}
\|\phi^0\|_{0,\sigma,\nu,a}\lesssim \|h^0\|_{\nu,2+a}.
\end{equation*}
\end{remark}

\medskip

The strategy to prove Proposition \ref{propmode0} is a new inner--outer gluing scheme. We shall decompose $\phi^0$ into inner and outer profiles to get more refined estimates. Before we prove Proposition \ref{propmode0}, we first state a result for the following problem
\begin{equation}\label{modelpropmode0''}
\begin{cases}
\la^2\phi_t=\Delta_y \phi+3U^2(y)\phi+h(y,t)+\tilde c^0(t)Z_5-c(t)Z_0,~&\mbox{ in } \mathcal D_{2R}\times (0,T),\\
\phi(y,0)=0,~&\mbox{ in } \mathcal D_{2R(0)}.
\end{cases}
\end{equation}
\begin{prop}\label{modelpropmode0}
Let $\nu,a\in(0,1)$. Then for sufficiently large $R$ and any $h$ satisfying $\|h\|_{\nu,2+a}<+\infty$, there exists a solution $(\phi,\tilde c^0,c)$ to \eqref{modelpropmode0''} which is linear in $h$ such that
\begin{equation}\label{mode0est444}
|\phi(y,t)|+(1+|y|)|\nabla \phi(y,t)|\lesssim \la_*^{\nu}\frac{R^{4-a}\log R}{1+|y|^4}\|h\|_{\nu,2+a},
\end{equation}
\begin{equation*}
\tilde c^0(t)=-\frac{\int_{\mathcal D_{2R}}h Z_5 dy}{\int_{\mathcal D_{2R}}|Z_5|^2 dy},
\end{equation*}
and
\begin{equation}\label{est-ccc}
\left|c(t)-\int_{\mathcal D_{2R}} h Z_0\right| \lesssim \la_*^{\nu}(t)\left[\left\|h-Z_0\int_{\mathcal D_{2R}}h Z_0\right\|_{\nu,2+a}+e^{-\vartheta R}\|h\|_{\nu,2+a}\right].
\end{equation}
\end{prop}
The proof of Proposition \ref{modelpropmode0} can be carried out similar to that of \cite[Section 7]{Green16JEMS} (see also \cite[Section 5.2]{ni4d}). Proposition \ref{modelpropmode0} will be needed to describe the inner profile of $\phi^0$ when the inner--outer gluing scheme is carried out.

\begin{proof}[Proof of Proposition \ref{propmode0}]
Suppose
$$\phi^0(y,t)=\phi^0_1+e(t)Z_0(y)$$
with $\phi_1^0$ solving problem
\begin{equation}\label{linear-inner''}
\begin{cases}
\la^2\phi_t=\Delta_y \phi+3U^2(y)\phi+h^0(y,t)+\tilde c^0(t)Z_5-c(t)Z_0,~&\mbox{ in } \mathcal D_{2R}\times (0,T),\\
\phi(y,0)=0,~&\mbox{ in } \mathcal D_{2R(0)}.
\end{cases}
\end{equation}
For $e\in C^1((0,T))$, we get
$$\la^2\phi^0_t=\Delta_y \phi^0+3U^2\phi^0+h^0(y,t)+\tilde c^0(t)Z_5+[\la^2\dot e(t)-\mu_0 e(t)-c(t)]Z_0(y),$$
from which we see that a natural choice of bounded solution $e(t)$ to
$$\la^2\dot e(t)-\mu_0 e(t)-c(t)=0,~t\in(0,T)$$
is
\begin{equation}\label{choiceeee}
e(t)=-\int_t^T \exp \left(-\int_t^{\eta}\frac{\mu_0}{\la^2(s)}ds\right)\frac{c(\eta)}{\la^2(\eta)}d\eta.
\end{equation}
Therefore, $\phi^0$ solves problem \eqref{eqnmode0'''} with the initial condition $\phi^0(y,0)=e(0)Z_0(y)$. It is clear from \eqref{choiceeee} and \eqref{est-ccc} that
$$|e_0|\lesssim \la_*^{\nu}\|h^0\|_{\nu,2+a}.$$
So, to solve \eqref{eqnmode0'''}, we only need to consider \eqref{linear-inner''}.

Now we carry out an inner--outer gluing scheme for the mode $0$. Consider
\begin{equation}\label{iomode0}
\begin{cases}
\la^2\phi_t=\Delta_y \phi+3U^2(y)\phi+h^0(y,t)+\tilde c^0(t)Z_5-c(t)Z_0,~&\mbox{ in } \mathcal D_{2R}\times (0,T),\\
\phi(y,0)=0,~&\mbox{ in } \mathcal D_{2R(0)},\\
\phi=0,~&\mbox{ on } \partial \mathcal D_{2R}\times (0,T).\\
\end{cases}
\end{equation}
We shall construct $\phi^0$ solving \eqref{iomode0} of the form
$$\phi^0=\phi^0_{out}+\eta_{R_*}\phi^0_{in},$$
where
$$\eta_{R_*}:=\eta\left(\frac{|y|}{R_*}\right)$$
with $\eta$ defined in \eqref{def-cutoff} and
$$R_*=R^{\sigma}~\mbox{ for }~\sigma\in(0,1).$$
A solution $\phi^0$ to \eqref{iomode0} is found if $\phi^0_{out}$ and $\phi^0_{in}$ solve the system
\begin{equation}\label{outer-mode0}
\left\{
\begin{aligned}
&\la^2\partial_t \phi^0_{out}=\Delta_y \phi^0_{out}+3(1-\eta_{R_*})U^2(y)\phi^0_{out}+\mathtt C[\phi^0_{in}]+(1-\eta_{R_*})h^0,~\mbox{ in } \mathcal D_{2R}\times (0,T)\\
&\phi_{out}^0(y,0)=0,~\mbox{ in } \mathcal D_{2R(0)}\\
&\phi_{out}^0=0,~\mbox{ on } \partial \mathcal D_{2R}\times (0,T)\\
\end{aligned}
\right.
\end{equation}
\begin{equation}\label{inner-mode0}
\left\{
\begin{aligned}
&\la^2\partial_t \phi^0_{in}=\Delta_y \phi_{in}^0+3U^2(y)\phi^0_{in}+3U^2(y)\phi^0_{out}+h^0+\tilde c^0Z_5-cZ_0,~\mbox{ in } \mathcal D_{2R_*}\times (0,T)\\
&\phi_{in}^0(y,0)=0,~\mbox{ in } \mathcal D_{2R_*(0)}\\
\end{aligned}
\right.
\end{equation}
where
$$\mathtt C[\phi^0_{in}]:=\phi_{in}^0(\Delta \eta_{R_*}-\la^2 \partial_t \eta_{R_*}) +2\nabla \eta_{R_*}\cdot \nabla \phi_{in}^0.$$
We first consider the outer part \eqref{outer-mode0}. For the model problem
\begin{equation*}
\begin{cases}
\la^2 \partial_t \psi=\Delta \psi+h^0~&\mbox{ in } \mathcal D_{2R}\times (0,T)\\
\psi(y,0)=0,~&\mbox{ in } \mathcal D_{2R(0)}\\
\psi=0,~&\mbox{ on } \partial \mathcal D_{2R}\times (0,T)\\
\end{cases}
\end{equation*}
we have
\begin{equation}\label{mode0est111}
\|\psi\|_{\nu,a}\lesssim\|h^0\|_{\nu,2+a}
\end{equation}
by $\beta\in(0,1/2)$ and the parabolic comparison. Then we apply the above estimate to the following problem
\begin{equation}\label{mode0eqn111}
\begin{cases}
\la^2 \partial_t \psi=\Delta \psi+3(1-\eta_{R_*})U^2\psi+h^0~&\mbox{ in } \mathcal D_{2R}\times (0,T)\\
\psi(y,0)=0,~&\mbox{ in } \mathcal D_{2R(0)}\\
\psi=0,~&\mbox{ on } \partial \mathcal D_{2R}\times (0,T)\\
\end{cases}
\end{equation}
and we claim that the solution $\psi$ to \eqref{mode0eqn111} satisfies
$$\|\psi\|_{\nu,a}\lesssim\|h^0\|_{\nu,2+a}.$$
Indeed, by \eqref{mode0est111}, we only need to estimate
\begin{equation*}
\begin{aligned}
3(1-\eta_{R_*})U^2\psi \lesssim&~ (1-\eta_{R_*})\frac{\la_*^{\nu}}{1+|y|^{4+a}}\|\psi\|_{\nu,a}\\
\lesssim&~R_*^{-2}\frac{\la_*^{\nu}}{1+|y|^{2+a}}\|\psi\|_{\nu,a}
\end{aligned}
\end{equation*}
and we conclude that
\begin{equation}\label{mode0est222}
\|3(1-\eta_{R_*})U^2\psi\|_{\nu,2+a}\lesssim R_*^{-2}\|\psi\|_{\nu,a}.
\end{equation}
So from \eqref{mode0est111} and \eqref{mode0est222}, we obtain
\begin{equation*}
\begin{aligned}
\|\psi\|_{\nu,a}\lesssim&~\|3(1-\eta_{R_*})U^2\psi+h^0\|_{\nu,2+a}\\
\lesssim&~R_*^{-2}\|\psi\|_{\nu,a}+\|h^0\|_{\nu,2+a}
\end{aligned}
\end{equation*}
and for $R_*$ sufficiently large, it follows that
\begin{equation}\label{mode0est333}
\|\psi\|_{\nu,a}\lesssim \|h^0\|_{\nu,2+a}
\end{equation}
as desired.

We look for a solution $\phi_{out}^0$ to problem \eqref{outer-mode0}. By \eqref{mode0est333}, we get
\begin{equation}\label{mode0est666}
\|\phi^0_{out}\|_{\nu,a'}\lesssim \|\mathtt C[\phi_{in}^0]\|_{\nu,2+a'}+\|(1-\eta_{R_*})h^0\|_{\nu,2+a'},
\end{equation}
where $a'\in(0,a)$. Here $\phi^0_{out}$ defines a linear operator of $\phi^0_{in}$ and $h^0$. We write it as $\phi^0_{out}[\phi^0_{in},h^0]$.
 Now we need to find $\phi^0_{in}$ solving the inner part
\begin{equation}\label{inner-mode0''}
\left\{
\begin{aligned}
&\la^2\partial_t \phi^0_{in}=\Delta_y \phi_{in}^0+3U^2(y)\phi^0_{in}+3U^2(y)\phi^0_{out}[\phi^0_{in},h^0]+h^0+\tilde c^0Z_5-cZ_0,~\mbox{ in } \mathcal D_{2R_*}\times (0,T),\\
&\phi_{in}^0(y,0)=0,~\mbox{ in } \mathcal D_{2R_*(0)}.\\
\end{aligned}
\right.
\end{equation}
To solve the inner part \eqref{inner-mode0''}, we consider the fixed point problem
$$\phi^0_{in}=\mathcal T\left[3U^2(y)\phi^0_{out}[\phi^0_{in},h^0]+h^0\right]$$
in the function space equipped with the norm
\begin{equation*}
\|\phi^0_{in}\|_{0,*}:=\sup_{(y,t)\in \mathcal D_{2R_*}\times(0,T)} \la_*^{-\nu}(t) R_*^{a-4}(\log R)^{-1}(1+|y|^4)\left[|\phi^0_{in}|+(1+|y|)|\nabla\phi^0_{in}|\right].
\end{equation*}
We apply Proposition \ref{modelpropmode0} in the inner regime $\mathcal D_{2R_*}\times (0,T)$, then \eqref{mode0est444} gives
\begin{equation}\label{mode0est999}
\|\mathcal T[g]\|_{0,*}\lesssim \|g\|_{\nu,2+a}.
\end{equation}
We claim that
\begin{equation}\label{mode0est555}
\|\mathtt C[\phi^0_{in}]\|_{\nu,2+a'}\lesssim R_*^{a'-a}\log R\|\phi^0_{in}\|_{0,*}.
\end{equation}
Indeed, we evaluate
\begin{equation*}
\begin{aligned}
|\mathtt C[\phi^0_{in}]|=&~\left|\phi_{in}^0(\Delta \eta_{R_*}-\la^2 \partial_t \eta_{R_*}) +2\nabla \eta_{R_*}\cdot \nabla \phi_{in}^0\right|\\
\lesssim&~R_*^{-2}|\eta''|\la_*^{\nu}\frac{R_*^{4-a}\log R}{1+|y|^4} \|\phi^0_{in}\|_{0,*}\\
\lesssim&~\frac{\la_*^{\nu}\log R}{1+|y|^{2+a'}}R_*^{a'-a}\|\phi^0_{in}\|_{0,*}\\
\end{aligned}
\end{equation*}
which proves \eqref{mode0est555}. From \eqref{mode0est666} and \eqref{mode0est555}, we then get that
\begin{equation}\label{mode0est777}
\begin{aligned}
\|\phi^0_{out}\|_{\nu,a'}\lesssim&~ R_*^{a'-a}\log R\|\phi^0_{in}\|_{0,*}+\|(1-\eta_{R_*})h^0\|_{\nu,2+a'}\\
\lesssim&~ R_*^{a'-a}\log R\|\phi^0_{in}\|_{0,*}+ R_*^{a'-a}\|h^0\|_{\nu,2+a}.\\
\end{aligned}
\end{equation}
Next, we compute
\begin{equation*}
\begin{aligned}
|3U^2(y)\phi^0_{out}|\lesssim&~ \frac{\la_*^{\nu}}{1+|y|^{4+a'}}\|\phi^0_{out}\|_{\nu,a'}\\
\lesssim&~ \frac{\la_*^{\nu}}{1+|y|^{2+a}}\|\phi^0_{out}\|_{\nu,a'}.
\end{aligned}
\end{equation*}
So we get
\begin{equation}\label{mode0est888}
\|3U^2(y)\phi^0_{out}\|_{\nu,2+a}\lesssim \|\phi^0_{out}\|_{\nu,a'}.
\end{equation}
By \eqref{mode0est777} and \eqref{mode0est888}, we obtain
\begin{equation}\label{mode0est101010}
\|3U^2(y)\phi^0_{out}\|_{\nu,2+a}\lesssim R_*^{a'-a}\log R\|\phi^0_{in}\|_{0,*}+ R_*^{a'-a}\|h^0\|_{\nu,2+a}.
\end{equation}
Therefore, we conclude from \eqref{mode0est999} that
$$\left\|\mathcal T\left[3U^2(y)\phi^0_{out}[\phi^0_{in},h^0]+h^0\right]\right\|_{0,*}\lesssim R_*^{a'-a}\log R\|\phi^0_{in}\|_{0,*}+\|h^0\|_{\nu,2+a},$$
which shows that the operator
$$\phi_{in}^0\mapsto \mathcal T\left[3U^2(y)\phi^0_{out}[\phi^0_{in},h^0]+h^0\right]$$
is a contraction if $R_*$ is sufficiently large. A unique fixed point $\phi^0_{in}$ thus exists and
\begin{equation}\label{est-mode0in}
\|\phi^0_{in}\|_{0,*}\lesssim \|h^0\|_{\nu,2+a}.
\end{equation}
Replacing $a'$ by $a$ in the computations of \eqref{mode0est666} and \eqref{mode0est555}, we obtain
\begin{equation}\label{est-mode0out}
\|\phi^0_{out}\|_{\nu,a}\lesssim \log R\|h^0\|_{\nu,2+a}.
\end{equation}
Recalling $\phi^0=\phi^0_{out}+\eta_{R_*}\phi^0_{in}$ and combining \eqref{est-mode0in} and \eqref{est-mode0out}, we conclude
\begin{equation*}
|\phi^0(y,t)|+(1+|y|)|\nabla \phi^0(y,t)|\lesssim \la_*^{\nu}\log R\|h^0\|_{\nu,2+a}\begin{cases}
\frac{R^{\sigma(4-a)}}{1+|y|^4},~&\mbox{ for }~|y|\leq 2R^{\sigma},\\
\frac{1}{1+|y|^a},~&\mbox{ for }~2R^{\sigma}\leq|y|\leq 2R.\\
\end{cases}
\end{equation*}

Finally, we prove the estimate of $c^0$. By Proposition \ref{modelpropmode0}, we get
$$\tilde c^0(t)=-\frac{\int_{\mathcal D_{2R_*}}\big(h^0+3U^2(y)\phi^0_{out}[\phi^0_{in},h^0]\big) Z_5 dy}{\int_{\mathcal D_{2R_*}}|Z_5|^2 dy}.$$
Notice that $3U^2(y)\phi^0_{out}[\phi^0_{in},h^0]$ is linear in $h^0$. By \eqref{mode0est101010} and \eqref{est-mode0in}, we conclude that
\begin{equation*}
\begin{aligned}
\left|\int_{\mathcal D_{2R_*}}3U^2(y)\phi^0_{out}[\phi^0_{in},h^0] Z_5 dy\right|\lesssim&~ \la_*^{\nu}\left(R_*^{a'-a}\log R\|\phi^0_{in}\|_{0,*}+ R_*^{a'-a}\|h^0\|_{\nu,2+a}\right)\\
\lesssim&~\la_*^{\nu} R_*^{a'-a}\log R\|h^0\|_{\nu,2+a}.
\end{aligned}
\end{equation*}
The proof is complete.
\end{proof}

\medskip

\noindent {\bf 2. Construction at modes $1$ to $4$.}

\medskip

As we can see in mode $0$, the estimates are somewhat deteriorated inside the inner regime, and this will result in difficulties when solving the inner problem. One can observe that, for modes $1$ to $4$, the kernel function for the corresponding linearized operator has faster decay than mode $0$ which suggests that the estimates at modes $1$ to $4$ should be better than mode $0$'s. Inspired by the argument in \cite[Section 7]{17HMF}, we shall carry out the construction for modes $1$ to $4$ by means of the blow-up argument.

We perform the change of variable
$$\tau=\tau_{\la}(t)=\tau_0+\int_0^t \frac{ds}{\la_*^2(s)}$$
so that
$$\tau\sim \tau_0+\frac{|\log(T-t)|^2}{\la_* |\log T|}.$$
We choose the constant $\nu_1'>0$ so that
$$\tau^{-\nu_1'}\sim \la_*^{\nu_1}$$
for $\nu_1\in (0,1)$.

The main proposition for modes $1$ to $4$ is the following.
\begin{prop}\label{propmode1}
Assume $a_1\in (1,2),\nu_1\in(0,1),\|h^1\|_{\nu_1,2+a_1}<+\infty,$ and
$$\int_{\mathcal D_{2R}} h^1(y,\tau)Z_i(y)dy=0,~\mbox{for all}~\tau\in(\tau_0,+\infty),i=1,\cdots,4.$$
For sufficiently large $R$, there exists a pair $(\phi^1,e_0)$ solving
\begin{equation*}
\begin{cases}
\partial_{\tau} \phi^1=\Delta \phi^1+3U^{2}\phi^1+h^1(y,\tau),~&y\in \mathcal D_{2R}\times (\tau_0,\infty)\\
\phi^1(y,\tau_0)=e_0 Z_0(y),~&y\in \mathcal D_{2R}
\end{cases}
\end{equation*}
and $(\phi^1,e_0)=(\phi^1[h^1],e_0[h^1])$ defines a linear operator of $h^1$ that satisfies
\begin{equation*}
\|\phi^1\|_{\nu_1,a_1}\lesssim \|h^1\|_{\nu_1,2+a_1}
\end{equation*}
and
\begin{equation*}
|e_0[h^1]|\lesssim \tau^{-\nu_1'} \|h^1\|_{\nu_1,2+a_1}.
\end{equation*}
\end{prop}

In order to prove Proposition \ref{propmode1}, we consider the following Cauchy problem
\begin{equation}\label{blowupmode1''}
\begin{cases}
\partial_{\tau} \phi^1=\Delta \phi^1+3U^{2}\phi^1+h^1(y,\tau)-c(\tau)Z_0(y),&~y\in \mathbb{R}^4,~\tau\geq\tau_0\\
\phi^1(y,\tau_0)=0,&~y\in \mathbb{R}^4
\end{cases}
\end{equation}
with $h^1$ supported in $\mathcal D_{2R}\times (\tau_0,+\infty)$ and $\|h^1\|_{\nu_1,2+a_1}<+\infty$ in the $(y,\tau)$ variable, where $\nu_1\in(0,1)$ and $a_1\in(1,2)$. For notational convenience, we denote $\phi^1$ by $\phi$ and $h^1$ by $h$ in the following lemma.
\begin{lemma}\label{lemmamode1}
Assume $a_1\in (1,2),\nu_1\in(0,1),\|h\|_{\nu_1,2+a_1}<+\infty,$ and
$$\int_{\mathbb{R}^4} h(y,\tau)Z_i(y)dy=0,~\mbox{for all}~\tau\in(\tau_0,+\infty),~i=1,\cdots,4.$$
For $\tau_1$ sufficiently large, the solution $\phi(y,\tau)$ to
\begin{equation}\label{blowupmode1}
\left\{
\begin{aligned}
&\partial_{\tau} \phi=\Delta \phi+3U^{2}\phi+h(y,\tau)-c(\tau)Z_0(y),~~~y\in \mathbb{R}^4,~\tau\geq\tau_0\\
&\int_{\mathbb{R}^4} \phi(y,\tau) Z_0(y)dy=0,~~~\mbox{for all}~\tau\in(\tau_0,+\infty),\\
&\phi(y,\tau_0)=0,~~~y\in \mathbb{R}^4
\end{aligned}
\right.
\end{equation}
satisfies
\begin{equation}\label{est-blowupmode1}
\|\phi(y,\tau)\|_{a_1,\tau_1}\lesssim\|h\|_{2+a_1,\tau_1}.
\end{equation}
Further,
\begin{equation*}
|c(\tau)|\lesssim \tau^{-\nu_1'} R^{a_1} \|h\|_{2+a_1,\tau_1}~~\mbox{for}~\tau\in[\tau_0,\tau_1),
\end{equation*}
where $\|h\|_{b,\tau_1}:=\sup_{\tau\in [\tau_0,\tau_1)}\tau^{\nu_1'}\sup\limits_{y\in \mathbb{R}^4}(1+|y|^b)|h(y,\tau)|$.
\end{lemma}
\begin{proof}
Note that problem \eqref{blowupmode1} is equivalent to problem \eqref{blowupmode1''} for
$$c(\tau)=\frac{\int_{\mathbb R^4}h(y,\tau)Z_0(y)dy}{\int_{\mathbb R^4}|Z_0(y)|^2dy}.$$
By the time decay of $h$ and spatial decay of $Z_0$ (see \eqref{decay-Z_0}), we have
\begin{equation}\label{2}
|c(\tau)|\lesssim \tau^{-\nu_1'} R^{a_1} \|h\|_{2+a_1,\tau_1}.
\end{equation}
Now we prove \eqref{est-blowupmode1} by blow-up argument.

By standard parabolic theory, for any $R'>0$, there exists a constant $K$ depending on $R'$ and $\tau_1$ such that
$$|\phi(y,\tau)|\leq K,~\mbox{in}~B_{R'}\times[\tau_0,\tau_1].$$
It is easy to check that $\bar{\phi}=\frac{C}{1+|y|^{a_1}}$ is a super-solution to the original equation \eqref{blowupmode1}. Thus, $\|\phi\|_{a_1,\tau_1}<+\infty$. We claim that
$$\int_{\mathcal D_{2R}} \phi Z_i=0~\mbox{for all}~\tau\in[\tau_0,\tau_1],~i=1,\cdots,4.$$
Indeed, we multiply \eqref{blowupmode1} by $Z_i\eta_{R'},$ where $\eta_{R'}:=\eta(\frac{|y|}{R'})$ and $\eta$ is the standard cut-off function defined in \eqref{def-cutoff}. Then we have
$$\int_{\mathbb{R}^4}\phi(\cdot, \tau)\cdot Z_i\eta_{R'} = \int_{\tau_0}^\tau ds\int_{\mathbb{R}^4}(\phi(\cdot, s)\cdot L_0[\eta_{R'} Z_i] + h Z_i\eta_{R'} - c(s)Z_0Z_i\eta_{R'}).$$
Further computation gives
\begin{equation*}
\begin{aligned}
&\quad\int_{\mathbb{R}^4}[\phi(\cdot, s)\cdot L_0[\eta_{R'} Z_i] + h Z_i\eta_{R'} - c(s)Z_0Z_i\eta_{R'}].\\
&=\int_{\mathbb{R}^4}\phi(\cdot, s)[\eta_{R'}L_0[Z_i]+Z_i \Delta\eta_{R'}+2\nabla\eta_{R'}\cdot\nabla Z_i] + h Z_i\eta_{R'} - c(s)Z_0Z_i\eta_{R'}\\
&=O((R')^{-\epsilon})
\end{aligned}
\end{equation*}
for some $\epsilon>0$.
By taking $R'\rightarrow +\infty$, we get the desired result.

Now we want to prove
$$\|\phi\|_{a_1,\tau_1}\lesssim\|h\|_{2+a_1,\tau_1}.$$
We prove by contradiction. Suppose that there exist sequences $\tau_1^k\to +\infty$ and $\phi_k$, $h_k$, $c_k$ satisfying
\begin{equation*}
\left\{
\begin{aligned}
&\partial_\tau\phi_k = \Delta\phi_k + 3U^{2}(y)\phi_k + h_k - c_k(\tau)Z_0(y),~ y\in \mathbb{R}^4,~ \tau\geq \tau_0,\\
&\int_{\mathbb{R}^4}\phi_k(y, \tau)\cdot Z_i(y)dy = 0\text{ for all }\tau\in [\tau_0,\tau_1^k), ~i= 0, 1,\cdots, 4,\\
&\phi_k(y,\tau_0) = 0, ~y\in \mathbb{R}^4,
\end{aligned}
\right.
\end{equation*}
and
\begin{equation}\label{e5:38}
\|\phi_k\|_{a_1,\tau_1^k}=1,\quad \|h_k\|_{2+a_1,\tau_1^k}\to 0.
\end{equation}
By \eqref{2}, we know $\sup_{\tau\in (\tau_0, \tau_1^k)}\tau^{\nu_1'} c_k(\tau)\to 0$.
We claim that
\begin{equation}\label{e5:37}
\sup_{\tau_0 < \tau < \tau_1^k}\tau^{\nu_1'}|\phi_k(y,\tau)|\to 0
\end{equation}
uniformly on compact subsets of $\mathbb{R}^4$. We prove \eqref{e5:37} by contradiction.

\noindent {\bf Case 1.} For some $|y_k|\leq M$ and $\tau_0 < \tau_2^k < \tau_1^k$, if
\begin{equation*}
(\tau_2^k)^{\nu_1'}|\phi_k(y_k,\tau_2^k)|\geq \frac{1}{2},
\end{equation*}
then we know that $\tau_2^k\to +\infty$. Define
\begin{equation*}
\tilde{\phi}_k(y,\tau) = (\tau_2^k)^{\nu_1'}\phi_k(y,\tau_2^k + \tau).
\end{equation*}
Then
\begin{equation*}
\partial_\tau\tilde{\phi}_k = L_0[\tilde{\phi}_k] + \tilde{h}_k - \tilde{c}_k(\tau)Z_0(y)\text{ in }\mathbb{R}^4\times (\tau_0-\tau_2^k,0].
\end{equation*}
Due to the spatial decay of $h$ and $c$, we know $\tilde{h}_k\to 0$, $\tilde{c}_k\to 0$. By comparison, we get
\begin{equation*}
|\tilde{\phi}_k(y,\tau)|\leq \frac{1}{1+|y|^{a_1}}\text{ in }\mathbb{R}^4\times (\tau_0-\tau_2^k,0].
\end{equation*}
Hence, up to a subsequence, $\tilde{\phi}_k\to\tilde{\phi}$ uniformly on compact subsets with $\tilde{\phi}\neq 0$ and
\begin{equation}\label{eqn-case1}
\left\{
\begin{aligned}
&\partial_\tau\tilde{\phi} =\Delta\tilde{\phi} + 3U^{2}(y)\tilde{\phi}~\text{ in }~\mathbb{R}^4\times (-\infty, 0],\\
&\int_{\mathbb{R}^4}\tilde{\phi}(y, \tau)\cdot Z_j(y)dy = 0~\text{ for all }~\tau\in (-\infty, 0], ~ j= 0, 1,\cdots, 4,\\
&|\tilde{\phi}(y,\tau)|\leq \frac{1}{1+|y|^{a_1}}~\text{ in }~\mathbb{R}^4\times (-\infty, 0],\\
&\tilde{\phi}(y,\tau_0) = 0, ~y\in\mathbb{R}^4.
\end{aligned}
\right.
\end{equation}
Note that the orthogonality conditions above are well-defined if $a_1>1$. We now claim that $\tilde{\phi} = 0$. Indeed, by parabolic regularity theory, $\tilde{\phi}(y,\tau)$ is smooth. By scaling argument, we get
\begin{equation*}
\frac{1}{1+|y|}|\nabla\tilde{\phi}| + |\tilde{\phi}_\tau| + |\Delta\tilde{\phi}|\lesssim \frac{1}{1+|y|^{2+a_1}}.
\end{equation*}
Differentiating \eqref{eqn-case1} with respect to $\tau$, we get $\partial_\tau\tilde{\phi}_\tau =\Delta\tilde{\phi}_\tau + 3U^{2}(y)\tilde{\phi}_\tau$ and
\begin{equation*}
\frac{1}{1+|y|}|\nabla\tilde{\phi}_\tau| + |\tilde{\phi}_{\tau\tau}| + |\Delta\tilde{\phi}_\tau|\lesssim \frac{1}{1+|y|^{4+a_1}}.
\end{equation*}
Differentiating \eqref{eqn-case1} with respect to $\tau$ and integrating, we get
\begin{equation*}
\frac{1}{2}\partial_\tau\int_{\mathbb{R}^4}|\tilde{\phi}_\tau|^2 + B(\tilde{\phi}_\tau, \tilde{\phi}_\tau) = 0,
\end{equation*}
where
\begin{equation*}
B(\tilde{\phi}, \tilde{\phi}) = \int_{\mathbb{R}^4}|\nabla\tilde{\phi}|^2 - 3U^{2}(y)|\tilde{\phi}|^2dy.
\end{equation*}
Since $\int_{\mathbb{R}^4}\tilde{\phi}(y, \tau)\cdot Z_j(y)dy = 0$ for all $\tau\in (-\infty, 0]$, $j= 0, 1,\cdots, 4 $, $B(\tilde{\phi}, \tilde{\phi})\geq 0$. Also, we have
\begin{equation*}
\int_{\mathbb{R}^4}|\tilde{\phi}_\tau|^2 = -\frac{1}{2}\partial_\tau B(\tilde{\phi}, \tilde{\phi}).
\end{equation*}
From above, we get
\begin{equation*}
\partial_\tau\int_{\mathbb{R}^4}|\tilde{\phi}_\tau|^2 \leq 0,\quad \int_{-\infty}^0d\tau\int_{\mathbb{R}^4}|\tilde{\phi}_\tau|^2 < +\infty.
\end{equation*}
Hence $\tilde{\phi}_\tau = 0$. So $\tilde{\phi}$ is independent of $\tau$ and $L_0[\tilde{\phi}] = 0$. Since $\tilde{\phi}$ is bounded, by the non-degeneracy of $L_0$, $\tilde{\phi}$ is a linear combination of $Z_j$, $j = 1,\cdots, 4$. From orthogonality conditions $\int_{\mathbb{R}^4}\tilde{\phi}\cdot Z_j = 0$, $j = 1,\cdots, 4$, we obtain $\tilde{\phi} = 0$, a contradiction. Thus,
$$\sup_{\tau_0 < \tau < \tau_1^k}\tau^{\nu_1'}|\phi_k(y,\tau)|\to 0.$$

\noindent {\bf Case 2.}
Suppose there exists $y_k$ with $|y_k|\to +\infty$ such that
\begin{equation*}
(\tau_2^k)^{\nu_1'}(1+|y_k|^{a_1})|\phi_k(y_k, \tau_2^k)|\geq \frac{1}{2}.
\end{equation*}
Let
\begin{equation*}
\tilde{\phi}_k(z, \tau):=(\tau_2^k)^{\nu_1'}|y_k|^{a_1}\phi_k(y_k+|y_k|z,|y_k|^{2}\tau + \tau_2^k).
\end{equation*}
Then
\begin{equation*}
\partial_\tau \tilde{\phi}_k = \Delta\tilde{\phi}_k + a_k\tilde{\phi}_k + \tilde{h}_k(z,\tau),
\end{equation*}
where
$$a_k=3U^2(y_k+|y_k|z)$$
and
\begin{equation*}
\tilde{h}_k(z,\tau) = (\tau_2^k)^{\nu_1'}|y_k|^{2+a_1}h_k(y_k+|y_k|z,|y_k|^{2}\tau + \tau_2^k)-(\tau_2^k)^{\nu_1'}|y_k|^{2+a_1}c(|y_k|^{2}\tau+\tau_2^k)Z_0(y_k+|y_k|z).
\end{equation*}
By the definition of $h_k$,
\begin{equation*}
|\tilde{h}_k(z,\tau)| \lesssim o(1)\frac{((\tau_2^k)^{-1}|y_k|^{2}\tau + 1)^{-\nu_1'}}{|\hat{y}_k+z|^{2+a_1}}
\end{equation*}
with
\begin{equation*}
\hat{y}_k = \frac{y_k}{|y_k|}\to -{\hat e}
\end{equation*}
and $|\hat e|= 1$. Thus $\tilde{h}_k(z,\tau)\to 0$ uniformly on compact subsets of $\mathbb{R}^4\setminus\{\hat e\}\times (-\infty, 0]$ and $a_k$ has the same property. Moreover, $|\tilde{\phi}_k(0, \tau_0)|\geq \frac{1}{2}$ and
\begin{equation*}
|\tilde{\phi}_k(z,\tau)| \lesssim \frac{((\tau_2^k)^{-1}|y_k|^{2}\tau + 1)^{-\nu_1'}}{|\hat{y}_k+z|^{a_1}}.
\end{equation*}
Hence we may assume $\tilde{\phi}_k\to \tilde{\phi}\neq 0$ uniformly on compact subsets of $\mathbb{R}^4\setminus\{{\hat e}\}\times (-\infty,0]$ with $\tilde{\phi}$ satisfying
\begin{equation}\label{e5:39}
\tilde{\phi}_\tau = \Delta\tilde{\phi}\quad\text{in }\mathbb{R}^4\setminus\{{\hat e}\}\times (-\infty,0]
\end{equation}
and
\begin{equation}\label{e5:40}
|\tilde{\phi}(z,\tau)|\leq |z-{e}|^{-a_1}\quad\text{in }\mathbb{R}^4\setminus\{{\hat e}\}\times (-\infty,0].
\end{equation}
{\bf Claim}: functions $\tilde{\phi}$ satisfying (\ref{e5:39}) and (\ref{e5:40}) are 0.

Without loss of generality, we assume ${\hat e} = 0$. Then
\begin{equation}\label{e5:41}
\left\{
\begin{aligned}
&\tilde{\phi}_\tau = \Delta\tilde{\phi},\quad\text{in }\mathbb{R}^4\setminus\{0\}\times (-\infty,0],\\
&|\tilde{\phi}(z,\tau)|\leq |z|^{-a_1},\quad\text{in }\mathbb{R}^4\setminus\{0\}\times (-\infty,0].
\end{aligned}
\right.
\end{equation}

We consider the function $\bar u(\rho,\tau) = (\rho^{2} + c\tau)^{-\frac{a_1}{2}} + \epsilon \rho^{-2}$ for some constant $c > 0$.
Direct computations give us
$$\bar u_{\tau}-\Delta \bar u=a_1(\rho^2+c\tau)^{-\frac{a_1}{2}-2}\left[(2-a_1-\frac{c}{2})\rho^2+(4c-\frac{c^2}{2})\tau\right].$$
Then we know that if $a_1<2$, we can always find $c>0$ such that $\bar u(\rho,\tau+M)$ is a super-solution, where $M$ is a large constant. Thus, $|\tilde{\phi}|\leq 2 \bar u(\rho,\tau+M)$. By letting $M\to \infty$ and the arbitrariness of $\epsilon$, we get $\tilde{\phi} = 0$, a contradiction. The proof is complete.
\end{proof}

\begin{proof}[Proof of Proposition \ref{propmode1}]
From Lemma \ref{lemmamode1}, for any $\tau_1 > \tau_0$ with $\tau_0$ fixed sufficiently large, we have
\begin{equation*}
|\phi^1(y,\tau)|\lesssim\tau^{-\nu_1'}(1+|y|)^{-a_1}\|h^1\|_{2+a_1, \tau_1}\text{ for all }\tau\in (\tau_0, \tau_1), \,\,y\in \mathbb{R}^4
\end{equation*}
and
\begin{equation*}
|c(\tau)|\leq \tau^{-\nu_1'}R^{a_1}\|h^1\|_{2+a_1,\tau_1}\text{ for all }\tau\in (\tau_0,\tau_1).
\end{equation*}
By assumption, $\|h^1\|_{\nu_1,2+a_1} < +\infty$ and $\|h^1\|_{2+a_1, \tau_1}\leq \|h^1\|_{\nu_1,2+a_1}$ for an arbitrary $\tau_1$. It then follows that
\begin{equation*}
|\phi^1(y,\tau)|\lesssim\tau^{-\nu_1'}(1+|y|)^{-a_1}\|h^1\|_{\nu_1,2+a_1}\text{ for all }\tau\in (\tau_0, \tau_1),\,\, y\in \mathbb{R}^4
\end{equation*}
and
\begin{equation*}
|c(\tau)|\leq \tau^{-\nu_1'}R^{a_1}\|h^1\|_{\nu_1,2+a_1}\text{ for all }\tau\in (\tau_0, \tau_1).
\end{equation*}
By the arbitrariness of $\tau_1$, we have
\begin{equation*}
|\phi^1(y,\tau)|\lesssim\tau^{-\nu_1'}(1+|y|)^{-a_1}\|h^1\|_{\nu_1,2+a_1}\text{ for all }\tau\in (\tau_0, +\infty),\,\, y\in \mathbb{R}^4
\end{equation*}
and
\begin{equation*}
|c(\tau)|\leq \tau^{-\nu_1'}R^{a_1}\|h\|_{\nu_1,2+a_1}\text{ for all }\tau\in (\tau_0, +\infty).
\end{equation*}
The gradient estimates follows from the scaling argument and the standard parabolic theory. The proof is complete.
\end{proof}

\medskip

\noindent {\bf 3. Construction at higher modes $j\geq 5$.}

\medskip

For higher modes $j\geq 5$, we recall that
$$h^{\perp}=\sum\limits_{j=5}^{\infty} h_j(r,t)\Theta_j,~\phi^{\perp}[h^{\perp}]=\sum\limits_{j=5}^{\infty} \phi_j(r,t)\Theta_j$$
and let $\phi^{\perp}[h^{\perp}]$ solve the following problem
\begin{equation*}
\begin{cases}
\la^2\phi_t=\Delta_y \phi +3U^2(y)\phi+h^{\perp},~&\mbox{ in } \mathcal D_{2R}\times(0,T),\\
\phi=0,~&\mbox{ on } \partial \mathcal D_{2R}\times (0,T),\\
\phi(\cdot,0)=0,~&\mbox{ in } \mathcal D_{2R}.\\
\end{cases}
\end{equation*}
Similarly, it follows from \cite[Section 7]{Green16JEMS} that
\begin{equation}\label{highermodes}
|\phi^{\perp}(y,t)|+(1+|y|)|\nabla\phi^{\perp}(y,t)|\lesssim \la_*^{\nu} \frac{1}{1+|y|^a}\|h^{\perp}\|_{\nu,2+a}.
\end{equation}

\begin{proof}[Proof of Proposition \ref{lineartheory}]
Recall that
$$\phi[h]=\phi^0[h^0]+\phi^1[h^1]+\phi^{\perp}[h^{\perp}].$$
The validity of Proposition \ref{lineartheory} is concluded from Proposition \ref{propmode0}, Proposition \ref{propmode1} and \eqref{highermodes}. The proof is complete.
\end{proof}

\medskip

%%%%%%%%%%%%%%%%%%%%%%%%%%%%%%%%%%%%%%%%%%%%%%%%%%%%%%%%%%%%

\section{Solving the inner--outer gluing system}\label{sec-innerouter}

\medskip

In this section, we shall solve the inner--outer gluing system by the linear theories developed in Section \ref{sec-linearouter} and Section \ref{sec-linearinner}, and the Schauder fixed point theorem. Our goal is to find a solution $(\phi^0,\phi^1,\phi^{\perp},\psi,\la,\xi)$ to the inner--outer gluing system in Section \ref{sec-inneroutergluingscheme} so that the desired blow-up solution is constructed. We shall solve the inner--outer gluing system in the function space $\mathcal X$ defined in \eqref{def-mX}. First, we make some assumptions about the parameter functions.

We write
$$\la_*(t)=\frac{|\log T|(T-t)}{|\log(T-t)|^2}$$
and assume that for some numbers $c_1,c_2>0$,
$$c_1|\dot\la_*(t)|\leq |\dot\la(t)|\leq c_2 |\dot\la_*(t)|~\mbox{ for all }~t\in(0,T).$$
Recall that we take $R(t)=\la_*^{-\beta}(t)$ for $\beta\in(0,1/2).$

In Section \ref{subsec-outer} and Section \ref{subsec-inner}, for given $\|\phi^0\|_{0,\sigma,\nu,a},~\|\phi^1\|_{\nu_1,a_1},~\|\phi^{\perp}\|_{\nu,a},~\|\psi\|_*,~\|Z^*\|_{\infty},~\|\la\|_F,~\|\xi\|_G$ bounded,
we shall first estimate right hand sides $\mathcal G(\phi,\psi,\la,\xi)$ and $\mathcal H(\phi,\psi,\la,\xi)$ in the inner and outer problems. Here the above norms are defined in \eqref{def-normm0}, \eqref{def-normnua}, \eqref{def-norm*}, \eqref{def-normla} and \eqref{def-normxi}.

\medskip

\subsection{The outer problem: estimates of \texorpdfstring{$\mathcal G$}{G}}\label{subsec-outer}

\medskip

Recall from \eqref{outer} that the outer problem
\begin{equation*}
\psi_t =\Delta_{(r,z)} \psi+\frac{n-4}{r}\partial_r \psi  + \mathcal G(\phi,\psi,\la,\xi)~~\mbox{ in }\mathcal D\times (0,T)
\end{equation*}
where
\begin{equation*}
\begin{aligned}
\mathcal G(\phi,\psi,\la,\xi):=&~3\la^{-2}(1-\eta_R)U^2(y)(\psi+Z^*+\eta_*\Psi_0+\Psi_1)\\
&~+\la^{-3}\left[(\Delta_y \eta_R) \phi+2\nabla_y\eta_R\cdot\nabla_y \phi-\la^2\phi\partial_t \eta_R\right]+\frac{(n-4)\la^{-1}}{r}\phi\partial_r \eta_R\\
&~+(1-\eta_R)\mathcal K[\la,\xi]+\mathcal S_{\rm out}[\la,\xi]-\mathcal S_{\rm out}[\la_*,\xi_*]+(1-\eta_R)\mathcal N(\mathtt w)
\end{aligned}
\end{equation*}
with $\mathcal S_{\rm out}$ defined in \eqref{def-Sout}, $\mathcal K[\la,\xi]$ defined in \eqref{def-mK} and
\begin{equation*}
\mathcal N(\mathtt w):=(U^*+\mathtt w)^3-(U^*)^3-3U_{\la,\xi}^2\mathtt w.
\end{equation*}

In order to apply the linear theory Proposition \ref{outer-apriori}, we estimate all the terms in $\mathcal G(\phi,\psi,\la,\xi)$. Define
$$\mathcal G(\phi,\psi,\la,\xi)=g_1+g_2+g_3$$
with
$$
\begin{aligned}
&g_1:=3\la^{-2}(1-\eta_R)U^2(y)(\psi+Z^*+\eta_*\Psi_0+\Psi_1),\\
&g_2:=\la^{-3}\left[(\Delta_y \eta_R) \phi+2\nabla_y\eta_R\cdot\nabla_y \phi-\la^2\phi\partial_t\eta_R \right]+\frac{(n-4)\la^{-1}}{r}\,\phi\partial_{r}\eta_R,\\
&g_3:=(1-\eta_R)\mathcal K[\la,\xi]+\mathcal S_{\rm out}[\la,\xi]-\mathcal S_{\rm out}[\la_*,\xi_*]+(1-\eta_R)\mathcal N(\mathtt w).\\
\end{aligned}
$$
To estimate $g_1$, we need to estimate the corrections $\Psi_0$ and $\Psi_1$ defined in \eqref{def-Psi0} and \eqref{def-Psi1}, respectively.

\medskip

\noindent {\bf Estimates of $\Psi_0$ and $\partial_r \Psi_0$}

\medskip

We first estimate the size of $\Psi_0$. Decompose
\begin{equation}\label{est-Psi01}
\begin{aligned}
\Psi_0=-\alpha_0\int_{-T}^t \dot\la(s)\frac{1-e^{-\frac{\zeta^2}{4(t-s)}}}{\zeta^2} ds=-\alpha_0\left(\int_{-T}^{t-\frac{\zeta^2}{4}}+\int_{t-\frac{\zeta^2}{4}}^t\right) \dot\la(s)\frac{1-e^{-\frac{\zeta^2}{4(t-s)}}}{\zeta^2} ds.
\end{aligned}
\end{equation}
For the first integral above, we have two cases
\begin{itemize}
\item For $T-t>\frac{\zeta^2}{4}$, we further decompose
\begin{equation*}
\begin{aligned}
\int_{-T}^{t-\frac{\zeta^2}{4}} \dot\la(s)\frac{1-e^{-\frac{\zeta^2}{4(t-s)}}}{\zeta^2} ds=\left(
\int_{-T}^{t-(T-t)}+\int_{t-(T-t)}^{t-\frac{\zeta^2}{4}}\right)\dot\la(s)\frac{1-e^{-\frac{\zeta^2}{4(t-s)}}}{\zeta^2} ds.
\end{aligned}
\end{equation*}
Since $T-s<2(t-s)$ and $\frac{\zeta^2}{4(t-s)}<1$, the first integral above can be evaluated as
\begin{equation}\label{est-Psi02}
\begin{aligned}
&\quad\int_{-T}^{t-(T-t)}\dot\la(s)\frac{1-e^{-\frac{\zeta^2}{4(t-s)}}}{\zeta^2} ds\lesssim \int_{-T}^{t-(T-t)}\frac{|\dot \la(s)|}{T-s} ds\\
&\lesssim |\log T|\int_{-T}^{t-(T-t)}\frac{1}{(T-s)|\log(T-s)|^2} ds\\
%&\lesssim |\log T|\left|\frac{1}{|\log 2(T-t)|}-\frac{1}{\log 2T}\right|\\
&\lesssim 1.
\end{aligned}
\end{equation}
Similarly, for the second integral
\begin{equation}\label{est-Psi03}
\begin{aligned}
&\quad \int_{t-(T-t)}^{t-\frac{\zeta^2}{4}} \dot\la(s)\frac{1-e^{-\frac{\zeta^2}{4(t-s)}}}{\zeta^2} ds\lesssim \int_{t-(T-t)}^{t-\frac{\zeta^2}{4}}\frac{|\log T|}{(t-s)|\log (T-s)|^2} ds\\
&\lesssim \frac{|\log T|}{|\log(T-t)|^2}\left|\log\left(\frac{\zeta^2}{4}\right)-\log(T-t)\right|\\
&\lesssim |\dot\la|\left[\log(\rho^2+\la^2)+1\right].
\end{aligned}
\end{equation}

\item For $T-t<\frac{\zeta^2}{4}$, since $s<t-\frac{\zeta^2}{4}<t-(T-t)$, we compute
\begin{equation}\label{est-Psi04}
\int_{-T}^{t-\frac{\zeta^2}{4}}\frac{|\dot\la(s)|}{t-s}ds\lesssim \int_{-T}^{t-(T-t)} \frac{|\dot\la(s)|}{T-s}ds\lesssim 1.
\end{equation}
\end{itemize}

Then we evaluate
\begin{equation}\label{est-Psi05}
\int_{t-\frac{\zeta^2}{4}}^t \dot\la(s)\frac{1-e^{-\frac{\zeta^2}{4(t-s)}}}{\zeta^2} ds \lesssim \frac{1}{\zeta^2} \int_{t-\frac{\zeta^2}{4}}^t |\dot\la(s)|ds\lesssim 1.
\end{equation}
Combining \eqref{est-Psi01}--\eqref{est-Psi05}, we conclude that
\begin{equation}\label{est-Psi0}
|\Psi_0|\lesssim |\dot\la|\left[\log(\rho^2+\la^2)+1\right].
\end{equation}

Then we want to compute the size of $\partial_r \Psi_0$. Similarly, we have
\begin{equation}\label{drpsi01}
\begin{aligned}
\partial_r \Psi_0=&~-\alpha_0\int_{-T}^t \dot\la(s) k_{\zeta}(\zeta(\rho,t),t-s)\zeta_r ds\\
\lesssim&~(r-\sqrt{2(n-4)(T-t)})\int_{-T}^t \dot\la(s) \frac{\left(\frac{\zeta^2 e^{-\frac{\zeta^2}{4(t-s)}}}{2(t-s)}-2+2e^{-\frac{\zeta^2}{4(t-s)}}\right)}{\zeta^4} ds\\
=&~(r-\sqrt{2(n-4)(T-t)})\left(\int_{-T}^{t-\frac{\zeta^2}{4}}+\int_{t-\frac{\zeta^2}{4}}^t \right)\dot\la(s) \frac{\left(\frac{\zeta^2 e^{-\frac{\zeta^2}{4(t-s)}}}{2(t-s)}-2+2e^{-\frac{\zeta^2}{4(t-s)}}\right)}{\zeta^4} ds.\\
\end{aligned}
\end{equation}
For the first integral, we decompose
\begin{equation}\label{drpsi02}
\begin{aligned}
&\quad\int_{-T}^{t-\frac{\zeta^2}{4}} \dot\la(s) \frac{\left(\frac{\zeta^2 e^{-\frac{\zeta^2}{4(t-s)}}}{2(t-s)}-2+2e^{-\frac{\zeta^2}{4(t-s)}}\right)}{\zeta^4} ds\\
&=\left(\int_{-T}^{t-(T-t)}+\int_{t-(T-t)}^{t-\frac{\zeta^2}{4}} \right)\dot\la(s) \frac{\left(\frac{\zeta^2 e^{-\frac{\zeta^2}{4(t-s)}}}{2(t-s)}-2+2e^{-\frac{\zeta^2}{4(t-s)}}\right)}{\zeta^4} ds\\
&\lesssim \left(\int_{-T}^{t-(T-t)}+\int_{t-(T-t)}^{t-\frac{\zeta^2}{4}} \right)\dot\la(s)\left[\frac{1-\frac{\zeta^2}{4(t-s)}}{2\zeta^2(t-s)}-\frac{1}{2\zeta^2(t-s)}\right]ds\\
&= \left(\int_{-T}^{t-(T-t)}+\int_{t-(T-t)}^{t-\frac{\zeta^2}{4}} \right)\frac{\dot\la(s)}{(t-s)^2}ds\\
&\lesssim \int_{-T}^{t-(T-t)} \frac{\dot\la(s)}{(T-s)^2}ds +\int_{t-(T-t)}^{t-\frac{\zeta^2}{4}} \frac{\dot\la(s)}{(t-s)^2}ds\\
&\lesssim \frac{|\dot\la|}{T-t}+\frac{|\dot\la|}{\rho^2+\la^2}.
\end{aligned}
\end{equation}
On the other hand, one has
\begin{equation}\label{drpsi03}
\begin{aligned}
\int_{t-\frac{\zeta^2}{4}}^t \dot\la(s) \frac{\left(\frac{\zeta^2 e^{-\frac{\zeta^2}{4(t-s)}}}{2(t-s)}-2+2e^{-\frac{\zeta^2}{4(t-s)}}\right)}{\zeta^4} ds \lesssim \frac{|\dot\la|}{\rho^2+\la^2}.
\end{aligned}
\end{equation}
By \eqref{drpsi01}, \eqref{drpsi02} and \eqref{drpsi03}, we obtain
\begin{equation}\label{est-drPsi0}
\partial_r \Psi_0 \lesssim (r-\sqrt{2(n-4)(T-t)}) \max\left\{\frac{|\dot\la|}{T-t},\frac{|\dot\la|}{\rho^2+\la^2}\right\}.
\end{equation}

\medskip

\noindent {\bf Estimate of $\Psi_1$}

\medskip

Similar to the proof of Lemma \ref{lemma-rhs1}, Lemma \ref{lemma-rhs2} and Lemma \ref{lemma-rhs3} in the Appendix, the correction term $\Psi_1$ given by \eqref{def-Psi1} can be estimated as
\begin{equation}\label{duha-Psi1}
|\Psi_1|\lesssim \int_0^t\int_{\R^n} \frac{e^{-\frac{|x-w|^2}{4(t-s)}}}{(t-s)^{n/2}}\mathcal S_{\rm out}[\la_*,\xi_*](w,s)dw ds,
\end{equation}
where
\begin{equation*}
\begin{aligned}
\mathcal S_{\rm out}[\la_*,\xi_*]:=&~U_{\la_*,\xi_*}(\Delta_{(r,z)}\eta_*-\partial_t \eta_*)+2\nabla \eta_*\cdot \nabla U_{\la_*,\xi_*}+(\eta_*^3 -\eta_* )U_{\la_*,\xi_*}^3 +U_{\la_*,\xi_*}\frac{n-4}r \partial_r \eta_*\\
&~+\frac{n-4}{r}\Psi_0[\la_*]\partial_r \eta_*+\Psi_0[\la_*]\Delta_{(r,z)}\eta_*+2\nabla \eta_*\cdot\nabla \Psi_0[\la_*] -\Psi_0[\la_*]\partial_t \eta_*\\
&~+\frac{n-4}{r}\eta_*(1-\eta_R)\partial_r \Psi_0[\la_*]+\eta_*(1-\eta_R)\frac{n-4}{\la_* y_1+\xi_{r,*}}\la_*^{-2}\partial_{y_1} U(y)\\
&~+\eta_*(1-\eta_R)\la_*^{-2}\nabla U\cdot \dot\xi_*.
\end{aligned}
\end{equation*}
Here we take the term
$$U_{\la_*,\xi_*}\Delta_{(r,z)}\eta_*+\frac{n-4}{r}\eta_*(1-\eta_R)\partial_r \Psi_0[\la_*]+\eta_*(1-\eta_R)\left(\frac{n-4}{\la_* y_1+\xi_{r,*}}\la_*^{-2}\partial_{y_1} U(y)+\la_*^{-2}\nabla U\cdot \dot\xi_*\right)$$
in $\mathcal S_{\rm out}[\la_*,\xi_*]$ as an example, and all the estimates for other terms can be carried out in a similar manner. By the definition of $\eta_*$ in \eqref{def-eta*} and \eqref{est-drPsi0}, we have
\begin{equation*}
\begin{aligned}
&\quad \left|U_{\la_*,\xi_*}\Delta_{(r,z)}\eta_*+\frac{n-4}{r}\eta_*(1-\eta_R)\partial_r \Psi_0[\la_*]\right|\\
&\lesssim \frac{\la_*}{(T-t)^2}\chi_{\{\rho\sim\sqrt{T-t}\}}+\frac{1}{\la_* R\sqrt{T-t}}\chi_{\{\la_*R\lesssim\rho\lesssim\sqrt{T-t}\}}.
\end{aligned}
\end{equation*}
On the other hand, thanks to our choice
$$\xi_{r,*}(t)=\sqrt{2(n-4)(T-t)},$$
we have
$$\dot\xi_{r,*}+\frac{n-4}{\la_* y_1+\xi_{r,*}}= -\frac{(n-4)\la_* y_1}{\xi_{r,*}(\la_* y_1+\xi_{r,*})},$$
and thus
\begin{equation*}
\begin{aligned}
&\quad \left|\eta_*(1-\eta_R)\left(\frac{n-4}{\la_* y_1+\xi_{r,*}}\la_*^{-2}\partial_{y_1} U(y)+\la_*^{-2}\nabla U\cdot \dot\xi_*\right)\right|\\
&\lesssim \frac{\la_*^{-1}y_1^2}{\xi_{r,*}(\la_* y_1+\xi_{r,*})(1+|y|^2)^2}\chi_{\{R\leq |y|\leq \frac{2\delta\sqrt{T-t}}{\la_*}\}}\\
&\lesssim \frac{\la_*^{-1}}{(T-t)R^2}.
\end{aligned}
\end{equation*}
Then from \eqref{duha-Psi1}, we obtain
\begin{equation}\label{est-Psi1}
\begin{aligned}
|\Psi_1|\lesssim& ~\int_0^t \frac{\la_*(s)}{(T-s)^2}+\frac{1}{\la_*(s) R(s)\sqrt{T-s}}+\frac{\la_*^{-1}(s)}{(T-s)R^2(s)} ~ds\\
\lesssim&~ |\log(T-t)|\la_*^{2\beta-1}(t),
\end{aligned}
\end{equation}
where we have used $R(t)=\la_*^{-\beta}(t)$.

\medskip

\noindent {\bf Estimate of $g_1$.}

\medskip

Since we shall solve $\psi$ in the function space $X_{\psi}$ defined in \eqref{def-fcnspaces}, we get
\begin{equation*}
\begin{aligned}
g_1=&~3\la^{-2}(1-\eta_R)U^2(y)(\psi+Z^*+\eta_*\Psi_0+\Psi_1)\\
\lesssim&~\frac{R^{-2}(t)\la_*^{\nu-2+\frac4n}(0)R^{-2-\alpha+\frac8n}(0)|\log T|}{|(r,z)-\xi(t)|^2}\chi_{\{|(r,z)-\xi(t)|\geq\la_*R\}}\|\psi\|_*\\
&~+\frac{R^{-2}(t)}{|(r,z)-\xi(t)|^2}\chi_{\{|(r,z)-\xi(t)|\geq\la_*R\}}\|Z^*\|_{\infty}+\frac{R^{-2}(t)|\log(T-t)|}{|(r,z)-\xi(t)|^2}\chi_{\{|(r,z)-\xi(t)|\geq\la_*R\}}\|\la\|_{\infty}\\
&~+\frac{R^{-2}(t)\la_*^{2\beta-1}(t)|\log(T-t)|}{|(r,z)-\xi(t)|^2}\chi_{\{|(r,z)-\xi(t)|\geq\la_*R\}}
\end{aligned}
\end{equation*}
by using \eqref{est-Psi1}. So by the choice of the weight $\varrho_2$ as in \eqref{def-weights}, we have
\begin{equation}\label{g1**}
\|g_1\|_{**}\lesssim T^{\epsilon_0}(\|\psi\|_*+\|Z^*\|_{\infty}+\|\la\|_{\infty}+1)
\end{equation}
provided
\begin{equation}\label{g1**cond}
\begin{cases}
\nu-2+\frac4n+\beta(4+\alpha-\frac8n)-\nu_2>0,\\
2\beta-\nu_2>0,\\
4\beta-1-\nu_2>0.\\
\end{cases}
\end{equation}
Here $\epsilon_0$ is a small positive number.

\medskip

\noindent {\bf Estimate of $g_2$.}

\medskip

Thanks to the cut-off, $g_2$ is supported in
$$\left\{(r,z,t)\in\mathcal D\times(0,T):~\la_*R\leq|(r,z)-\xi(t)|\leq 2\la_*R,~t\in(0,T)\right\},$$
and we have
\begin{equation*}
\begin{aligned}
&\quad g_2= \la^{-3}\left[(\Delta_y \eta_R) \phi+2\nabla_y\eta_R\cdot\nabla_y \phi-\la^2(\partial_t\eta_R)\phi\right]+\frac{(n-4)\la^{-1}}{r}\,\phi\partial_{r}\eta_R\\
&\lesssim \la_*^{\nu-3}R^{-2-a}\log R\left(1+r-\sqrt{2(n-4)(T-t)}+\dot\la_* \la_* R^2\right)\chi_{\{|(r,z)-\xi(t)|\sim\la_*R\}}\\
&\qquad \times\left(\|\phi^0\|_{0,\sigma,\nu,a}+\|\phi^1\|_{\nu_1,a_1}+\|\phi^{\perp}\|_{\nu,a}\right)\\
&\quad +\frac{\la_*^{-2}}{\sqrt{T-t}}\left(\frac{\la_*^{\nu}\log R}{1+R^{1+a}}\|\phi^0\|_{0,\sigma,\nu,a}+\frac{\la_*^{\nu_1}}{1+R^{1+a_1}}\|\phi^1\|_{\nu_1,a_1}+\frac{\la_*^{\nu}}{1+R^{1+a}}\|\phi^{\perp}\|_{\nu,a}\right)\\
&\lesssim (R^{\alpha-a}\log R+\la_*^{1/2}R \log R+\la_*^{\nu_1+1/2-\nu}R^{1+a-a_1})\varrho_1\left(\|\phi^0\|_{0,\sigma,\nu,a}+\|\phi^1\|_{\nu_1,a_1}+\|\phi^{\perp}\|_{\nu,a}\right).
\end{aligned}
\end{equation*}
So it follows that
\begin{equation}\label{g2**}
\left\|g_2\right\|_{**}\lesssim T^{\epsilon_0}\left(\|\phi^0\|_{0,\sigma,\nu,a}+\|\phi^1\|_{\nu_1,a_1}+\|\phi^{\perp}\|_{\nu,a}\right)
\end{equation}
provided
\begin{equation}\label{g2**cond}
\begin{cases}
0<\alpha<a<1,\\
\beta<\frac12,\\
\beta(1+\alpha-a_1)+\nu-\nu_1-\frac12<0.\\
\end{cases}
\end{equation}
Here $\epsilon_0$ is a small positive number.

\medskip

\noindent {\bf Estimate of $g_3$.}

\medskip

We want to estimate $g_3=(1-\eta_R)\mathcal K_1[\la,\xi]+\mathcal S_{\rm out}[\la,\xi]-\mathcal S_{\rm out}[\la_*,\xi_*]+(1-\eta_R)\mathcal N(\mathtt w)$. Recall from \eqref{def-mK1} that
$$\mathcal K_1[\la,\xi]=\eta_*\left[\frac{2\alpha_0 \la^{-2}(t)\dot{\la}(t)}{(1+|y|^2)^2}-\mathcal R[\la]\right].$$
We first estimate $\mathcal R[\la]$ defined in \eqref{def-mR}. Recalling
\begin{equation*}
\begin{aligned}
\mathcal R[\la]=&~\alpha_0\left[\frac{(r-\xi_r)\dot{\xi}_r+(z-\xi_z)\dot{\xi}_z-\la(t)\dot{\la}(t)}{\la(t)(1+|y|^2)^{1/2}}\right]\int_{-T}^t \dot{\la}(s) k_{\zeta}(\zeta,t-s)ds\\
&~+\frac{\alpha_0}{\la(t)(1+|y|^2)^{3/2}}\int_{-T}^t \dot{\la}(s)\left[-\zeta k_{\zeta\zeta}(\zeta,t-s)+k_\zeta(\zeta,t-s)\right]ds\\
\end{aligned}
\end{equation*}
and
$$k(\zeta,t):=\frac{1-e^{-\frac{\zeta^2}{4t}}}{\zeta^2},$$
we have
\begin{equation}\label{est-mRnew11}
\begin{aligned}
\int_{-T}^t \dot{\la}(s) k_{\zeta}(\zeta,t-s)ds=&~\int_{-T}^t\dot\la(s)\left[\frac{e^{-\frac{\zeta^2}{4(t-s)}}}{2\zeta(t-s)}-\frac{2\left(1-e^{-\frac{\zeta^2}{4(t-s)}}\right)}{\zeta^3}\right] ds\\
\lesssim&~\left(\int_{-T}^{t-(T-t)}+\int_{t-(T-t)}^{t-\frac{\zeta^2}{4}}\right)\dot\la(s)\frac{\zeta}{(t-s)^2} ds+\int_{t-\frac{\zeta^2}{4}}^t \frac{\dot\la(s)}{\zeta^3}ds\\
\lesssim&~\int_{-T}^{t-(T-t)}\dot\la(s)\frac{\zeta}{(T-s)^2}ds + \int_{t-(T-t)}^{t-\frac{\zeta^2}{4}} \dot\la(s)\frac{\zeta}{(t-s)^2} ds +\frac{|\dot\la|}{\zeta}\\
\lesssim&~\frac{|\dot\la|\zeta}{T-t}+\frac{|\dot\la|}{\zeta}
\end{aligned}
\end{equation}
and
\begin{equation}\label{est-mRnew22}
\begin{aligned}
&\quad \int_{-T}^t \dot{\la}(s)\left[-\zeta k_{\zeta\zeta}(\zeta,t-s)+k_\zeta(\zeta,t-s)\right]ds\\
&=\int_{-T}^t \dot\la(s) \left[\frac{2e^{-\frac{\zeta^2}{4(t-s)}}}{\zeta(t-s)}-\frac{8\left(1-e^{-\frac{\zeta^2}{4(t-s)}}\right)}{\zeta^3}+\frac{\zeta e^{-\frac{\zeta^2}{4(t-s)}}}{4(t-s)^2}\right] ds\\
&\lesssim \left(\int_{-T}^{t-(T-t)}+\int_{t-(T-t)}^{t-\frac{\zeta^2}{4}}\right)\dot\la(s)\left(\frac{\zeta}{(t-s)^2}+\frac{\zeta^3}{(t-s)^3} \right)ds+\int_{t-\frac{\zeta^2}{4}}^t \frac{\dot\la(s)}{\zeta^3}ds\\
&\lesssim \frac{|\dot\la|\zeta}{T-t}+\frac{|\dot\la|\zeta^3}{(T-t)^2}+\frac{|\dot\la|}{\zeta}.
\end{aligned}
\end{equation}
From \eqref{est-mRnew11} and \eqref{est-mRnew22}, we obtain
\begin{equation}\label{est-mR*}
\begin{aligned}
\mathcal R[\la] \lesssim |\dot\la|\left(\frac{1}{(T-t)(1+|y|^2)}+\frac{\la^2}{(T-t)^2}+\frac{1}{\la^2(1+|y|^2)^2}+\frac{\la(y\cdot \dot\xi+\dot\la)}{T-t}+\frac{y\cdot \dot\xi+\dot\la}{\la(1+|y|^2)}\right).
\end{aligned}
\end{equation}

Now we evaluate the first term $(1-\eta_R)\mathcal K_1[\la,\xi]$ in $g_3$ by using \eqref{est-mR*}. Thanks to the cut-off $(1-\eta_R)\eta_*$, the term $(1-\eta_R)\mathcal K_1[\la,\xi]$ is supported in
$$\left\{(r,z,t)\in\mathcal D\times (0,T):\la_*R\leq |(r,z)-\xi(t)|\leq 2\delta\sqrt{T-t}\right\}.$$
So we get
\begin{equation*}
\begin{aligned}
(1-\eta_R)\mathcal K_1[\la,\xi]=&~(1-\eta_R)\eta_* \left[ \frac{2\alpha_0 \la^{-2}(t)\dot{\la}(t)}{(1+|y|^2)^2}-\mathcal R[\la] \right]\\
\lesssim&~ R^{-2}\la_*^{-\nu_2}\|\la\|_{\infty} \varrho_2+\|\la\|_{\infty}\la_*^{1-\nu_2}\varrho_2(1+\|\xi\|_G) +T^{\epsilon_0}\|\la\|_{\infty} \varrho_3,\\
\end{aligned}
\end{equation*}
where the $\|\cdot\|_G$-norm is defined in \eqref{def-normxi}.
We see that if
\begin{equation}\label{g3**1cond}
\begin{cases}
\nu_2-2\beta<0\\
\nu_2-1<0\\
\end{cases}
\end{equation}
then one has
\begin{equation}\label{g3**1}
\|(1-\eta_R)\mathcal K_1[\la,\xi]\|_{**}\lesssim T^{\epsilon_0}\left(\|\la\|_{\infty}+\|\xi\|_{G}+1\right)
\end{equation}
for some $\epsilon_0>0$.

Next, we consider $\mathcal S_{\rm out}[\la,\xi]-\mathcal S_{\rm out}[\la_*,\xi_*]$. Similarly, direct computations yield that
\begin{equation*}
\begin{aligned}
\left|\mathcal S_{\rm out}[\la,\xi]-\mathcal S_{\rm out}[\la_*,\xi_*]\right|\lesssim&~ \frac{\la_*^{1+\sigma'}(t)}{T-t}\frac{1}{|(r,z)-\xi(t)|^2}\chi_{\{|(r,z)-\xi(t)|\geq\la_*R\}}\\
\lesssim&~ T^{\epsilon_0} \varrho_2
\end{aligned}
\end{equation*}
where $\sigma',\epsilon>0$. So we get
\begin{equation}\label{g3**2}
\|\mathcal S_{\rm out}[\la,\xi]-\mathcal S_{\rm out}[\la_*,\xi_*]\|_{**}\lesssim T^{\epsilon_0}.
\end{equation}

Finally, we compute the nonlinear terms
\begin{equation*}
\begin{aligned}
(1-\eta_R)\mathcal N(\mathtt w)=&~(1-\eta_R)\left((U^*+\mathtt w)^3-(U^*)^3-3U_{\la,\xi}^2\mathtt w \right) \lesssim (1-\eta_R)U^* \mathtt w^2\\
\lesssim&~\frac{\la^{-1}}{1+|y|^2}\left[(\la^{-1}\eta_R\phi)^2+(\eta_*\Psi_0)^2+\psi^2+(Z^*)^2+\Psi_1^2\right](1-\eta_R)\\
%\lesssim&~\frac{\la_*^{-3}(1-\eta_R)}{1+|y|^2}\eta_R^2 \bigg(\frac{\la_*^{2\nu}R^{2\sigma(4-a)}(\log R)^2}{1+|y|^8}\|\phi^0\|^2_{0,\sigma,\nu,a}+\frac{\la_*^{2\nu_1}}{1+|y|^{2a_1}}\|\phi^1\|_{\nu_1,a_1}^2\\
%&~+\frac{\la_*^{2\nu}}{1+|y|^{2a}}\|\phi^{\perp}\|_{\nu,a}^2\bigg)+\frac{\la_*^{-1}(1-\eta_R)}{1+|y|^2}\eta_*^2\left(|\dot\la_*|\left[\log(\rho^2+\la_*^2)+1\right]\right)^2\\
%&~+\frac{\la_*^{-1}(1-\eta_R)}{1+|y|^2}\|Z^*\|^2_{\infty}+\frac{\la_*^{-1}(1-\eta_R)}{1+|y|^2}|\log(T-t)|^2\la_*^{4\beta-2}\\
%&~+\frac{\la_*^{-1}(t)(1-\eta_R)}{1+|y|^2}\la_*^{2\nu-4+\frac8n}(0)R^{-4-2\alpha+\frac{16}n}(0)|\log T|^2\|\psi\|^2_{*}\\
\lesssim&~\la_*^{\nu}R^{2\sigma(4-a)+\alpha-8}(\log R)^2\|\phi^0\|^2_{0,\sigma,\nu,a}\varrho_1+\la_*^{2\nu_1-\nu}R^{\alpha-2a_1}\|\phi^1\|^2_{\nu_1,a_1}\varrho_1\\
&~+\la_*^{\nu}R^{\alpha-2a}\|\phi^{\perp}\|^2_{\nu,a}\varrho_1+|\log(T-t)|^2|\dot\la_*|^2\la_*^{1-\nu_2}\varrho_2\\
&~+\la_*^{1-\nu_2}(t)\la_*^{2\nu-4+\frac8n}(0)R^{-4-2\alpha+\frac{16}n}(0)|\log T|^2\varrho_2\|\psi\|_*^2+\la^{1-\nu_2}\varrho_2\|Z^*\|_{\infty}^2\\
&~+\la_*^{4\beta-\nu_2-1}|\log(T-t)|^2\varrho_2.\\
\end{aligned}
\end{equation*}
Therefore, we obtain that for $\epsilon_0>0$
\begin{equation}\label{g3**3}
\|(1-\eta_R)\mathcal N(\mathtt w)\|_{**}\lesssim T^{\epsilon_0}\left(\|\phi^0\|_{0,\sigma,\nu,a}^2+\|\phi^1\|^2_{\nu_1,a_1}+\|\phi^{\perp}\|_{\nu,a}^2+\|\psi\|_*^2+\|Z^*\|_{\infty}^2+\|\la\|_{\infty}+1\right)
\end{equation}
provided
\begin{equation}\label{g3**3cond}
\begin{cases}
\nu-\beta\left(2\sigma(4-a)+\alpha-8\right)>0,\\
2\nu_1-\nu-\beta(\alpha-2a_1)>0,\\
\nu-\beta(\alpha-2a)>0,\\
\nu_2<1,\\
2\nu-3+\frac8n-\nu_2+\beta(4+2\alpha-\frac{16}n)>0,\\
4\beta-\nu_2-1>0.\\
\end{cases}
\end{equation}

Collecting \eqref{g1**}, \eqref{g1**cond}, \eqref{g2**}, \eqref{g2**cond}, \eqref{g3**1}, \eqref{g3**1cond}, \eqref{g3**2}, \eqref{g3**3} and \eqref{g3**3cond},  we conclude that for a fixed number $\epsilon_0>0$
\begin{equation}\label{outer-contraction}
\|\mathcal G\|_{**}\lesssim T^{\epsilon_0}\left(\|\psi\|_*+\|Z^*\|_{\infty}+\|\phi^0\|_{0,\sigma,\nu,a}+\|\phi^1\|_{\nu_1,a_1}+\|\phi^{\perp}\|_{\nu,a}+\|\la\|_{\infty}+\|\xi\|_{G}+1\right)
\end{equation}
with the parameters $\beta,a,a_1,\alpha,\nu,\nu_1,\nu_2,\sigma$ chosen in the following range
\begin{equation}\label{cond-outer-contraction}
\begin{cases}
\nu-2+\frac4n+\beta(4+\alpha-\frac8n)-\nu_2>0,\\
2\beta-\nu_2>0,\\
4\beta-\nu_2-1>0,\\
0<\alpha<a<1,\\
\beta<\frac12,\\
\beta(1+\alpha-a_1)+\nu-\nu_1-\frac12<0,\\
\nu-\beta\left(2\sigma(4-a)+\alpha-8\right)>0,\\
2\nu_1-\nu-\beta(\alpha-2a_1)>0,\\
\nu-\beta(\alpha-2a)>0,\\
\nu_2<1,\\
2\nu-3+\frac8n-\nu_2+\beta(4+2\alpha-\frac{16}n)>0.\\
\end{cases}
\end{equation}

\medskip

\subsection{The inner problems: estimates of \texorpdfstring{$\mathcal H^0$}{H0}, \texorpdfstring{$\mathcal H^1$}{H1} and \texorpdfstring{$\mathcal H^{\perp}$}{Hperp}}\label{subsec-inner}

\medskip

Recall from \eqref{inner} that the inner problem is the following
\begin{equation*}
\begin{aligned}
\la^2 \phi_t = \Delta_y \phi +3U^2(y)\phi+\mathcal H(\phi,\psi,\la,\xi)~~\mbox{ in }\mathcal D_{2R}\times (0,T)
\end{aligned}
\end{equation*}
where
\begin{equation*}
\begin{aligned}
\mathcal H(\phi,\psi,\la,\xi)(y,t):=&~3\la U^{2}(y)[\eta_*\Psi_0+\psi+Z^*](\la y+\xi,t)+\frac{(n-4)\la}{\la y_1+\xi_r}\phi_{y_1}\\
&~+\la\left[\dot\la (\nabla_y\phi\cdot y+\phi)+\nabla_y \phi\cdot\dot\xi\right]+\frac{(n-4)\la^3}{r}\eta_*\partial_r\Psi_0\\
&~+\la^3 \mathcal N(\mathtt w)+\la^3 \mathcal K[\la,\xi]+3\la U^{2}(y)\Psi_1
\end{aligned}
\end{equation*}
with $\mathcal K[\la,\xi]$ defined in \eqref{def-mK}. Since the inner--outer gluing relies on delicate analysis of the space-time decay of solutions, we further decompose the inner problem \eqref{inner} into three different spherical harmonic modes
\begin{equation*}
\begin{cases}
\la^2 \phi^0_t = \Delta_y \phi^0 +3U^2(y)\phi^0+\mathcal H^0(\phi,\psi,\la,\xi)~~&\mbox{ in }\mathcal D_{2R}\times (0,T)\\
\phi^0(\cdot,0)=0~~&\mbox{ in }\mathcal D_{2R}\\
\end{cases}
\end{equation*}

\begin{equation*}
\begin{cases}
\la^2 \phi^1_t = \Delta_y \phi^1 +3U^2(y)\phi^1+\mathcal H^1(\phi,\psi,\la,\xi)~~&\mbox{ in }\mathcal D_{2R}\times (0,T)\\
\phi^1(\cdot,0)=0~~&\mbox{ in }\mathcal D_{2R}\\
\end{cases}
\end{equation*}

\begin{equation*}
\begin{cases}
\la^2 \phi^{\perp}_t = \Delta_y \phi^{\perp} +3U^2(y)\phi^{\perp}+\mathcal H^{\perp}(\phi,\psi,\la,\xi)~~&\mbox{ in }\mathcal D_{2R}\times (0,T)\\
\phi^{\perp}(\cdot,0)=0~~&\mbox{ in }\mathcal D_{2R}\\
\end{cases}
\end{equation*}
with
\begin{equation}\label{def-mHm0}
\mathcal H^0(\phi,\psi,\la,\xi)=\int_{\mathbb{S}^3} \mathcal H(\phi,\psi,\la,\xi) \Theta_0(\theta) d\theta,
\end{equation}
\begin{equation}\label{def-mHm1}
\mathcal H^1(\phi,\psi,\la,\xi)=\sum\limits_{j=1}^4\left(\int_{\mathbb{S}^3} \mathcal H(\phi,\psi,\la,\xi) \Theta_j(\theta) d\theta\right) \Theta_j,
\end{equation}
and
\begin{equation}\label{def-mHmh}
\mathcal H^{\perp}(\phi,\psi,\la,\xi)=\sum\limits_{j\geq 5}\left(\int_{\mathbb{S}^3} \mathcal H(\phi,\psi,\la,\xi) \Theta_j(\theta) d\theta\right) \Theta_j,
\end{equation}
where $\Theta_j~(j=0,1,\cdots)$ are spherical harmonics. From the linear theory in Section \ref{sec-linearinner}, we know that for $\mathcal H=\mathcal H^0+\mathcal H^1+\mathcal H^{\perp}$ satisfying
$$\|\mathcal H^0\|_{\nu,2+a},~\|\mathcal H^1\|_{\nu_1,2+a_1},~\|\mathcal H^{\perp}\|_{\nu,2+a}<+\infty,$$
there exists a solution $(\phi^0,\phi^1,\phi^{\perp},c^0,c^{\ell})$ ($\ell=1,\cdots,4$) solving the projected inner problems
\begin{equation}\label{eqn-mode000}
\begin{cases}
\la^2 \phi^0_t = \Delta_y \phi^0 +3U^2(y)\phi^0+\mathcal H^0(\phi,\psi,\la,\xi)+c^0 Z_5~~&\mbox{ in }\mathcal D_{2R}\times (0,T)\\
\phi^0(\cdot,0)=0~~&\mbox{ in }\mathcal D_{2R}\\
\end{cases}
\end{equation}

\begin{equation}\label{eqn-mode111}
\begin{cases}
\la^2 \phi^1_t = \Delta_y \phi^1 +3U^2(y)\phi^1+\mathcal H^1(\phi,\psi,\la,\xi)+\sum\limits_{\ell=1}^4 c^{\ell}Z_{\ell}~~&\mbox{ in }\mathcal D_{2R}\times (0,T)\\
\phi^1(\cdot,0)=0~~&\mbox{ in }\mathcal D_{2R}\\
\end{cases}
\end{equation}

\begin{equation}\label{eqn-modehhh}
\begin{cases}
\la^2 \phi^{\perp}_t = \Delta_y \phi^{\perp} +3U^2(y)\phi^{\perp}+\mathcal H^{\perp}(\phi,\psi,\la,\xi)~~&\mbox{ in }\mathcal D_{2R}\times (0,T)\\
\phi^{\perp}(\cdot,0)=0~~&\mbox{ in }\mathcal D_{2R}\\
\end{cases}
\end{equation}
and the inner solution $\phi[\mathcal H]=\phi^0[\mathcal H^0]+\phi^1[\mathcal H^1]+\phi^{\perp}[\mathcal H^{\perp}]$ with proper space-time decay can be found ensuring the inner--outer gluing to be carried out. First, we choose all the parameters such that
$$\|\mathcal H^0\|_{\nu,2+a},~\|\mathcal H^1\|_{\nu_1,2+a_1},~\|\mathcal H^{\perp}\|_{\nu,2+a}<+\infty.$$
To this end, we first give some estimates for $\mathcal H$.
\begin{itemize}
\item By \eqref{est-Psi0}, we have
\begin{equation}\label{est-mH1}
\begin{aligned}
&\quad \left|3\la U^{2}(y)[\eta_*\Psi_0(\la y+\xi,t)+\psi(\la y+\xi,t)+Z^*(\la y+\xi,t)]\right|\\
&\lesssim \frac{\la_*(t)}{1+|y|^4}\left[|\dot\la_*|\left(\log{\la_*}+\log(1+|y|)\right)+\la_*^{\nu-2+\frac4n}(0)R^{-2-\alpha+\frac8n}(0)|\log T|\|\psi\|_{*}+\|Z^*\|_{\infty}\right].
\end{aligned}
\end{equation}

\medskip

\item By \eqref{est-drPsi0}, we obtain
\begin{equation}\label{est-mH2}
\begin{aligned}
&\quad\left|\la\left[\dot\la (\nabla_y\phi\cdot y+\phi)+ \nabla_y \phi\cdot\dot\xi\right]+\frac{(n-4)\la^3}{r}\eta_*\partial_r\Psi_0\right|\\
&\lesssim \la_*|\dot\la_*|\left(\frac{\la_*^{\nu}R^{\sigma(4-a)}\log R}{1+|y|^{4}}\|\phi^0\|_{0,\sigma,\nu,a}+\frac{\la_*^{\nu_1}}{1+|y|^{a_1}}\|\phi^1\|_{\nu_1,a_1}+\frac{\la_*^{\nu}}{1+|y|^{a}}\|\phi^{\perp}\|_{\nu,a}\right)\\
&\quad + \la_*\dot\xi \frac{\la_*^{\nu_1}}{1+|y|^{1+a_1}}\|\phi^1\|_{\nu_1,a_1}+\frac{\la_*}{\sqrt{T-t}}(r-\sqrt{2(n-4)(T-t)})\frac{|\dot\la_*|}{1+|y|^2}.
\end{aligned}
\end{equation}

\medskip

\item Using \eqref{est-Psi0}, \eqref{est-Psi1} and \eqref{est-mR*}, we evaluate
\begin{equation}\label{est-mH3}
\begin{aligned}
&\quad \left|\la^3\mathcal N(\mathtt w)+\la^3 \mathcal K[\la,\xi]+3\la U^{2}(y)\Psi_1\right|\\
&\lesssim \left|\frac{\la^2}{1+|y|^2}\left[(\la^{-1}\eta_R\phi)^2+(\eta_*\Psi_0)^2+\psi^2+(Z^*)^2+\Psi_1^2\right]+\la^3 \mathcal K[\la,\xi]\right|\\
&\lesssim \frac{\la_*^{2\nu}R^{2\sigma(4-a)}(\log R)^2}{1+|y|^{10}}\|\phi^0\|^2_{0,\sigma,\nu,a}+\frac{\la_*^{2\nu_1}}{1+|y|^{2+2a_1}}\|\phi^1\|^2_{\nu_1,a_1}+\frac{\la_*^{2\nu}}{1+|y|^{2+2a}}\|\phi^{\perp}\|^2_{\nu,a}\\
&\quad+\frac{\la_*^2(t)\la_*^{2\nu-4+\frac8n}(0)R^{-4-2\alpha+\frac{16}n}(0)|\log T|^2}{1+|y|^2}\|\psi\|_*^2+\frac{\la_*^2}{1+|y|^2}|\dot\la_*|^2|\log(T-t)|^2\\
&\quad+\frac{\la_*^2}{1+|y|^2}\|Z^*\|^2_{\infty}+\frac{\la_*^{4\beta}|\log(T-t)|^2}{1+|y|^2}+\frac{\la_*\dot\la_*}{(1+|y|^2)^2}+\frac{\la_*|\dot\xi|}{1+|y|^3}+\frac{\la_*}{\sqrt{T-t}}\frac{y_1}{1+|y|^4}\\
&\quad+|\dot\la|\left(\frac{\la_*^5}{(T-t)^2}+\frac{\la_*}{(1+|y|^2)^2}+\frac{\la_*^4(y\cdot \dot\xi+\dot\la)}{T-t}+\frac{\la_*^2\left(y\cdot \dot\xi+\dot\la\right)}{1+|y|^2}\right).
\end{aligned}
\end{equation}

\medskip

\item In the spherical coordinates, the projection of
$$\la^{-2}(t)\nabla U(y)\cdot\dot\xi(t)+\frac{n-4}{\la(t)y_1+\xi_r(t)}\la^{-2}(t)\partial_{y_1}U(y)$$
on mode $0$ is given by
\begin{equation}\label{proj-mode0}
\begin{aligned}
&\quad\int_{\mathbb S^3}\frac{n-4}{\la(t)y_1+\xi_r(t)}\la^{-2}(t)\partial_{y_1}U(y) \Theta_0 d\theta\\
&=C\frac{\la^{-2}}{(1+|y|^2)^2}\int_0^{\pi} \frac{|y|\cos\theta\sin^2\theta}{\la|y|\cos\theta+\xi_r} d\theta\\
&=-C\frac{\la^{-2}}{(1+|y|^2)^2}\frac{\la|y|^2\pi}{2(\xi_r+\sqrt{\xi_r^2-(\la|y|)^2})^2},
\end{aligned}
\end{equation}
where $C$ is a constant. Note that since our choice of $\xi_r$ is $$\xi_r\sim\sqrt{2(n-4)(T-t)},$$
the above projection on mode $0$ behaves exactly like the first error $\mathcal E_0$ defined in \eqref{def-mE0}, and direct computations show that the sum of  these two terms does not vanish. So we can deal with \eqref{proj-mode0} by slightly modifying the first correction $\Psi_0$. Here we omit the details.

\medskip

\item Similarly, the projection of
$$\la^{-2}(t)\nabla U(y)\cdot\dot\xi(t)+\frac{n-4}{\la(t)y_1+\xi_r(t)}\la^{-2}(t)\partial_{y_1}U(y)$$
on mode $1$ can be computed as
\begin{equation*}
\begin{aligned}
&\quad\int_{\mathbb S^3}\la^{-2}(t)\partial_{y_1}U(y)\left(\frac{n-4}{\la(t)y_1+\xi_r(t)}+\dot\xi_r\right) \Theta_1(\theta) d\theta\\
&=-\int_{\mathbb S^3}\la^{-2}(t)\partial_{y_1}U(y)\frac{(n-4)\la(t) y_1}{\xi_r(t)[\la(t)y_1+\xi_r(t)]} \Theta_1(\theta) d\theta\\
&=C'\frac{\la^{-1}(t)|y|^2}{\xi_r(t)(1+|y|^2)^2}\int_0^{\pi} \frac{\cos^3\theta\sin^2\theta}{\la|y|\cos\theta+\xi_r} d\theta\\
&=C'\frac{\la^{-1}(t)|y|^2}{\xi_r(t)(1+|y|^2)^2}\frac{\left((\la|y|)^4+4(\la|y|)^2\xi_r^2-8\xi_r^4+8\xi_r^3\sqrt{\xi_r^2-(\la|y|)^2}\right)\pi}{8(\la|y|)^5},
\end{aligned}
\end{equation*}
where $C'$ is a constant and we have used that $\dot\xi_r\sim-\frac{n-4}{\xi_r}$. Note that in $\mathcal D_{2R}$, namely $|y|\leq 2R$, we have $\la|y|\ll\xi_r$ for $T$ sufficiently small. Therefore, by directly expanding the above expression, we obtain
\begin{equation}\label{proj-mode1}
\int_{\mathbb S^3}\la^{-2}(t)\partial_{y_1}U(y)\left(\frac{n-4}{\la(t)y_1+\xi_r(t)}+\dot\xi_r\right) \Theta_1(\theta) d\theta\lesssim \frac{1}{\xi_r^3(1+|y|)}.
\end{equation}
\end{itemize}

Then we estimate $\mathcal H$ in three different modes.

\medskip

\noindent {\bf Estimate of $\mathcal H^0$.}

\medskip

By \eqref{est-mH1}--\eqref{proj-mode0}, we obtain
%\begin{equation*}
%\begin{aligned}
%|\mathcal H^0|\lesssim&~ \frac{\la_*}{1+|y|^4}\left[|\dot\la_*|\left(\log{\la_*}+\log(1+|y|)\right)+\la_*^{\nu-2+\frac4n}(0)R^{-2-\alpha+\frac8n}(0)|\log T|\|\psi\|_{*}+\|Z^*\|_{\infty}\right]\\
%&~+\frac{\la_*}{\sqrt{T-t}}\frac{\la_*^{\nu}R^{\sigma(4-a)}\log R}{1+|y|^5}\|\phi^0\|_{0,\sigma,\nu,a}+\la_*|\dot\la_*|\frac{\la_*^{\nu}R^{\sigma(4-a)}\log R}{1+|y|^{4}}\|\phi^0\|_{0,\sigma,\nu,a}\\
%&~+\frac{\la_*^{2\nu}R^{2\sigma(4-a)}(\log R)^2}{1+|y|^{10}}\|\phi^0\|^2_{0,\sigma,\nu,a}+\frac{\la_*^2(t)\la_*^{2\nu-4+\frac8n}(0)R^{-4-2\alpha+\frac{16}n}(0)|\log T|^2}{1+|y|^2}\|\psi\|_*^2\\
%&~+\frac{\la_*^2}{1+|y|^2}|\dot\la_*|^2|\log(T-t)|^2+\frac{\la_*^2}{1+|y|^2}\|Z^*\|^2_{\infty}+\frac{\la_*^{4\beta}|\log(T-t)|^2}{1+|y|^2}\\
%&~+\frac{\la_*\dot\la_*}{(1+|y|^2)^2}+|\dot\la_*|\Bigg(\frac{\la_*^5}{(T-t)^2}+\frac{\la_*}{(1+|y|^2)^2}+\frac{\la_*^4(y\cdot \dot\xi+\dot\la)}{T-t}+\frac{\la_*^2\left(y\cdot \dot\xi+\dot\la\right)}{1+|y|^2}\Bigg),
%\end{aligned}
%\end{equation*}
%and thus
\begin{equation*}
\begin{aligned}
\|\mathcal H^0\|_{\nu,2+a}\lesssim &~ \la_*^{1-\nu}|\dot\la_*||\log(T-t)|+\la_*^{1-\nu}(t)\la_*^{\nu-2+\frac4n}(0)R^{-2-\alpha+\frac8n}(0)|\log T|\|\psi\|_*\\
&~+\la_*^{1-\nu}\|Z^*\|_{\infty}+\la_*^{\frac12}R^{\sigma(4-a)}\log R\|\phi^0\|_{0,\sigma,\nu,a}+\la_*|\dot\la_*|R^{\sigma(4-a)}\log R\|\phi^0\|_{0,\sigma,\nu,a}\\
&~+\la_*^{\nu}R^{2\sigma(4-a)}(\log R)^2\|\phi^0\|^2_{0,\sigma,\nu,a}+\la_*^{2-\nu}(t)R^a(t)\la_*^{2\nu-4+\frac8n}(0)R^{-4-2\alpha+\frac{16}n}(0)|\log T|^2\|\psi\|^2_*\\
&~+\la_*^{2-\nu}|\dot\la_*|^2R^a |\log(T-t)|^2+\la_*^{2-\nu}R^a\|Z^*\|_{\infty}^2+\la_*^{4\beta-\nu}R^a|\log(T-t)|^2\\
&~+\la_*^{1-\nu}|\dot\la_*|+|\dot\la_*|\bigg(\frac{\la_*^{5-\nu}}{(T-t)^2} R^{2+a}+\la_*^{1-\nu}+\frac{\la_*^{4-\nu}R^{2+a}(R|\dot\xi|+\dot\la)}{T-t}\\
&~+\la_*^{2-\nu}R^a(R|\dot\xi|+\dot\la)\bigg),
\end{aligned}
\end{equation*}
from which we conclude that
\begin{equation}\label{est-mHm0}
\|\mathcal H^0\|_{\nu,2+a} \lesssim T^{\epsilon_0}\left(\|\phi^0\|_{0,\sigma,\nu,a}+\|\psi\|_*+\|Z^*\|_{\infty}+\|\la\|_{\infty}+\|\xi\|_G+1\right)
\end{equation}
provided
\begin{equation}\label{cond-mHm0}
\begin{cases}
\nu<1\\
\frac4n-1+\beta(2+\alpha-\frac8n)>0\\
\frac12-\beta\sigma(4-a)>0\\
1-\beta\sigma(4-a)>0\\
\nu-2\beta\sigma(4-a)>0\\
2-\nu-a\beta>0\\
\nu-2+\frac8n+\beta(4+2\alpha-\frac{16}n)>0\\
2-\nu-a\beta>0\\
(4-a)\beta-\nu>0\\
\frac32-\nu-\beta(1+a)>0\\
\end{cases}
\end{equation}

\medskip

\noindent {\bf Estimate of $\mathcal H^1$.}

\medskip

From \eqref{est-mH1}, \eqref{est-mH2}, \eqref{est-mH3} and \eqref{proj-mode1}, we have
%\begin{equation*}
%\begin{aligned}
%|\mathcal H^1|\lesssim &~ \frac{\la_*}{1+|y|^4}\left[|\dot\la_*|\left(\log{\la_*}+\log(1+|y|)\right)+\la_*^{\nu-2+\frac4n}(0)R^{-2-\alpha+\frac8n}(0)|\log T|\|\psi\|_{*}+\|Z^*\|_{\infty}\right]\\
%&~+\frac{\la_*^{1+\nu_1}}{1+|y|^{a_1}}\|\phi^1\|_{\nu_1,a_1}+\frac{\la_*^{1+\nu_1}|\dot\la_*|}{1+|y|^{a_1}}\|\phi^1\|_{\nu_1,a_1}+\frac{\la_*^{\frac12+\nu_1}}{1+|y|^{1+a_1}}\|\phi^1\|_{\nu_1,a_1}+\frac{\la_*^2R}{\sqrt{T-t}}\frac{|\dot\la_*|}{1+|y|^2}\\
%&~+\frac{\la_*^{2\nu_1}}{1+|y|^{2+2a_1}}\|\phi^1\|^2_{\nu_1,a_1}+\frac{\la_*^2(t)\la_*^{2\nu-4+\frac8n}(0)R^{-4-2\alpha+\frac{16}n}(0)|\log T|^2}{1+|y|^2}\|\psi\|_*^2\\
%&~+\frac{\la_*^2}{1+|y|^2}|\dot\la_*|^2|\log(T-t)|^2+\frac{\la_*^2}{1+|y|^2}\|Z^*\|^2_{\infty}+\frac{\la_*^{4\beta}|\log(T-t)|^2}{1+|y|^2}\\
%&~+\frac{\la_*^3}{(\sqrt{T-t})^3}\frac{1}{1+|y|}+|\dot\la_*|\Bigg(\frac{\la_*^5}{(T-t)^2}+\frac{\la_*}{(1+|y|^2)^2}+\frac{\la_*^4(y\cdot \dot\xi+\dot\la)}{T-t}\\
%&~+\frac{\la_*^2\left(y\cdot \dot\xi+\dot\la\right)}{1+|y|^2}\Bigg),
%\end{aligned}
%\end{equation*}
%which implies that
\begin{equation*}
\begin{aligned}
\|\mathcal H^1\|_{\nu_1,2+a_1}\lesssim &~ \la_*^{1-\nu_1}|\log(T-t)||\dot\la_*|+\la_*^{1-\nu_1}(t)\la_*^{\nu-2+\frac4n}(0)R^{-2-\alpha+\frac8n}(0)|\log T|\|\psi\|_*\\
&~+\la_*^{1-\nu_1}\|Z^*\|_{\infty}+\la_* R^2\|\phi^1\|_{\nu_1,a_1}+\la_*|\dot\la_*| R^2\|\phi^1\|_{\nu_1,a_1}+\la_*^{\frac12}R\|\phi^1\|_{\nu_1,a_1}\\
&~+\la^{\frac32-\nu_1}|\dot\la_*|R^{1+a_1}+\la_*^{\nu_1}\|\phi^1\|_{\nu_1,a_1}^2+\la_*^{2-\nu_1}|\dot\la_*|^2 R^{a_1}|\log(T-t)|^2\\
&~+\la_*^{2-\nu_1}(t)R^{a_1}(t)\la_*^{2\nu-4+\frac8n}(0)R^{-4-2\alpha+\frac{16}n}(0)|\log T|^2\|\psi\|_*^2\\
&~+\la_*^{2-\nu_1} R^{a_1}\|Z^*\|_{\infty}^2+\la_*^{4\beta-\nu_1}R^{a_1}|\log(T-t)|^2+\la_*^{\frac32-\nu_1} R^{1+a_1}\\
&~+|\dot\la_*|\bigg(\frac{\la_*^{5-\nu_1}}{(T-t)^2} R^{2+a_1}+\la_*^{1-\nu_1}+\frac{\la_*^{4-\nu_1}R^{2+a_1}(R|\dot\xi|+\dot\la)}{T-t}\\
&~+\la_*^{2-\nu_1}R^{a_1}(R|\dot\xi|+\dot\la)\bigg).\\
\end{aligned}
\end{equation*}
Therefore, we obtain that for $\epsilon_0>0$
\begin{equation}\label{est-mHm1}
\|\mathcal H^1\|_{\nu_1,2+a_1}\lesssim T^{\epsilon_0}\left(\|\phi^1\|_{\nu_1,a_1}+\|\psi\|_*+\|Z^*\|_{\infty}+\|\la\|_{\infty}+\|\xi\|_G+1\right)
\end{equation}
provided
\begin{equation}\label{cond-mHm1}
\begin{cases}
\nu_1<1\\
\nu-\nu_1-1+\frac4n+\beta(2+\alpha-\frac8n)>0\\
1-2\beta>0\\
\frac32-\nu_1-\beta(1+a_1)>0\\
\nu_1>0\\
2-\nu_1-a_1\beta>0\\
2\nu-\nu_1-2+\frac8n+\beta(4+2\alpha-a_1-\frac{16}n)>0\\
2-\nu_1-a_1\beta>0\\
(4-a_1)\beta-\nu_1>0\\
\frac32-\nu_1-\beta(1+a_1)>0\\
\end{cases}
\end{equation}

\medskip

\noindent {\bf Estimate of $\mathcal H^{\perp}$.}

\medskip

Using \eqref{est-mH1}--\eqref{est-mH3}, we get
%\begin{equation*}
%\begin{aligned}
%|\mathcal H^{\perp}|\lesssim &~\frac{\la_*}{1+|y|^4}\left[|\dot\la_*|\left(\log{\la_*}+\log(1+|y|)\right)+\la_*^{\nu-2+\frac4n}(0)R^{-2-\alpha+\frac8n}(0)|\log T|\|\psi\|_{*}+\|Z^*\|_{\infty}\right]\\
%&~+\frac{\la_*^{1+\nu}|\dot\la_*|}{1+|y|^a}\|\phi^{\perp}\|_{\nu,a}+\frac{\la_*^{2\nu}}{1+|y|^{2+2a}}\|\phi^{\perp}\|^2_{\nu,a}\\
%&~+\frac{\la_*^2(t)\la_*^{2\nu-4+\frac8n}(0)R^{-4-2\alpha+\frac{16}n}(0)|\log T|^2}{1+|y|^2}\|\psi\|_*^2+\frac{\la_*^2}{1+|y|^2}|\dot\la_*|^2|\log(T-t)|^2\\
%&~+\frac{\la_*^2}{1+|y|^2}\|Z^*\|^2_{\infty}+\frac{\la_*^{4\beta}|\log(T-t)|^2}{1+|y|^2}+|\dot\la|\Bigg(\frac{\la_*^5}{(T-t)^2}+\frac{\la_*}{(1+|y|^2)^2}\\
%&~+\frac{\la_*^4(y\cdot \dot\xi+\dot\la)}{T-t}+\frac{\la_*^2\left(y\cdot \dot\xi+\dot\la\right)}{1+|y|^2}\Bigg)
%\end{aligned}
%\end{equation*}
%and
\begin{equation*}
\begin{aligned}
\|\mathcal H^{\perp}\|_{\nu,2+a}\lesssim &~  \la_*^{1-\nu}|\dot\la_*||\log(T-t)|+\la_*^{1-\nu}(t)\la_*^{\nu-2+\frac4n}(0)R^{-2-\alpha+\frac8n}(0)|\log T|\|\psi\|_*\\
&~+\la_*^{1-\nu}\|Z^*\|_{\infty}+\la_* R^2|\dot\la_*|\|\phi^{\perp}\|_{\nu,a}+\la_*^{\nu}\|\phi^{\perp}\|^2_{\nu,a}+\la_*^{2-\nu}|\dot\la_*|^2R^a |\log(T-t)|^2\\
&~+\la_*^{2-\nu}(t)R^a(t)\la_*^{2\nu-4+\frac8n}(0)R^{-4-2\alpha+\frac{16}n}(0)|\log T|^2\|\psi\|^2_*\\
&~+\la_*^{2-\nu}R^a\|Z^*\|_{\infty}^2+\la_*^{4\beta-\nu}R^a|\log(T-t)|^2+|\dot\la_*|\bigg(\frac{\la_*^{5-\nu}}{(T-t)^2} R^{2+a}\\
&~+\la_*^{1-\nu}+\frac{\la_*^{4-\nu}R^{2+a}(R|\dot\xi|+\dot\la)}{T-t}+\la_*^{2-\nu}R^a(R|\dot\xi|+\dot\la)\bigg).
\end{aligned}
\end{equation*}
Thus, one has
\begin{equation}\label{est-mHmp}
\|\mathcal H^{\perp}\|_{\nu,2+a}\lesssim T^{\epsilon_0}\left(\|\phi^{\perp}\|_{\nu,a}+\|\psi\|_*+\|Z^*\|_{\infty}+\|\la\|_{\infty}+\|\xi\|_G+1\right)
\end{equation}
provided
\begin{equation}\label{cond-mHmp}
\begin{cases}
0<\nu<1\\
\frac4n-1+\beta(2+\alpha-\frac8n)>0\\
1-2\beta>0\\
2-\nu-a\beta>0\\
\nu-2+\frac8n+\beta(4+2\alpha-\frac{16}n)>0\\
2-\nu-a\beta>0\\
(4-a)\beta-\nu>0\\
\frac32-\nu-\beta(1+a)>0\\
\end{cases}
\end{equation}

Collecting \eqref{est-mHm0}--\eqref{cond-mHmp}, we conclude that for some $\epsilon_0>0$
\begin{equation}\label{inner-contraction}
\begin{aligned}
\|\mathcal H^0\|_{\nu,2+a}+\|\mathcal H^1\|_{\nu_1,2+a_1}+\|\mathcal H^{\perp}\|_{\nu,2+a}\lesssim&~ T^{\epsilon_0}\bigg(\|\phi^0\|_{0,\sigma,\nu,a}+\|\phi^1\|_{\nu_1,a_1}+\|\phi^{\perp}\|_{\nu,a}\\
&~+\|\psi\|_*+\|\la\|_{\infty}+\|\xi\|_{G}+\|Z^*\|_{\infty}+1\bigg)
\end{aligned}
\end{equation}
provided the parameters $\beta,a,a_1,\alpha,\nu,\nu_1,\sigma$ satisfy
\begin{equation}\label{cond-inner-contraction}
\begin{cases}
0<\nu<1\\
\frac4n-1+\beta(2+\alpha-\frac8n)>0\\
\frac12-\beta\sigma(4-a)>0\\
\nu-2\beta\sigma(4-a)>0\\
2-\nu-a\beta>0\\
\nu-2+\frac8n+\beta(4+2\alpha-\frac{16}n)>0\\
(4-a)\beta-\nu>0\\
\frac32-\nu-\beta(1+a)>0\\
0<\nu_1<1\\
\nu-\nu_1-1+\frac4n+\beta(2+\alpha-\frac8n)>0\\
1-2\beta>0\\
\frac32-\nu_1-\beta(1+a_1)>0\\
2-\nu_1-a_1\beta>0\\
2\nu-\nu_1-2+\frac8n+\beta(4+2\alpha-a_1-\frac{16}n)>0\\
2-\nu_1-a_1\beta>0\\
(4-a_1)\beta-\nu_1>0\\
\end{cases}
\end{equation}

\medskip

\subsection{The parameter problems}

\medskip

From \eqref{eqn-mode000}--\eqref{eqn-modehhh}, we need to adjust the parameter functions $\la(t)$, $\xi(t)$ such that
$$c^0[\la,\xi,\Psi^*]=0,~~c^{\ell}[\la,\xi,\Psi^*]=0,~~\ell=1,\cdots,4,$$
where
\begin{equation}\label{def-cmode0}
c^0[\la,\xi,\Psi^*]=-\frac{\int_{\mathcal D_{2R_*}}\mathcal H^0 Z_5 dy}{\int_{\mathcal D_{2R_*}}|Z_5|^2 dy}-\mathcal O[\mathcal H^0],
\end{equation}
\begin{equation}\label{def-cmode1}
c^{\ell}[\la,\xi,\Psi^*]=-\frac{\int_{\mathcal D_{2R}}\mathcal H^1 Z_{\ell} dy}{\int_{\mathcal D_{2R}}|Z_{\ell}|^2 dy}~\mbox{ for }~\ell=1,\cdots,4.
\end{equation}
It turns out that we can easily achieve at the translation mode \eqref{def-cmode1}, but the scaling mode \eqref{def-cmode0} is more complicated.

\medskip

\subsubsection{The reduced problem of \texorpdfstring{$\xi(t)$}{xi(t)}}

\medskip

We first consider the reduced equation for $\xi(t)=(\xi_r(t),\xi_z(t)).$ Notice that \eqref{def-cmode1} is equivalent to
\begin{equation*}
\int_{\mathcal D_{2R}} \mathcal H^1(\phi,\psi,\la,\xi)(y,t) Z_i(y) dy=0,~\mbox{ for all }t\in (0,T),~i=1,\cdots,4.
\end{equation*}
Recall that
$$\xi_r(t)=\sqrt{2(n-4)(T-t)}+\xi_{r,1}(t),~~\xi_z(t)=z_0+\xi_{z,1}(t)$$
and write $\Psi^*=\psi+Z^*.$ Then for $i=1,\cdots,4,$
$$\int_{\mathcal D_{2R}} \mathcal H^1(\phi,\psi,\la,\xi)(y,t) Z_i(y) dy=0$$
yield that
\begin{equation}\label{eqn-xi}
\left\{
\begin{aligned}
&\dot\xi_r+\frac{n-4}{\xi_r}=b_r[\la,\xi,\phi,\Psi^*]\\
&\dot\xi_{z_j}=b_{z_j}[\la,\xi,\phi,\Psi^*]\\
\end{aligned}
\right.
\end{equation}
where
\begin{equation*}
b_r[\la,\xi,\phi,\Psi^*]=\int_{\mathcal D_{2R}} \mathcal H_r[\la,\xi,\phi,\Psi^*](y,t)Z_1(y) dy
\end{equation*}
\begin{equation*}
b_{z_j}[\la,\xi,\phi,\Psi^*]=\int_{\mathcal D_{2R}} \mathcal H_{z_j}[\la,\xi,\phi,\Psi^*](y,t)Z_j(y) dy~\mbox{ for }~ j=2,3,4
\end{equation*}
with
\begin{equation*}
\mathcal H_r[\la,\xi,\phi,\Psi^*](y,t)=\left[\int_{\mathbb S^3} \left(\mathcal H-\la U_{y_1}\left(\dot\xi_r+\frac{n-4}{\la y_1+\xi_r}\right)\right)\Theta_1(\theta)d\theta\right]\Theta_1
\end{equation*}
and
\begin{equation*}
\mathcal H_{z_j}[\la,\xi,\phi,\Psi^*]=\sum\limits_{j=2}^4\left[\int_{\mathbb S^3} \left(\mathcal H-\la U_{y_j}\dot\xi_{z_j}\right)\Theta_j(\theta)d\theta\right]\Theta_j.
\end{equation*}
Here $\Theta_j$ $(j=1,\cdots,4)$ are the eigenfunctions corresponding to the second eigenvalue of $-\Delta_{\mathbb S^3}.$ Now we want to evaluate the sizes of $b_r[\la,\xi,\phi,\Psi^*]$ and $b_{z_j}[\la,\xi,\phi,\Psi^*]$. By direct computations, we get
\begin{equation}\label{est-br}
\begin{aligned}
|b_r[\la,\xi,\phi,\Psi^*]|\lesssim&~\left(\la_*|\dot\la_*||\log(T-t)|+\la_*\|Z^*\|_{\infty}\right)(1+O(R^{-3}))\\
&~+\la_*(t)\la_*^{\nu-2+\frac4n}(0)R^{-2-\alpha+\frac8n}(0)|\log T|\|\psi\|_*(1+O(R^{-3}))\\
&~+\la_*^{\frac12+\nu_1}\|\xi\|_G\|\phi^1\|_{\nu_1,a_1}(1+O(R^{-a_1}))+\la_*^{1+\nu_1}|\dot\la_*|\|\phi^1\|_{\nu_1,a_1}(1+O(R^{1-a_1}))\\
&~+\la_*^{\frac32}|\dot\la_*|R(1+O(R^{-1}))+\la_*^{2\nu_1}\|\phi^1\|^2_{\nu_1,a_1}(1+O(R^{-2a_1-1}))\\
&~+\la_*^{2}(t)\la_*^{2\nu-4+\frac8n}(0)R^{-4-2\alpha+\frac{16}{n}}(0)|\log T|^2\|\psi\|_*^2(1+O(R^{-1}))\\
&~+\la_*^2\|Z^*\|^2_{\infty}(1+O(R^{-1}))+\la_*^2|\dot\la_*|^2|\log(T-t)|^2(1+O(R^{-1}))\\
&~+\la_*^{4\beta}|\log(T-t)|^2 (1+O(R^{-1}))+\la_*|\dot\la_*|(1+O(R^{-3}))\\
&~+\la_*^2|\dot\la_*|(1+O(R^{-1}))+\la_*^3|\dot\la_*|R+\la_*^{\frac52}|\dot\la_*|R^2\|\xi\|_G\\
&~+\la_*^{\frac32}|\dot\la_*|R\|\xi\|_G(1+O(R^{-1}))
\end{aligned}
\end{equation}
and
\begin{equation}\label{est-bz}
\begin{aligned}
|b_{z_j}[\la,\xi,\phi,\Psi^*]|\lesssim&~\left(\la_*|\dot\la_*||\log(T-t)|+\la_*\|Z^*\|_{\infty}\right)(1+O(R^{-3}))\\
&~+\la_*(t)\la_*^{\nu-2+\frac4n}(0)R^{-2-\alpha+\frac8n}(0)|\log T|\|\psi\|_*(1+O(R^{-3}))\\
&~+\la_*^{1+\nu_1}|\dot\la_*|\|\phi^1\|_{\nu_1,a_1}(1+O(R^{1-a_1}))+\la_*^{1+\nu_1}|\dot\xi_{z_j}|\|\phi^1\|_{\nu_1,a_1}(1+O(R^{-a_1}))\\
&~+\la_*^{\frac32}|\dot\la_*|R(1+O(R^{-1}))+\la_*^{2\nu_1}\|\phi^1\|^2_{\nu_1,a_1}(1+O(R^{-2a_1-1}))\\
&~+\la_*^{2}(t)\la_*^{2\nu-4+\frac8n}(0)R^{-4-2\alpha+\frac{16}{n}}(0)|\log T|^2\|\psi\|_*^2(1+O(R^{-1}))\\
&~+\la_*^2\|Z^*\|^2_{\infty}(1+O(R^{-1}))+\la_*^2|\dot\la_*|^2|\log(T-t)|^2(1+O(R^{-1}))\\
&~+\la_*^{4\beta}|\log(T-t)|^2 (1+O(R^{-1}))+\la_*|\dot\la_*|(1+O(R^{-3}))\\
&~+\la_*^2|\dot\la_*|(1+O(R^{-1}))+\la_*^3|\dot\la_*|R+\la_*^{3+\upsilon}|\dot\la_*|R^2\|\xi\|_G\\
&~+\la_*^{2+\upsilon}|\dot\la_*|R\|\xi\|_G(1+O(R^{-1})).
\end{aligned}
\end{equation}
Since $\xi_r(t)=\sqrt{2(n-4)(T-t)}+\xi_{r,1}(t),$
problem \eqref{eqn-xi} becomes
\begin{equation}\label{eqn-xi''}
\left\{
\begin{aligned}
&\dot\xi_{r,1}-\frac{(n-4)\xi_{r,1}}{\sqrt{2(n-4)(T-t)}\left(\sqrt{2(n-4)(T-t)}+\xi_{r,1}\right)}=b_r[\la,\xi,\phi,\Psi^*],\\
&\dot\xi_{z_j}=b_{z_j}[\la,\xi,\phi,\Psi^*].\\
\end{aligned}
\right.
\end{equation}
Then we analyze the reduced problem \eqref{eqn-xi''}, which defines operators $\Xi_r$ and $\Xi_{z_j}$ $(j=2,3,4)$ that return the solutions $\xi_{r,1}$ and $\xi_{z_j}$ respectively. Here we write
\begin{equation}\label{def-XiXi}
\Xi=(\Xi_r,\Xi_{z_2},\Xi_{z_3},\Xi_{z_4}).
\end{equation}
We shall solve $(\xi_{r,1},\xi_{z,1})$ under the norm
\begin{equation*}
\|\xi\|_G=\sup_{t\in(0,T)} \left[(T-t)^{-\frac12-\upsilon}|\xi_{r,1}(t)|+M_1(T-t)^{\frac12-\upsilon}|\dot\xi_{r,1}(t)|+|\xi_{z_j}(t)|+(T-t)^{-\upsilon}|\dot\xi_{z_j}(t)|\right]
\end{equation*}
for $\upsilon>0$ and $0<M_1<1$. From \eqref{eqn-xi''}, we have
\begin{equation*}
|\xi_{r,1}(t)|\leq \left(\frac{(T-t)^{\frac12+\upsilon}\|\xi\|_G}{2(T-t)}+\left\|b_r[\la,\xi,\phi,\Psi^*]\right\|_{L^{\infty}(0,T)}\right)(T-t)
\end{equation*}
and
\begin{equation*}
|\xi_{z_j}(t)|\leq |z_0|+\left\|b_{z_j}[\la,\xi,\phi,\Psi^*]\right\|_{L^{\infty}(0,T)}(T-t).
\end{equation*}
Therefore, we obtain
\begin{equation}\label{est-Xir''}
\|\Xi_r\|_G\leq \frac{1+M_1}{2}\|\xi\|_G+(1+M_1)(T-t)^{\frac12-\upsilon}\left\|b_r[\la,\xi,\phi,\Psi^*]\right\|_{L^{\infty}(0,T)}
\end{equation}
and
\begin{equation}\label{est-Xiz''}
\|\Xi_{z_j}\|_G\leq |z_0|+(T-t)^{-\upsilon}\left\|b_{z_j}[\la,\xi,\phi,\Psi^*]\right\|_{L^{\infty}(0,T)}.
\end{equation}
By \eqref{est-br}, \eqref{est-bz}, \eqref{est-Xir''} and \eqref{est-Xiz''}, we conclude that for some constant $C>0$
\begin{equation}\label{est-Xir}
\begin{aligned}
\|\Xi_r\|_G\leq&~\left[\frac{1+M_1}{2}+C(1+M_1)(T-t)^{\frac12-\upsilon}\left(\la_*^{\frac12+\nu_1}+\la_*^{\frac32}|\dot\la_*|R\right)\right]\|\xi\|_G\\
&~+C(1+M_1)(T-t)^{\frac12-\upsilon}\bigg[\la_*(t)\la_*^{\nu-2+\frac4n}(0)R^{-2-\alpha+\frac8n}\|\psi\|_*+\la_*\|Z^*\|_{\infty}\\
&~+(\la_*^{\frac12+\nu_1}+\la_*^{2\nu_1})\|\phi^1\|_{\nu_1,a_1}+\la_*\|\la\|_{\infty}+\la_*\bigg]\\
\end{aligned}
\end{equation}
and
\begin{equation}\label{est-Xiz}
\begin{aligned}
\|\Xi_{z_j}\|_G\leq&~ |z_0|+C(T-t)^{-\upsilon}\Bigg[\la_*(t)\la_*^{\nu-2+\frac4n}(0)R^{-2-\alpha+\frac8n}\|\psi\|_*+\la_*\|Z^*\|_{\infty}\\
&~+(\la_*^{1+\nu_1}+\la_*^{2\nu_1})\|\phi^1\|_{\nu_1,a_1}+\la_*\|\la\|_{\infty}+\la_*+\left(\la_*^{1+\nu_1+\upsilon}+\la_*^{2+\upsilon}|\dot\la_*|R\right)\|\xi\|_G\Bigg].\\
\end{aligned}
\end{equation}

\medskip

\subsubsection{The reduced problem of \texorpdfstring{$\lambda(t)$}{lambda(t)}}\label{subsec-la}

\medskip

Since the reduced problem of $\la$ is essentially the same as that of \cite{17HMF}, we shall follow the strategy and logic in \cite[Section 8]{17HMF}.

From direct computations, we see that \eqref{def-cmode0}
gives a non-local integro-differential equation
\begin{equation}\label{indi-la}
\begin{aligned}
\int_{-T}^t \frac{\dot{\la}(s)}{t-s}\Gamma\left(\frac{\la^2(t)}{t-s}\right) ds+{\bf c_0} \dot\la=a[\la,\xi,\Psi^*](t)+\mathtt a_r[\la,\xi,\phi,\Psi^*](t),
\end{aligned}
\end{equation}
where
$${\bf c_0}=2\alpha_0\int_{\R^4}\frac{Z_5(y)}{(1+|y|^2)^2}dy,$$
\begin{equation}\label{def-aaa}
a[\la,\xi,\Psi^*]=-\int_{\mathcal D_{2R_*}} 3U^2(y)\left(\Psi_0+\Psi^*\right)Z_5(y) dy,
\end{equation}
and the remainder term $\mathtt a_r[\la,\xi,\phi,\Psi^*](t)$ turns out to be smaller order and has the following bound
\begin{equation*}
\begin{aligned}
\left|\mathtt a_r[\la,\xi,\phi,\Psi^*](t)\right|\lesssim&~\left[\la_*^{\nu}R^{\sigma(4-a)}(1+|\dot\la_*|)+\la_*^{\nu-\frac12}R^{\sigma(4-a)}\|\xi\|_G\right]\log R\|\phi^0\|_{0,\sigma,\nu,a}(1+O(R_*^{-2}))\\
&~+\la_*^{\frac12} R |\dot\la_*||\log(T-t)|+\la_*^{2\nu-1}R^{2\sigma(4-a)}(\log R)^2\|\phi^0\|_{0,\sigma,\nu,a}^2(1+O(R_*^{-8}))\\
&~+\la_*(t)\la_*^{2\nu-4+\frac8n}(0)R^{-4-2\alpha+\frac{16}{n}}(0)|\log T|^2|\log(T-t)|\|\psi\|_*^2\\
&~+\la_*|\dot\la_*|^2|\log(T-t)|^3+\la_*|\log(T-t)|\|Z^*\|_{\infty}^2+\la_*^{4\beta-1}|\log(T-t)|^3.
\end{aligned}
\end{equation*}
We first introduce the following norms
\begin{equation*}
\|f\|_{\Theta,l}:=\sup_{t\in[0,T]}\frac{|\log(T-t)|^l}{(T-t)^{\Theta}}|f(t)|,
\end{equation*}
where $f\in C([-T,T];\R)$ with $f(T)=0$, and $\Theta\in(0,1)$, $l\in\R$.
\begin{equation*}
[g]_{\gamma,m,l}:=\sup_{I_T} \frac{|\log(T-t)|^l}{(T-t)^m (t-s)^{\gamma}}|g(t)-g(s)|,
\end{equation*}
where $I_T=\left\{0\leq s\leq t\leq T: t-s\leq\frac{1}{10}(T-t)\right\},$ $g\in C([-T,T];\R)$ with $g(T)=0$ and $0<\gamma<1$, $m>0$, $l\in\R$. Also, we define
\begin{equation}\label{def-mB0}
\mathcal B_0[\la](t):=\int_{-T}^t \frac{\dot{\la}(s)}{t-s}\Gamma\left(\frac{\la^2(t)}{t-s}\right) ds+{\bf c_0} \dot\la
\end{equation}
and write
\begin{equation}\label{def-c^0}
c^0[\mathcal H]=\frac{\mathcal B_0[\la]-(a[\la,\xi,\Psi^*]+\mathtt a_r[\la,\xi,\phi,\Psi^*])}{\int_{\mathcal D_{2R_*}}|Z_5(y)|^2dy}.
\end{equation}

A key proposition concerning the solvability of $\la$ is stated as follows.

\begin{prop}[\cite{17HMF}]\label{keyprop-la}
Let $\omega,\Theta\in(0,\frac12)$, $\gamma\in(0,1)$, $m\leq \Theta-\gamma$ and $l\in\R$. If $a(t)$ satisfies $a(T)<0$ with $1/C\leq a(T)\leq C$ for some constant $C>1$, and
\begin{equation}\label{assump-a}
T^{\Theta}|\log T|^{1+c-l}\|a(\cdot)-a(T)\|_{\Theta,l-1}+[a]_{\gamma,m,l-1}\leq C_1
\end{equation}
for some $c>0$, then there exist two operators $\mathcal P$ and $\mathcal R_0$ such that $\la=\mathcal P[a]:[-T,T]\rightarrow \R$ satisfies
\begin{equation}\label{eqn-mB0}
\mathcal B_0[\la](t)=a(t)+\mathcal R_0[a](t)
\end{equation}
with
\begin{equation*}
\begin{aligned}
|\mathcal R_0[a](t)|\lesssim \left(T^{\frac12+c}+T^{\Theta}\frac{\log|\log T|}{|\log T|}\|a(\cdot)-a(T)\|_{\Theta,l-1}+[a]_{\gamma,m,l-1}\right)\frac{(T-t)^{m+(1+\omega)\gamma}}{|\log(T-t)|^l}.
\end{aligned}
\end{equation*}
\end{prop}
The proof of Proposition \ref{keyprop-la} is in \cite{17HMF}.  The idea of the proof is to observe that
$$\mathcal B_0[\la]\approx \int_{-T}^{t-\la_*^2(t)}\frac{\dot\la(s)}{t-s}ds,$$
and we decompose
$$\mathcal B_0[\la]=\mathcal B_0^*[\la]+\mathcal S_{\omega}[\dot\la]+\mathcal R_{\omega}[\dot\la],$$
where
\begin{equation*}
\mathcal B_0^*[\la]:=\mathcal B_0[\la]-\int_{-T}^{t-\la_*^2(t)}\frac{\dot\la(s)}{t-s} ds,
\end{equation*}
\begin{equation*}
\mathcal S_{\omega}[\dot\la]:=\dot\la\left[(1+\omega)\log(T-t)-2\log\la_*(t)\right]+\int_{-T}^{t-(T-t)^{1+\omega}}\frac{\dot\la(s)}{t-s}ds
\end{equation*}
and
\begin{equation}\label{def-mRw}
\mathcal R_{\omega}[\dot\la]:=-\int_{t-(T-t)^{1+\omega}}^{t-\la_*^2(t)}\frac{\dot\la(t)-\dot\la(s)}{t-s} ds.
\end{equation}
Here $\omega>0$ is a fixed number. We solve a modified equation where we drop $\mathcal R_{\omega}[\dot\la]$ in \eqref{eqn-mB0}, and thus the remainder $\mathcal R_0$ is essentially $\mathcal R_{\omega}[\dot\la]$ and $\mathtt a_r[\la,\xi,\phi,\Psi^*]$.

In another aspect, we modify problem \eqref{indi-la} replacing $a[\la,\xi,\Psi^*]$ by its main term. To this end, we define
$$a[\la,\xi,\Psi^*]=a^0[\la,\xi,\Psi^*]+a^1[\la,\xi,\Psi^*]+a^{\perp}[\la,\xi,\Psi^*]$$
with
$$a^0[\la,\xi,\Psi^*]=-\int_{\mathcal D_{2R_*}} \tilde L^0[\Psi] Z_5(y) dy,$$
$$a^1[\la,\xi,\Psi^*]=-\int_{\mathcal D_{2R_*}} \tilde L^1[\Psi] Z_5(y) dy,$$
and
$$a^{\perp}[\la,\xi,\Psi^*]=-\int_{\mathcal D_{2R_*}} \tilde L^{\perp}[\Psi] Z_5(y) dy,$$
where $\tilde L[\Psi]:=3U^2(\Psi_0+\Psi^*)$, $\tilde L^0[\Psi]$ is the projection of $\tilde L[\Psi]$ on mode $0$, $\tilde L^1[\Psi]$ is the projection of $\tilde L[\Psi]$ on modes $1$ to $4$, and $\tilde L^{\perp}[\Psi]$ is the projection of $\tilde L[\Psi]$ on higher modes $j\geq 5$.

We define
\begin{equation}\label{def-c_*^0}
\begin{aligned}
c_*^0[\la,\xi,\Psi^*](t):=&~\frac{\mathcal R_0\left[a^0[\la,\xi,\Psi^*]\right](t)+a^1[\la,\xi,\Psi^*](t)+a^{\perp}[\la,\xi,\Psi^*](t)}{\int_{\mathcal D_{2R_*}}|Z_5(y)|^2dy}\\
&~-\left(c^0[\mathcal H[\la,\xi,\Psi^*]]-\tilde c^0[\mathcal H^0[\la,\xi,\Psi^*]]\right),
\end{aligned}
\end{equation}
where $\mathcal R_0$ is the operator given in Proposition \ref{keyprop-la}, $c^0$ is defined in \eqref{def-c^0}, and $\tilde c^0$ is the operator given in Proposition \ref{propmode0}. The reason for choosing such $c_*^0$ is the following. By Proposition \ref{keyprop-la}, the equation we solve is
$$\mathcal B_0[\la](t)=a^0[\la,\xi,\Psi^*](t)+\mathcal R_0[a^0[\la,\xi,\Psi^*]]$$
which is equivalent to
$$c^0[\mathcal H]=\frac{\mathcal R_0\left[a^0[\la,\xi,\Psi^*]\right](t)+a^1[\la,\xi,\Psi^*](t)+a^{\perp}[\la,\xi,\Psi^*](t)}{\int_{\mathcal D_{2R_*}}|Z_5(y)|^2dy}.$$
We shall consider the following reduced equation
$$\tilde c^0[\mathcal H^0]=c_*^0[\la,\xi,\Psi^*],$$
from which we get \eqref{def-c_*^0}.

By \eqref{def-aaa}, Proposition \ref{keyprop-la} and Proposition \ref{outer-apriori}, it is natural to choose
\begin{equation*}
\Theta=\nu-2+\frac4n+\beta(2+\alpha-\frac8n)
\end{equation*}
and
\begin{equation*}
m=\nu-2-\gamma+\beta(2+\alpha).
\end{equation*}
In order for $\|a(\cdot)-a(T)\|_{\Theta,l-1}$ and $[a]_{\gamma,m,l-1}$ to be finite, we require
$$l<\max\{1+2\Theta,1+2m\}$$
and it then follows that
$$\|a(\cdot)-a(T)\|_{\Theta,l-1}\lesssim |\log T|^{l-\Theta-1},\quad [a]_{\gamma,m,l-1}\lesssim |\log T|^{l-m-1}.$$
Another assumption $m<\Theta-\gamma$ in Proposition \ref{keyprop-la} is valid since $\beta<1/2$.
Finally, in order to make the remainder $\mathcal R_0[a]$ small, we impose
\begin{equation*}
m+(1+\omega)\gamma>\Theta,
\end{equation*}
which implies that
\begin{equation*}
\beta>\frac{1}{2}-\frac{\omega\gamma n}{8}.
\end{equation*}

\medskip

%%%%%%%%%%%%%%%%%%%%%%%%%%%%%%%%%%%%%%%%%%%%%%%%%%%%%%%%%%%%

\subsection{Inner--outer gluing system}

\medskip

By the discussions in Section \ref{subsec-la}, we transform the inner--outer problems \eqref{inner}, \eqref{outer} into the problems of finding solutions $(\psi,\phi^0,\phi^1,\phi^{\perp},\la,\xi)$ solving the following {\em inner--outer gluing system}
\begin{equation}\label{eqn-outer}
\begin{cases}
\psi_t=\Delta_{(r,z)} \psi+\frac{n-4}{r}\partial_r \psi +\mathcal G(\phi^0+\phi^1+\phi^{\perp},\psi+Z^*,\la,\xi),~&\mbox{ in }\mathcal D\times (0,T)\\
\psi=-\Psi^0,~&\mbox{ on }(\partial \mathcal D\backslash \{r=0\})\times (0,T)\\
\psi_r=0,~&\mbox{ on } (\mathcal D\cap\{r=0\})\times (0,T)\\
\psi(r,z,0)=-(1-\eta_*)\Psi^0,~&\mbox{ in }\mathcal D\\
\end{cases}
\end{equation}
\begin{equation}\label{eqn-projm0}
\begin{cases}
\la^2 \phi^0_t = \Delta_y \phi^0 +3U^2(y)\phi^0+\mathcal H^0(\phi,\psi,\la,\xi)+\tilde c^0[\mathcal H^0] Z_5~~&\mbox{ in }\mathcal D_{2R}\times (0,T)\\
\phi^0(\cdot,0)=0~~&\mbox{ in }\mathcal D_{2R}\\
\end{cases}
\end{equation}

\begin{equation}\label{eqn-projm1}
\begin{cases}
\la^2 \phi^1_t = \Delta_y \phi^1 +3U^2(y)\phi^1+\mathcal H^1(\phi,\psi,\la,\xi)+\sum\limits_{\ell=1}^4 c^{\ell}[\mathcal H^1]Z_{\ell}~~&\mbox{ in }\mathcal D_{2R}\times (0,T)\\
\phi^1(\cdot,0)=0~~&\mbox{ in }\mathcal D_{2R}\\
\end{cases}
\end{equation}

\begin{equation}\label{eqn-projmh}
\begin{cases}
\la^2 \phi^{\perp}_t = \Delta_y \phi^{\perp} +3U^2(y)\phi^{\perp}+\mathcal H^{\perp}(\phi,\psi,\la,\xi)+c^0_*[\la,\xi,\Psi^*]Z_5~~&\mbox{ in }\mathcal D_{2R}\times (0,T)\\
\phi^{\perp}(\cdot,0)=0~~&\mbox{ in }\mathcal D_{2R}\\
\end{cases}
\end{equation}
\begin{eqnarray}
c^0[\mathcal H](t)-\tilde c^0[\la,\xi,\Psi^*](t)=0~\mbox{ for all }~t\in(0,T)\label{redu-la}\\
c^1[\mathcal H](t)=0~\mbox{ for all }~t\in(0,T)\label{redu-xi}
\end{eqnarray}
where $\eta_*$ is defined in \eqref{def-eta*}, $\mathcal G$ is defined in \eqref{def-mG}, $\mathcal H^0$, $\mathcal H^1$, $\mathcal H^{\perp}$ are the projections on different modes defined in \eqref{def-mHm0}--\eqref{def-mHmh}.

It is direct to see that if $(\psi,\phi^0,\phi^1,\phi^{\perp},\la,\xi)$ satisfies the system \eqref{eqn-outer}--\eqref{redu-xi}, then
$$\Psi^*=\psi+Z^*,~~\phi=\phi^0+\phi^1+\phi^{\perp}$$
solve the inner--outer problems \eqref{inner}, \eqref{outer}, and thus the desired blow-up solution is found.

\medskip

%%%%%%%%%%%%%%%%%%%%%%%%%%%%%%%%%%%%%%%%%%%%%%%%%%%%%%%%%%%%

\subsection{The fixed point formulation}

\medskip

The inner--outer gluing system \eqref{eqn-outer}--\eqref{redu-xi} can be formulated as a fixed point problem for operators we shall describe below.

We first define the following function spaces
\begin{equation}\label{def-fcnspaces}
\begin{aligned}
&X_{\phi^0}:=\left\{\phi^0\in L^{\infty}(\mathcal D_{2R}\times (0,T)):~\nabla_y \phi^0\in L^{\infty}(\mathcal D_{2R}\times (0,T)),~\|\phi^0\|_{0,\sigma,\nu,a}<+\infty\right\},\\
&X_{\phi^1}:=\left\{\phi^1\in L^{\infty}(\mathcal D_{2R}\times (0,T)):~\nabla_y \phi^1\in L^{\infty}(\mathcal D_{2R}\times (0,T)),~\|\phi^1\|_{\nu_1,a_1}<+\infty\right\},\\
&X_{\phi^{\perp}}:=\left\{\phi^{\perp}\in L^{\infty}(\mathcal D_{2R}\times (0,T)):~\nabla_y \phi^{\perp}\in L^{\infty}(\mathcal D_{2R}\times (0,T)),~\|\phi^{\perp}\|_{\nu,a}<+\infty\right\},\\
&X_{\psi}:=\big\{\psi\in L^{\infty}(\mathcal D\times (0,T)):~\|\psi\|_{*}<+\infty,~\mbox{ $\psi$ is Lipschitz continuous with respect}\\
&\qquad\qquad\mbox{ to $(r,z)$ in $\mathcal D\times (0,T)$}\big\}.\\
\end{aligned}
\end{equation}

In order to introduce the space for the parameter function $\la(t)$, we recall from \eqref{def-mB0} that the integral operator $\mathcal B_0$ takes the following approximate form
$$\mathcal B_0[\la]=\int_{-T}^{t-\la_*^2(t)}\frac{\dot\la(s)}{t-s} ds+O(\|\dot\la\|_{\infty}).$$
Proposition \ref{keyprop-la} provides an approximate inverse operator $\mathcal P$ of the integral operator $\mathcal B_0$ such that for $a(t)$ satisfying \eqref{assump-a}, $\la:=\mathcal P[a]$ satisfies
$$\mathcal B_0[\la]=a+\mathcal R_0[a]~\mbox{ in }~[-T,T],$$
where $\mathcal R_0[a]$ is a small remainder. Also, the proof  in \cite{17HMF} gives the following decomposition
\begin{equation}\label{def-mP1}
\mathcal P[a]=\la_{0,\kappa}+\mathcal P_1[a]
\end{equation}
with
\begin{equation*}
\la_{0,\kappa}:=\kappa|\log T|\int_t^T \frac{1}{|\log(T-s)|^2}ds,~t\leq T
\end{equation*}
$\kappa=\kappa[a]\in\R$, and the function $\la_1=\mathcal P_1[a]$ satisfies
\begin{equation}\label{est-la1la1}
\|\la_1\|_{*,3-\iota}\lesssim |\log T|^{1-\iota}\log^2(|\log T|)
\end{equation}
for $0<\iota<1,$ where the $\|\cdot\|_{*,3-\iota}$-norm is defined as follows
\begin{equation*}
\|f\|_{*,k}:=\sup_{t\in[-T,T]}|\log(T-t)|^k |\dot f(t)|.
\end{equation*}
So we define
$$X_{\la}:=\{\la_1\in C^1([-T,T]):\la_1(T)=0,\|\la_1\|_{*,3-\iota}<\infty\}.$$
Here by $(\kappa,\la_1)$, we represent $\la$ in the form
$$\la=\la_{0,\kappa}+\la_1,$$
and from \cite{17HMF}, one can write the norm
\begin{equation}\label{def-normla}
\|\la\|_F=|\kappa|+\|\la_1\|_{*,3-\iota}.
\end{equation}

Recall that $\xi(t)=(\xi_r(t),\xi_z(t))$ with $\xi_r(t)=\sqrt{2(n-4)(T-t)}+\xi_{r,1}(t)$, $\xi_z(t)=z_0+\xi_{z,1}(t)$ and write $\xi(t)=\xi_*(t)+\xi_1(t)$. We define the following space for $(\xi_{r,1},\xi_{z,1})$
\begin{equation*}
X_{\xi}=\left\{\xi\in C^1((0,T);\R^4),~\dot\xi(T)=0,~\|\xi\|_{G}<+\infty\right\}
\end{equation*}
with
\begin{equation}\label{def-normxi}
\|\xi\|_G:=\sup_{t\in(0,T)} \left[(T-t)^{-\frac12-\upsilon}|\xi_{r,1}(t)|+M_1(T-t)^{\frac12-\upsilon}|\dot\xi_{r,1}(t)|+|\xi_{z_j}(t)|+(T-t)^{-\upsilon}|\dot\xi_{z_j}(t)|\right]
\end{equation}
for some $0<\upsilon<1$.

Define
\begin{equation}\label{def-mX}
\mathcal X=X_{\phi^0}\times X_{\phi^1}\times X_{\phi^{\perp}}\times X_{\psi}\times\R\times X_{\la}\times X_{\xi}.
\end{equation}
We shall solve the inner--outer gluing system in a closed ball $\mathcal B$ in which $(\phi^0,\phi^1,\phi^{\perp},\psi,\kappa,\la_1,\xi_1)\in\mathcal X$ satisfies
\begin{equation}\label{def-closedball}
\left\{
\begin{aligned}
&\|\phi^0\|_{0,\sigma,\nu,a}+\|\phi^1\|_{\nu_1,a_1}+\|\phi^{\perp}\|_{\nu,a}\leq 1\\
&\|\psi\|_*\leq 1\\
&|\kappa-\kappa_0|\leq |\log T|^{-1/2}\\
&\|\la_1\|_{*,3-\iota}\leq C|\log T|^{1-\iota}\log^2(|\log T|)\\
&\|\xi\|_G\leq 1\\
\end{aligned}
\right.
\end{equation}
for some large and fixed constant $C$, where $\kappa_0=Z^*_0(0)$.

The inner--outer gluing system \eqref{eqn-outer}--\eqref{redu-xi} can be formulated as a fixed point problem, where we define an operator $\mathcal F$ which returns the solution from $\mathcal B$ to $\mathcal X$
\begin{equation*}
\begin{aligned}
\mathcal F: \mathcal B\subset \mathcal X~\rightarrow& ~\mathcal X\\
v~\mapsto&~ \mathcal F(v)=(\mathcal F_{\phi^0}(v),\mathcal F_{\phi^1}(v),\mathcal F_{\phi^{\perp}}(v),\mathcal F_{\psi}(v),\mathcal F_{\kappa}(v),\mathcal F_{\la_1}(v),\mathcal F_{\xi}(v))
\end{aligned}
\end{equation*}
with
\begin{equation}\label{def-operators}
\begin{aligned}
\mathcal F_{\phi^0}(\phi^0,\phi^1,\phi^{\perp},\psi,\kappa,\la_1,\xi_1)=&~\mathcal T_0(\mathcal H^0[\la,\xi,\Psi^*]),\\
\mathcal F_{\phi^1}(\phi^0,\phi^1,\phi^{\perp},\psi,\kappa,\la_1,\xi_1)=&~\mathcal T_1(\mathcal H^1[\la,\xi,\Psi^*]),\\
\mathcal F_{\phi^{\perp}}(\phi^0,\phi^1,\phi^{\perp},\psi,\kappa,\la_1,\xi_1)=&~\mathcal T_{\perp}\left(\mathcal H^{\perp}[\la,\xi,\Psi^*]+c^0_*[\la,\xi,\Psi^*]Z_5\right),\\
\mathcal F_{\psi}(\phi^0,\phi^1,\phi^{\perp},\psi,\kappa,\la_1,\xi_1)=&~\mathcal T_{\psi}\left(\mathcal G(\phi^0+\phi^1+\phi^{\perp},\Psi^*,\la,\xi)\right),\\
\mathcal F_{\kappa}(\phi^0,\phi^1,\phi^{\perp},\psi,\kappa,\la_1,\xi_1)=&~\kappa\left[a^0[\la,\xi,\Psi^*]\right],\\
\mathcal F_{\la_1}(\phi^0,\phi^1,\phi^{\perp},\psi,\kappa,\la_1,\xi_1)=&~\mathcal P_1\left[a^0[\la,\xi,\Psi^*]\right],\\
\mathcal F_{\xi}(\phi^0,\phi^1,\phi^{\perp},\psi,\kappa,\la_1,\xi_1)=&~\Xi(\phi^0,\phi^1,\phi^{\perp},\psi,\la,\xi).\\
\end{aligned}
\end{equation}
Here $\mathcal T_0$, $\mathcal T_1$ and $\mathcal T_{\perp}$ are the operators given from Proposition \ref{lineartheory} which solve different modes of the inner problems \eqref{eqn-projm0}--\eqref{eqn-projmh}. The operator $\mathcal T_{\psi}$ defined by Proposition \ref{outer-apriori} deals with the outer problem \eqref{eqn-outer}. Operators $\kappa[a]$, $\mathcal P_1$ and $\Xi$ handle the equations for $\la$ and $\xi$ which are defined in Proposition \ref{keyprop-la}, \eqref{def-mP1} and \eqref{def-XiXi}.

\medskip

%%%%%%%%%%%%%%%%%%%%%%%%%%%%%%%%%%%%%%%%%%%%%%%%%%%%%%%%%%%%

\medskip

\subsection{Choice of constants}\label{subsec-choices}

In this section, we list all the constraints of the parameters
$$\beta,\alpha,a,a_1,\nu,\nu_1,\nu_2,\sigma$$
which are sufficient for the inner--outer gluing scheme to work.

We first indicate all the parameters used in different norms.
\begin{itemize}
\item $R(t)=\la_*^{-\beta}(t)$ with $\beta\in(0,1/2)$.

\medskip

\item The norm for $\phi^0$ solving mode $0$ of the inner problem \eqref{eqn-projm0} is $\|\cdot\|_{0,\sigma,\nu,a}$ which is defined in \eqref{def-normm0}, where we require that $\nu,a\in(0,1)$, and $\sigma>0$ is fixed and sufficiently small.

\medskip

\item The norm for $\phi^1$ solving modes $1$ to $4$ of the inner problem \eqref{eqn-projm1} is $\|\cdot\|_{\nu_1,a_1}$ which is defined in \eqref{def-normnua}, where we require that $\nu_1\in(0,1)$ and $a_1\in(1,2)$.

\medskip

\item The norm for $\phi^{\perp}$ solving higher modes $(j\geq 5)$ of the inner problem \eqref{eqn-projmh} is $\|\cdot\|_{\nu,a}$ which is defined in \eqref{def-normnua}, where $\nu,a\in(0,1)$.

\medskip

\item The norm for $\psi$ solving the outer problem \eqref{eqn-outer} is $\|\cdot\|_*$ which is defined in \eqref{def-norm*}, while the $\|\cdot\|_{**}$-norm for the right hand side of the outer problem \eqref{eqn-outer} is defined in \eqref{def-norm**}. Here we require that $\nu,\alpha,\nu_2\in(0,1)$ and $\gamma\in(0,1)$.

\medskip

\item In Proposition \ref{keyprop-la}, we have the parameters $\omega,\Theta,m,l,\gamma$. Here $\omega$ is the parameter used to describe the remainder $\mathcal R_{\omega}$ in \eqref{def-mRw} and $\omega\in(0,1/2)$. To apply Proposition \ref{keyprop-la} in our setting, we let
$$\Theta=\nu-2+\frac4n+\beta(2+\alpha-\frac8n),$$
$$m=\nu-2-\gamma+\beta(2+\alpha),\quad l<1+2m,$$
and require that
$$\beta>\frac{1}{2}-\frac{\omega\gamma n}{8}$$
such that
$$m+(1+\omega)\gamma>\Theta$$
is guaranteed. Also, we need
$$\nu-2+\frac4n+\beta(2+\alpha-\frac8n)>0$$
to ensure that $\Theta>0$.
\end{itemize}

In order to get the desired estimates for the outer problem \eqref{eqn-outer}, by the computations in Section \ref{subsec-outer}, we need the following restrictions
\begin{equation*}
\begin{cases}
\nu-2+\frac4n+\beta(4+\alpha-\frac8n)-\nu_2>0,\\
2\beta-\nu_2>0,\\
4\beta-\nu_2-1>0,\\
0<\alpha<a<1,\\
\beta<\frac12,\\
\beta(1+\alpha-a_1)+\nu-\nu_1-\frac12<0,\\
\nu-\beta\left(2\sigma(4-a)+\alpha-8\right)>0,\\
2\nu_1-\nu-\beta(\alpha-2a_1)>0,\\
\nu-\beta(\alpha-2a)>0,\\
\nu_2<1,\\
2\nu-3+\frac8n-\nu_2+\beta(4+2\alpha-\frac{16}n)>0.\\
\end{cases}
\end{equation*}
In order to get the desired estimates for the inner problems at different modes \eqref{eqn-projm0}--\eqref{eqn-projmh}, by the computations in Section \ref{subsec-inner}, we need
\begin{equation*}
\begin{cases}
0<\nu<1,\\
\frac4n-1+\beta(2+\alpha-\frac8n)>0,\\
\frac12-\beta\sigma(4-a)>0,\\
\nu-2\beta\sigma(4-a)>0,\\
2-\nu-a\beta>0,\\
\nu-2+\frac8n+\beta(4+2\alpha-\frac{16}n)>0,\\
(4-a)\beta-\nu>0,\\
\frac32-\nu-\beta(1+a)>0,\\
0<\nu_1<1,\\
\nu-\nu_1-1+\frac4n+\beta(2+\alpha-\frac8n)>0,\\
1-2\beta>0,\\
\frac32-\nu_1-\beta(1+a_1)>0,\\
2-\nu_1-a_1\beta>0,\\
2\nu-\nu_1-2+\frac8n+\beta(4+2\alpha-a_1-\frac{16}n)>0,\\
2-\nu_1-a_1\beta>0,\\
(4-a_1)\beta-\nu_1>0.\\
\end{cases}
\end{equation*}
It turns out that suitable choices of the parameters satisfying all the restrictions in this section can be found for the space dimensions $n=5,6,7$. Here we give specific example for each case. Sound choices are listed as follows.
\begin{itemize}
\item $n=5:~\beta=\frac12-\epsilon,~\alpha=4\epsilon,~a=\frac{9}{2}\epsilon,~a_1=2-2\epsilon,~\nu=1-\frac32\epsilon,~\nu_1=3\epsilon,~\nu_2<1,~\mbox{ $\sigma>0$ small}$

\medskip

\item $n= 6:~\beta=\frac12-\epsilon,~\alpha=4\epsilon,~a=\frac{13}{3}\epsilon,~a_1=2-2\epsilon,~\nu=1-\frac54\epsilon,~\nu_1=\frac72\epsilon,~\nu_2<1,~\mbox{ $\sigma>0$ small}$

\medskip

\item $n= 7:~\beta=\frac12-\epsilon,~\alpha=4\epsilon,~a=\frac{57}{14}\epsilon,~a_1=2-2\epsilon,~\nu=1-\frac{15}{14}\epsilon,~\nu_1=3\epsilon,~\nu_2<1,~\mbox{ $\sigma>0$ small}$
\end{itemize}
where $\epsilon>0$ is fixed and sufficiently small.

\medskip

%%%%%%%%%%%%%%%%%%%%%%%%%%%%%%%%%%%%%%%%%%%%%%%%%%%%%%%%%%%%

\subsection{Proof of Theorem \texorpdfstring{\ref{thm}}{1.1}}

\medskip

Consider the operator
\begin{equation}\label{def-mF}
\mathcal F=(\mathcal F_{\phi^0},\mathcal F_{\phi^1},\mathcal F_{\phi^{\perp}},\mathcal F_{\psi},\mathcal F_{\kappa},\mathcal F_{\la_1},\mathcal F_{\xi})
\end{equation}
given in \eqref{def-operators}.

To prove Theorem \ref{thm}, our strategy is to show that the operator $\mathcal F$ has a fixed point in $\mathcal B$ by the Schauder fixed point theorem. Here the closed ball $\mathcal B$ is defined in \eqref{def-closedball}. By collecting the estimates \eqref{outer-contraction}, \eqref{est-mHm0}, \eqref{est-mHm1}, \eqref{est-mHmp}, \eqref{est-Xir}, \eqref{est-Xiz}, \eqref{est-la1la1}, and using Proposition \ref{outer-apriori}, Proposition \ref{lineartheory}, Proposition \ref{keyprop-la}, we conclude that for $(\phi^0,\phi^1,\phi^{\perp},\psi,\kappa,\la_1,\xi_1)\in\mathcal B$
\begin{equation}\label{contraction-mF}
\left\{
\begin{aligned}
&\|\mathcal F_{\phi^0}(\phi^0,\phi^1,\phi^{\perp},\psi,\kappa,\la_1,\xi_1)\|_{0,\sigma,\nu,a}\leq CT^{\epsilon}\\
&\|\mathcal F_{\phi^1}(\phi^0,\phi^1,\phi^{\perp},\psi,\kappa,\la_1,\xi_1)\|_{\nu_1,a_1}\leq CT^{\epsilon}\\
&\|\mathcal F_{\phi^{\perp}}(\phi^0,\phi^1,\phi^{\perp},\psi,\kappa,\la_1,\xi_1)\|_{\nu,a}\leq CT^{\epsilon}\\
&\|\mathcal F_{\psi}(\phi^0,\phi^1,\phi^{\perp},\psi,\kappa,\la_1,\xi_1)\|_*\leq CT^{\epsilon}\\
&\left|\mathcal F_{\kappa}(\phi^0,\phi^1,\phi^{\perp},\psi,\kappa,\la_1,\xi_1)-\kappa_0\right|\leq C|\log T|^{-1}\\
&\|\mathcal F_{\la_1}(\phi^0,\phi^1,\phi^{\perp},\psi,\kappa,\la_1,\xi_1)\|_{*,3-\iota}\leq C|\log T|^{1-\iota}\log^2(|\log T|)\\
&\|\mathcal F_{\xi}(\phi^0,\phi^1,\phi^{\perp},\psi,\kappa,\la_1,\xi_1)\|_G\leq CT^{\epsilon}\\
\end{aligned}
\right.
\end{equation}
where $C>0$ is a constant independent of $T$, and $\epsilon>0$ is a small fixed number. On the other hand, compactness of the operator $\mathcal F$ defined in \eqref{def-mF} can be proved by suitable variants of \eqref{contraction-mF}. Indeed, if we vary the parameters $\beta,\alpha,a,a_1,\nu,\nu_1,\nu_2,\sigma$ slightly such that all the restrictions in Section \ref{subsec-choices} are satisfied, then we get \eqref{contraction-mF} with the norms in the left hand side defined by the new parameters, while the closed ball $\mathcal B$ remains the same. To be more specific, for fixed $\sigma',\nu',a'$ which are close to $\sigma,\nu,a$, one can show that if $(\phi^0,\phi^1,\phi^{\perp},\psi,\kappa,\la_1,\xi_1)\in\mathcal B$, then
$$\|\mathcal F_{\phi^0}(\phi^0,\phi^1,\phi^{\perp},\psi,\kappa,\la_1,\xi_1)\|_{0,\sigma',\nu',a'}\leq CT^{\epsilon'}.$$
Moreover, one can show that for $\nu'>\nu$ and $\nu'-\sigma'\beta(4-a')>\nu-\sigma\beta(4-a)$, one has a compact embedding in the sense that if a sequence $\{\phi^0_n\}$ is bounded in the $\|\cdot\|_{0,\sigma',\nu',a'}$-norm, then there exists a subsequence which converges in the $\|\cdot\|_{0,\sigma,\nu,a}$-norm. Thus, the compactness follows directly from a standard diagonal argument by Arzel\`a--Ascoli's theorem. Arguing in a similar manner, one can prove the compactness of the rest operators. Therefore, the existence of the desired solution follows from the Schauder fixed point theorem. The proof is complete.\qed

\medskip

%%%%%%%%%%%%%%%%%%%%%%%%%%%%%%%%%%%%%%%%%%%%%%%%%%%%%%%%%%%%

\medskip

\appendix
\section{Proofs of technical Lemmas}

\medskip

\begin{proof}[Proof of Lemma \ref{lemma-rhs1}]
The proof is achieved by considering the following Cauchy problem in $\R^n$
\begin{equation}\label{outer-linear''''}
\begin{cases}
\partial_t\psi_0=\Delta \psi_0+f,~&\mbox{ in }\R^n\times(0,T)\\
\psi_0(x,0)=0,~&\mbox{ in }\R^n\\
\end{cases}
\end{equation}
If we decompose the solution to \eqref{outer-linear''} in the form
$$\psi=\psi_0+\psi_1,$$
then $\psi_1$ solves the homogeneous heat equation in $\Omega\times(0,T)$ with boundary condition $-\psi_1$. By standard parabolic estimates, it suffices to establish the estimates \eqref{outer-rhs1}--\eqref{outerholder-rhs1} for $\psi_0$. In the sequel, we denote $\psi$ by the solution to \eqref{outer-linear''''} given by Duhamel's formula
\begin{equation*}
\begin{aligned}
\psi(x,t)=&~\int_0^t \int_{\R^n} \frac{e^{-\frac{|x-w|^2}{4(t-s)}}}{(4\pi(t-s))^{n/2}} f(w,s) dw ds\\
\lesssim&~ \int_0^t \la_*^{\nu-3}(s) R^{-2-\alpha}(s) \int_{\left|(r,z)-\xi(s)\right|\leq 2\la_*(s)R(s)} \frac{e^{-\frac{|x-w|^2}{4(t-s)}}}{(4\pi(t-s))^{n/2}} dw ds,\\
\end{aligned}
\end{equation*}
where $w=(w_1,\cdots,w_n)$, $r=\sqrt{w_1^2+\cdots+w_{n-3}^2}$ and  $z=(w_{n-2},w_{n-1},w_n)$. We decompse
\begin{equation*}
\begin{aligned}
&\quad \int_0^t \la_*^{\nu-3}(s) R^{-2-\alpha}(s) \int_{\left|(r,z)-\xi(s)\right|\leq 2\la_*(s)R(s)} \frac{e^{-\frac{|x-w|^2}{4(t-s)}}}{(4\pi(t-s))^{n/2}} dw ds\\
&=\left(\int_0^{t-(T-t)}+\int_{t-(T-t)}^{t-\la^{\delta_1}_*(t)}+\int_{t-\la^{\delta_1}_*(t)}^t\right) \la_*^{\nu-3}(s) R^{-2-\alpha}(s) \int_{\left|(r,z)-\xi(s)\right|\leq 2\la_*(s)R(s)} \frac{e^{-\frac{|x-w|^2}{4(t-s)}}}{(4\pi(t-s))^{n/2}} dw ds\\
&:=I_{11}+I_{12}+I_{13}
\end{aligned}
\end{equation*}
for some $\delta_1\geq 1$ to be found. Here we recall that
$$\xi(t)\sim(\sqrt{2(n-4)(T-t)},z_0)$$
and take $z_0=(0,0,0)$ for convenience. Directly integrating, we obtain
\begin{equation*}
\begin{aligned}
I_{11}=&~ \int_0^{t-(T-t)}\la_*^{\nu-3}(s) R^{-2-\alpha}(s) \int_{\left|(r,z)-\xi(s)\right|\leq 2\la_*(s)R(s)} \frac{e^{-\frac{|x-w|^2}{4(t-s)}}}{(4\pi(t-s))^{n/2}} dw ds\\
\lesssim&~\int_0^{t-(T-t)}\la_*^{\nu-3}(s) R^{-2-\alpha}(s) \int_{\left|(\tilde r,\tilde z)-\frac{\xi(s)}{\sqrt{t-s}}\right|\leq \frac{2\la_*(s)R(s)}{\sqrt{t-s}}} e^{-\frac{|\tilde x-\tilde w|^2}{4}} d\tilde w ds\\
\lesssim&~ \int_0^{t-(T-t)} \la_*^{\nu-3}(s) R^{-2-\alpha}(s)\int_{\left|(\tilde r,\tilde z)-\frac{\xi(s)}{\sqrt{t-s}}\right|\leq \frac{2\la_*(s)R(s)}{\sqrt{t-s}}} (\tilde r)^{n-4} d\tilde r d\tilde z ds\\
\lesssim&~\int_0^{t-(T-t)} \la_*^{\nu-3}(s) R^{-2-\alpha}(s)\frac{\la_*(s)R(s)(\sqrt{T-s})^{n-4}}{(t-s)^{(n-3)/2}}\frac{\left(\la_*(s)R(s)\right)^3}{(t-s)^{3/2}} ds\\
\lesssim&~ \int_0^{t-(T-t)} \frac{\la_*^{\nu-3}(s) R^{-2-\alpha}(s)}{(T-s)^{n/2}}\left(\la_*(s)R(s)\right)^4 (T-s)^{\frac{n-4}{2}} ~ds\\
\lesssim&~ \la_*^{\nu}(0)R^{2-\alpha}(0),
\end{aligned}
\end{equation*}
where $\tilde x=x(t-s)^{-1/2}$, $\tilde w=w(t-s)^{-1/2}$, $\tilde r=r(t-s)^{-1/2}$, $\tilde z=z(t-s)^{-1/2}$, and for the third inequality above, we have used the fact that
$$\sqrt{T-s}\gg \la_*(s)R(s)$$
for $T\ll 1$ since $\la_*(s)R(s)=\la_*^{1-\beta}(s)$ with $0<\beta<1/2$. Then similarly we compute

\begin{equation}\label{est-I12}
\begin{aligned}
I_{12}\lesssim&~ \int_{t-(T-t)}^{t-\la^{\delta_1}_*(t)} \la_*^{\nu-3}(s) R^{-2-\alpha}(s)\int_{\left|(\tilde r,\tilde z)-\frac{\xi(s)}{\sqrt{t-s}}\right|\leq \frac{2\la_*(s)R(s)}{\sqrt{t-s}}} (\tilde r)^{n-4} d\tilde r d\tilde z ds\\
\lesssim&~\int_{t-(T-t)}^{t-\la^{\delta_1}_*(t)} \la_*^{\nu-3}(s) R^{-2-\alpha}(s)\frac{\la_*(s)R(s)(\sqrt{T-s})^{n-4}}{(t-s)^{(n-3)/2}}\frac{\left(\la_*(s)R(s)\right)^3}{(t-s)^{3/2}} ds\\
\lesssim&~ \la_*^{\nu+\frac{(n-2)(1-\delta_1)}{2}}(t)R^{2-\alpha}(t)|\log (T-t)|
\end{aligned}
\end{equation}
and
\begin{equation}\label{est-I13}
\begin{aligned}
I_{13}\lesssim&~ \int_{t-\la^{\delta_1}_*(t)}^t \la_*^{\nu-3}(s) R^{-2-\alpha}(s) ds\\
\lesssim&~ \la_*^{\nu-3+\delta_1}(t) R^{-2-\alpha}(t)|\log(T-t)|.
\end{aligned}
\end{equation}
Since $\beta\in(0,1/2)$ and $n\geq 5$, we can choose $\delta_1=\frac{n+4-8\beta}{n}$. Therefore, we get
$$I_{11}+I_{12}+I_{13}\lesssim \la_*^{\nu-2+\frac4n}(0) R^{-2-\alpha+\frac8n}(0)|\log T|$$
as desired.

Similarly, to prove \eqref{outerT-rhs1}, we decompose
$$|\psi(x,t)-\psi(x,T)|\leq I_{21}+I_{22}+I_{23}$$
with
$$I_{21}=\int_0^{t-(T-t)} \int_{\R^n}|G(x-w,t-s)-G(x-w,T-s)||f(w,s)|dwds,$$
$$I_{22}=\int_{t-(T-t)}^t \int_{\R^n}|G(x-w,t-s)-G(x-w,T-s)||f(w,s)|dwds,$$
$$I_{23}=\int_t^T \int_{\R^n}|G(x-w,T-s)||f(w,s)|dwds,$$
where $G(x,t)$ is the heat kernel
\begin{equation}\label{def-heatkernel}
G(x,t)=\frac{e^{-\frac{|x|^2}{4t}}}{(4\pi t)^{n/2}}.
\end{equation}
For the first integral $I_{21}$, we have
$$I_{21}\leq (T-t)\int_0^1\int_0^{t-(T-t)}\int_{|(w_r,w_z)-\xi(s)|\leq 2\la_*(s)R(s)} |\partial_t G(x-w,t_v-s)|\la_*^{\nu-3}(s)R^{-2-\alpha}(s) dwdsdv,$$
where $(w_r,w_z)=\left(\sqrt{w_1^2+\cdots+w_{n-3}^2},w_{n-2},w_{n-1},w_n\right)$ and $t_v=vT+(1-v)t$. Changing variables
$$w_v=(w_{r,v},w_{z,v})=(w_r (t_v-s)^{-1/2},w_z (t_v-s)^{-1/2}),$$
we evaluate
\begin{equation*}
\begin{aligned}
&\quad \int_{|(w_r,w_z)-\xi(s)|\leq 2\la_*(s)R(s)} |\partial_t G(x-w,t_v-s)| dw\\
&\lesssim \int_{|(w_r,w_z)-\xi(s)|\leq 2\la_*(s)R(s)} e^{-\frac{|x-w|^2}{4(t_v-s)}}\left(\frac{|x-w|^2}{(t_v-s)^{\frac{n+4}{2}}}+\frac{1}{(t_v-s)^{\frac{n+2}{2}}}\right)dw\\
&=\int_{\left|(w_{r,v},w_{z,v})-\frac{\xi(s)}{\sqrt{t_v-s}}\right|\leq \frac{2 \la_*(s)R(s)}{\sqrt{t_v-s}}} e^{-\frac{|x_v-w_v|^2}{4}}\left(1+|x_v-w_v|^2\right)\frac{1}{t_v-s}dw_v\\
\end{aligned}
\end{equation*}
and thus
\begin{equation*}
\begin{aligned}
&\quad\int_0^{t-(T-t)}\int_{|(w_r,w_z)-\xi(s)|\leq 2\la_*(s)R(s)} |\partial_t G(x-w,t_v-s)|\la_*^{\nu-3}(s)R^{-2-\alpha}(s) dwds\\
&\lesssim \int_0^{t-(T-t)} \la_*^{\nu-3}(s)R^{-2-\alpha}(s) \frac{(T-s)^{\frac{n-4}{2}}(\la_*(s)R(s))^4}{(t_v-s)^{\frac{n+2}{2}}}ds\\
&\lesssim \int_0^{t-(T-t)} \la_*^{\nu+1}(s)R^{2-\alpha}(s) (T-s)^{-3} ds\\
&\lesssim \la_*^{\nu-1}(t)R^{2-\alpha}(t),
\end{aligned}
\end{equation*}
from which we conclude that
\begin{equation}\label{est-I21}
I_{21}\lesssim \la_*^{\nu}(t)R^{2-\alpha}(t)|\log (T-t)|.
\end{equation}
For $I_{22}$, we have
\begin{equation*}
\begin{aligned}
I_{22}\leq&~ \int_{t-(T-t)}^t\int_{|(w_r,w_z)-\xi(s)|\leq 2\la_*(s)R(s)} |G(x-w,t-s)|\la_*^{\nu-3}(s)R^{-2-\alpha}(s)dwds\\
&~+ \int_{t-(T-t)}^t\int_{|(w_r,w_z)-\xi(s)|\leq 2\la_*(s)R(s)} |G(x-w,T-s)|\la_*^{\nu-3}(s)R^{-2-\alpha}(s)dwds.
\end{aligned}
\end{equation*}
The first integral above can be estimated as
\begin{equation*}
\begin{aligned}
&\quad \int_{t-(T-t)}^t \int_{|(w_r,w_z)-\xi(s)|\leq 2\la_*(s)R(s)} |G(x-w,t-s)|\la_*^{\nu-3}(s)R^{-2-\alpha}(s)dwds\\
&=\left(\int_{t-(T-t)}^{t-\la_*^{\delta_1}(t)}+\int_{t-\la_*^{\delta_1}(t)}^t\right) |G(x-w,t-s)|\la_*^{\nu-3}(s)R^{-2-\alpha}(s)dwds.\\
\end{aligned}
\end{equation*}
Notice that we already estimate the above integral in \eqref{est-I12} and \eqref{est-I13}. So with the choice $\delta_1=\frac{n+4-8\beta}{n}$, one has
\begin{equation*}
\begin{aligned}
&\quad \int_{t-(T-t)}^t \int_{|(w_r,w_z)-\xi(s)|\leq 2\la_*(s)R(s)} |G(x-w,t-s)|\la_*^{\nu-3}(s)R^{-2-\alpha}(s)dwds\\
&\lesssim  \la_*^{\nu-2+\frac4n}(t) R^{-2-\alpha+\frac8n}(t)|\log(T-t)|.
\end{aligned}
\end{equation*}
Similarly, it holds that
\begin{equation*}
\begin{aligned}
&\quad \int_{t-(T-t)}^t \int_{|(w_r,w_z)-\xi(s)|\leq 2\la_*(s)R(s)} |G(x-w,T-s)|\la_*^{\nu-3}(s)R^{-2-\alpha}(s)dwds\\
&\lesssim  \la_*^{\nu-2+\frac4n}(t) R^{-2-\alpha+\frac8n}(t)|\log(T-t)|.
\end{aligned}
\end{equation*}
Therefore, we obtain
\begin{equation}\label{est-I22}
I_{22}\lesssim \la_*^{\nu-2+\frac4n}(t) R^{-2-\alpha+\frac8n}(t)|\log(T-t)|.
\end{equation}
For $I_{23}$, changing variables
$$\tilde x=x(T-s)^{-1/2},~\tilde w=w(T-s)^{-1/2}~\mbox{ and }~(\tilde w_r,\tilde w_z)=(w_r (T-s)^{-1/2},w_z (T-s)^{-1/2}),$$
one has
\begin{equation}\label{est-I23}
\begin{aligned}
I_{23}\lesssim&~ \int_t^T \int_{\left|(w_r,w_z)-\xi(s)\right|\leq 2\la_*(s)R(s)}\frac{e^{-\frac{|x-w|^2}{4(T-s)}}}{(T-s)^{n/2}} \la_*^{\nu-3}(s)R^{-2-\alpha}(s)dwds\\
\lesssim&~ \int_t^T \int_{\left|(\tilde w_r, \tilde w_z)-\frac{\xi(s)}{\sqrt{T-s}}\right|\leq \frac{2\la_*(s)R(s)}{\sqrt{T-s}}} e^{-\frac{|\tilde x-\tilde w|^2}{4}} \la_*^{\nu-3}(s)R^{-2-\alpha}(s) d\tilde w ds\\
\lesssim&~ \int_t^T  \la_*^{\nu-3}(s)R^{-2-\alpha}(s)\frac{(T-s)^{\frac{n-4}{2}}(\la_*(s)R(s))^4}{(T-s)^{n/2}} ds\\
\lesssim&~ \int_t^T \frac{|\log T|^{\nu+1-\beta(2-\alpha)}(T-s)^{\nu-1-\beta(2-\alpha)}}{|\log (T-s)|^{2(\nu+1-\beta(2-\alpha))}} ds\\
\lesssim&~ \la_*^{\nu}(t)R^{2-\alpha}(t)
\end{aligned}
\end{equation}
provided $\nu-\beta(2-\alpha)>0$.
Collecting \eqref{est-I21}, \eqref{est-I22} and \eqref{est-I23}, we conclude the validity of \eqref{outerT-rhs1}.

Then we prove the gradient estimate \eqref{outergradient-rhs1}. By the heat kernel, we get
\begin{equation*}
\begin{aligned}
|\nabla \psi(x,t)|\lesssim&~ \int_0^t \frac{\la_*^{\nu-3}(s)R^{-2-\alpha}(s)}{(t-s)^{\frac{n+2}{2}}}\int_{\left|(w_r,w_z)-\xi(s)\right|\leq2 \la_*(s)R(s)} e^{-\frac{|x-w|^2}{4(t-s)}}|x-w| dwds\\
%\lesssim&~ \int_0^t \frac{\la_*^{\nu-3}(s)R^{-2-\alpha}(s)}{(t-s)^{1/2}}\int_{\left|(\tilde w_r,\tilde w_z)-\frac{\xi(s)}{\sqrt{t-s}}\right|\leq \frac{2 \la_*(s)R(s)}{\sqrt{t-s}}} e^{-\frac{|\tilde x-\tilde w|^2}{4}}|\tilde x-\tilde w| d\tilde wds\\
\lesssim&~ \int_0^t \frac{\la_*^{\nu-3}(s)R^{-2-\alpha}(s)}{(t-s)^{1/2}}\int_{\left|(\tilde w_r,\tilde w_z)-\frac{\xi(s)}{\sqrt{t-s}}\right|\leq \frac{2 \la_*(s)R(s)}{\sqrt{t-s}}}  e^{-\frac{|\tilde w|^2}{4}}(1+|\tilde w|) d\tilde wds,\\
\end{aligned}
\end{equation*}
where $\tilde x=x(t-s)^{-1/2}$,
$$(w_r,w_z)=\left(\sqrt{w_1^2+\cdots+w_{n-3}^2},w_{n-2},w_{n-1},w_n\right),$$
and
$$(\tilde w_r,\tilde w_z)=\left(\sqrt{\tilde w_1^2+\cdots+\tilde w_{n-3}^2},\tilde w_{n-2},\tilde w_{n-1},\tilde w_n\right).$$
First, we compute
\begin{equation}\label{grad-rhs11}
\begin{aligned}
&\quad\int_0^{t-(T-t)} \frac{\la_*^{\nu-3}(s)R^{-2-\alpha}(s)}{(t-s)^{1/2}}\int_{\left|(\tilde w_r,\tilde w_z)-\frac{\xi(s)}{\sqrt{t-s}}\right|\leq \frac{2 \la_*(s)R(s)}{\sqrt{t-s}}} e^{-\frac{|\tilde w|^2}{4}}(1+|\tilde w|) d\tilde wds\\
&\lesssim \int_0^{t-(T-t)} \frac{\la_*^{\nu-3}(s)R^{-2-\alpha}(s)}{(t-s)^{1/2}}\frac{\la_*(s)R(s)(\sqrt{T-s})^{n-4}}{(t-s)^{(n-3)/2}}\frac{\left(\la_*(s)R(s)\right)^3}{(t-s)^{3/2}} ds\\
&\lesssim \int_0^{t-(T-t)} \frac{\la_*^{\nu+1}(s)R^{2-\alpha}(s)(T-s)^{\frac{n-4}{2}}}{(t-s)^{\frac{n+1}{2}}} ds\\
&\lesssim \int_0^{t-(T-t)} \la_*^{\nu+1}(s)R^{2-\alpha}(s)(T-s)^{-\frac{5}{2}} ds\\
&\lesssim  \la_*^{\nu-\frac12}(0)R^{2-\alpha}(0)|\log T|.
\end{aligned}
\end{equation}
Then we compute
\begin{equation}\label{grad-rhs12}
\begin{aligned}
&\quad\int_{t-(T-t)}^{t-\la_*^{\delta_2}(t)} \frac{\la_*^{\nu-3}(s)R^{-2-\alpha}(s)}{(t-s)^{1/2}}\int_{\left|(\tilde w_r,\tilde w_z)-\frac{\xi(s)}{\sqrt{t-s}}\right|\leq \frac{2 \la_*(s)R(s)}{\sqrt{t-s}}} e^{-\frac{|\tilde w|^2}{4}}(1+|\tilde w|) d\tilde wds\\
&\lesssim \int_{t-(T-t)}^{t-\la_*^{\delta_2}(t)} \frac{\la_*^{\nu+1}(s)R^{2-\alpha}(s)(T-s)^{\frac{n-4}{2}}}{(t-s)^{\frac{n+1}{2}}} ds\\
&\lesssim \la_*^{\nu+\frac{n-2}{2}+\frac{(1-n)\delta_2}{2}}(t)R^{2-\alpha}(t) |\log(T-t)|,
\end{aligned}
\end{equation}
where $\delta_2\geq 1$ is a constant to be determined. On the other hand, we have
\begin{equation}\label{grad-rhs13}
\begin{aligned}
&\quad\int_{t-\la_*^{\delta_2}(t)}^t \frac{\la_*^{\nu-3}(s)R^{-2-\alpha}(s)}{(t-s)^{\frac{n+2}{2}}}\int_{\left|(w_r,w_z)-\xi(s)\right|\leq2 \la_*(s)R(s)} e^{-\frac{|x-w|^2}{4(t-s)}}|x-w| dwds\\
&\lesssim \int_{t-\la_*^{\delta_2}(t)}^t \frac{\la_*^{\nu-3}(s)R^{-2-\alpha}(s)}{(t-s)^{1/2}} ds\\
&\lesssim \la_*^{\nu-3+\frac{\delta_2}{2}}(t)R^{-2-\alpha}(t).
\end{aligned}
\end{equation}
By choosing $\delta_2=\frac{n+4-8\beta}{n}$ and combining \eqref{grad-rhs11}--\eqref{grad-rhs13}, we prove the validity of the gradient estimate \eqref{outergradient-rhs1}. The proof of \eqref{outergradientT-rhs1} is similar to that of \eqref{outerT-rhs1}. We omit the details.

To prove the H\"older estimate \eqref{outerholder-rhs1}, we decompose
$$|\psi(x,t_2)-\psi(x,t_1)|\leq J_{11}+J_{12}+J_{13}$$
with
$$J_{11}=\int_0^{t_1-(t_2-t_1)}\int_{\R^n} |G(x-w,t_2-s)-G(x-w,t_1-s)|f(w,s)dwds,$$
$$J_{12}=\int_{t_1-(t_2-t_1)}^{t_1}\int_{\R^n} |G(x-w,t_2-s)-G(x-w,t_1-s)|f(w,s)dwds,$$
and
$$J_{13}=\int_{t_1}^{t_2}\int_{\R^n} G(x-w,t_2-s)f(w,s)dwds,$$
where $G(x,t)$ is the heat kernel \eqref{def-heatkernel}. Here we assume that $0<t_1<t_2<T$ with $t_2<2t_1$.
For $J_{11}$, by letting $t_v=vt_2+(1-v)t_1$, we have
\begin{equation*}
\begin{aligned}
J_{11}\leq &~(t_2-t_1)\int_0^1 \int_0^{t_1-(t_2-t_1)}\int_{\R^n} |\partial_t G(x-w,t_v-s)|f(w,s)dwdsdv\\
\lesssim &~ (t_2-t_1)\int_0^1 \int_0^{t_1-(t_2-t_1)}\int_{\left|(w_r,w_z)-\xi(s)\right|\leq2 \la_*(s)R(s)} e^{-\frac{|x-w|^2}{4(t_v-s)}}\bigg(\frac{|x-w|^2}{(t_v-s)^{\frac{n+4}{2}}}\\
&\qquad\qquad\qquad\qquad\qquad\qquad+\frac{1}{(t_v-s)^{\frac{n+2}{2}}}\bigg) \la_*^{\nu-3}(s)R^{-2-\alpha}(s) dwdsdv,\\
\end{aligned}
\end{equation*}
where $(w_r,w_z)=\left(\sqrt{w_1^2+\cdots+w_{n-3}^2},w_{n-2},w_{n-1},w_n\right).$ Taking
$$x_v=x(t_v-s)^{-1/2},~w_v=w(t_v-s)^{-1/2}~\mbox{ and }~(w_{r,v},w_{z,v})=(w_r (t_v-s)^{-1/2},w_z (t_v-s)^{-1/2}),$$ we get
\begin{equation*}
\begin{aligned}
&~\int_{\left|(w_r,w_z)-\xi(s)\right|\leq2 \la_*(s)R(s)} e^{-\frac{|x-w|^2}{4(t_v-s)}}\left(\frac{|x-w|^2}{(t_v-s)^{\frac{n+4}{2}}}+\frac{1}{(t_v-s)^{\frac{n+2}{2}}}\right)\la_*^{\nu-3}(s)R^{-2-\alpha}(s) dw\\
=&~\int_{\left|(w_{r,v},w_{z,v})-\frac{\xi(s)}{\sqrt{t_v-s}}\right|\leq \frac{2 \la_*(s)R(s)}{\sqrt{t_v-s}}} e^{-\frac{|x_v-w_v|^2}{4}}\left(1+|x_v-w_v|^2\right)\frac{\la_*^{\nu-3}(s)R^{-2-\alpha}(s)}{t_v-s}dw_v.\\
\end{aligned}
\end{equation*}
Observing that for any $\mu\in(0,1)$, we have
$$\int_{\left|(w_{r,v},w_{z,v})-\frac{\xi(s)}{\sqrt{t_v-s}}\right|\leq \frac{2 \la_*(s)R(s)}{\sqrt{t_v-s}}} e^{-\frac{|x_v-w_v|^2}{4}}\left(1+|x_v-w_v|^2\right) dw_v \lesssim \left(\frac{\sqrt{T-s}}{\sqrt{t_v-s}}\right)^{\mu},$$
where we have used the facts that $\xi(s)\sim(\sqrt{2(n-4)(T-s)},0,0,0)$ and $\sqrt{T-s}\gg \la_*(s)R(s)$ for $T\ll 1$. Thus, one has
\begin{equation*}
\begin{aligned}
J_{11}\lesssim (t_2-t_1) \int_0^{t_1-(t_2-t_1)} \frac{\la_*^{\nu-3}(s)R^{-2-\alpha}(s)(T-s)^{\frac{\mu}{2}}}{(t_2-s)^{1+\frac{\mu}{2}}}ds.
\end{aligned}
\end{equation*}
Recalling that $R(t)=\la_*^{-\beta}(t)$ for $\beta\in(0,1/2)$, we have the following two cases
\begin{itemize}
\item If $\nu-3+\beta(2+\alpha)+\frac{\mu}{2}< 0$, then we have
\begin{equation*}
\begin{aligned}
&\quad \int_0^{t_1-(t_2-t_1)} \frac{\la_*^{\nu-3}(s)R^{-2-\alpha}(s)(T-s)^{\frac{\mu}{2}}}{(t_2-s)^{1+\frac{\mu}{2}}}ds\\
&\lesssim \la_*^{\nu-3+\frac{\mu}{2}}(t_1)R^{-2-\alpha}(t_1)|\log(T-t_1)|\int_0^{t_1-(t_2-t_1)}\frac{1}{(t_2-s)^{1+\frac{\mu}{2}}}ds\\
&\lesssim \la_*^{\nu-3+\frac{\mu}{2}}(t_1)R^{-2-\alpha}(t_1)|\log(T-t_1)|(t_2-t_1)^{-\mu/2}.
\end{aligned}
\end{equation*}

\item If $\nu-3+\beta(2+\alpha)+\frac{\mu}{2}\geq 0$, then we decompose
\begin{equation*}
\begin{aligned}
&\quad\int_0^{t_1-(t_2-t_1)} \frac{\la_*^{\nu-3}(s)R^{-2-\alpha}(s)(T-s)^{\frac{\mu}{2}}}{(t_2-s)^{1+\frac{\mu}{2}}}ds\\
&= \left(\int_0^{t_1-(T-t_1)}+\int_{t_1-(T-t_1)}^{t_1-(t_2-t_1)}\right) \frac{\la_*^{\nu-3}(s)R^{-2-\alpha}(s)(T-s)^{\frac{\mu}{2}}}{(t_2-s)^{1+\frac{\mu}{2}}}ds.
\end{aligned}
\end{equation*}
Assuming
$$\nu-3+\beta(2+\alpha)<0,$$
we obtain that
\begin{equation*}
\begin{aligned}
&\quad\int_0^{t_1-(T-t_1)} \frac{\la_*^{\nu-3}(s)R^{-2-\alpha}(s)(T-s)^{\frac{\mu}{2}}}{(t_2-s)^{1+\frac{\mu}{2}}}ds\\
&\lesssim \int_0^{t_1-(T-t_1)} \frac{\la_*^{\nu-3}(s)R^{-2-\alpha}(s)(T-s)^{\frac{\mu}{2}}}{(T-s)^{1+\frac{\mu}{2}}}ds\\
&=\int_0^{t_1-(T-t_1)}\frac{|\log T|^{\nu-3+\beta(2+\alpha)}(T-s)^{\nu-4+\beta(2+\alpha)}}{|\log(T-s)|^{2(\nu-3+\beta(2+\alpha))}} ds\\
%&\lesssim\frac{|\log T|^{\nu-3+\beta(2+\alpha)}(T-t_1)^{\nu-3+\beta(2+\alpha)}}{|\log(T-t_1)|^{2(\nu-3+\beta(2+\alpha))}} ds\\
%&\lesssim\la_*^{\nu-3+\beta(2+\alpha)}(t_2)\\
&\lesssim\la_*^{\nu+\frac{\mu}{2}-3}(t_2)R^{-2-\alpha}(t_2)(t_2-t_1)^{-\mu/2}
\end{aligned}
\end{equation*}
and similarly
\begin{equation*}
\begin{aligned}
\int_{t_1-(T-t_1)}^{t_1-(t_2-t_1)} \frac{\la_*^{\nu-3}(s)R^{-2-\alpha}(s)(T-s)^{\frac{\mu}{2}}}{(t_2-s)^{1+\frac{\mu}{2}}}ds\lesssim \la_*^{\nu+\frac{\mu}{2}-3}(t_2)R^{-2-\alpha}(t_2)(t_2-t_1)^{-\mu/2}.
\end{aligned}
\end{equation*}
\end{itemize}
In both cases, we have
$$J_{11}\lesssim \la_*^{\nu+\frac{\mu}{2}-3}(t_2)R^{-2-\alpha}(t_2)(t_2-t_1)^{1-\mu/2}.$$

For $J_{12}$, we evaluate
\begin{equation*}
\begin{aligned}
&~\int_{t_1-(t_2-t_1)}^{t_1}\int_{\R^n} |G(x-w,t_1-s)|f(w,s)dwds\\
\lesssim&~ \int_{t_1-(t_2-t_1)}^{t_1} \la_*^{\nu-3}(s)R^{-2-\alpha}(s)\int_{\frac{\left|(w_r,w_z)-\xi(s)\right|}{\sqrt{t_1-s}}\leq \frac{2 \la_*(s)R(s)}{\sqrt{t_1-s}}} e^{-\frac{|\tilde x-\tilde w|^2}{4}}d \tilde w ds\\
\lesssim&~ \int_{t_1-(t_2-t_1)}^{t_1} \la_*^{\nu-3}(s)R^{-2-\alpha}(s)\left(\frac{\sqrt{T-s}}{\sqrt{t_1-s}}\right)^{\mu} ds\\
=&~ \int_{t_1-(t_2-t_1)}^{t_1} \frac{|\log T|^{\nu-3+\beta(2+\alpha)}(T-s)^{\nu+\frac{\mu}{2}-3+\beta(2+\alpha)}}{|\log(T-s)|^{2(\nu-3+\beta(2+\alpha))}(t_1-s)^{\mu/2}} ds\\
\lesssim&~ \la_*^{\nu+\frac{\mu}{2}-3}(t_2)R^{-2-\alpha}(t_2) (t_2-t_1)^{1-\mu/2},
\end{aligned}
\end{equation*}
where we have changed variables $\tilde x=x(t_1-s)^{-1/2}$, $\tilde w=w(t_1-s)^{-1/2}$, and $\mu\in(0,1)$. Similarly, we have
$$\int_{t_1-(t_2-t_1)}^{t_1}\int_{\R^n} |G(x-w,t_2-s)|f(w,s)dwds\lesssim \la_*^{\nu+\frac{\mu}{2}-3}(t_2)R^{-2-\alpha}(t_2) (t_2-t_1)^{1-\mu/2}.$$
Thus we conclude that
$$J_{12}\lesssim \la_*^{\nu+\frac{\mu}{2}-3}(t_2)R^{-2-\alpha}(t_2) (t_2-t_1)^{1-\mu/2}.$$
Finally, for $J_{13}$,
\begin{equation*}
\begin{aligned}
J_{13}=&~\int_{t_1}^{t_2}\int_{\R^n} G(x-w,t_2-s)f(w,s)dwds\\
\lesssim&~ \int_{t_1}^{t_2} \la_*^{\nu-3}(s)R^{-2-\alpha}(s) \int_{\frac{\left|(w_r,w_z)-\xi(s)\right|}{\sqrt{t_2-s}}\leq \frac{2 \la_*(s)R(s)}{\sqrt{t_2-s}}} e^{-\frac{|\tilde x-\tilde w|^2}{4}}d \tilde w ds\\
\lesssim&~ \la_*^{\nu+\frac{\mu}{2}-3}(t_2)R^{-2-\alpha}(t_2) (t_2-t_1)^{1-\mu/2}
\end{aligned}
\end{equation*}
follows from the same argument as before, where $\tilde x=x(t_2-s)^{-1/2}$ and $\tilde w=w(t_2-s)^{-1/2}$. This completes the proof of \eqref{outerholder-rhs1}.
\end{proof}

\begin{proof}[Proof of Lemma \ref{lemma-rhs2}]
We first prove \eqref{outer-rhs2}. Similar to the proof of Lemma \ref{lemma-rhs1}, Duhamel's formula gives
\begin{equation*}%\label{duhamel-rhs2}
\begin{aligned}
|\psi(x,t)|\lesssim& \int_0^t \frac{\la^{\nu_2}_*(s)}{(t-s)^{n/2}}\int_{\la_*(s)R(s)\leq \left|(w_r,w_z)-\xi(s)\right|\leq 2\delta\sqrt{T-s}} \frac{e^{-\frac{|x-w|^2}{4(t-s)}}}{\left|(w_r,w_z)-\xi(s)\right|^2} dwds\\
\lesssim& \int_0^t \frac{\la^{\nu_2}_*(s)}{t-s}\int_{\frac{\la_*(s)R(s)}{\sqrt{t-s}}\leq\left|(\tilde w_r,\tilde w_z)-\frac{\xi(s)}{\sqrt{t-s}}\right|\leq \frac{2\delta\sqrt{T-s}}{\sqrt{t-s}}} \frac{e^{-\frac{|\tilde x-\tilde w|^2}{4}}}{\left|(\tilde w_r,\tilde w_z)-\frac{\xi(s)}{\sqrt{t-s}}\right|^2} d\tilde wds,\\
\end{aligned}
\end{equation*}
where $\tilde x=x(t-s)^{-1/2}$,
$$(w_r,w_z)=\left(\sqrt{w_1^2+\cdots+w_{n-3}^2},w_{n-2},w_{n-1},w_n\right),$$
and
$$(\tilde w_r,\tilde w_z)=\left(\sqrt{\tilde w_1^2+\cdots+\tilde w_{n-3}^2},\tilde w_{n-2},\tilde w_{n-1},\tilde w_n\right)$$
with $\tilde w_i=w_i(t-s)^{-1/2}$, $i=1,\cdots,n$. We have
\begin{equation*}
\begin{aligned}
|\psi(x,t)|\lesssim&~ \int_0^t \frac{\la^{\nu_2}_*(s)}{t-s}\int_{\frac{\la_*(s)R(s)}{\sqrt{t-s}}\leq\left|(\tilde w_r,\tilde w_z)-\frac{\xi(s)}{\sqrt{t-s}}\right|\leq \frac{2\delta\sqrt{T-s}}{\sqrt{t-s}}} \frac{e^{-\frac{|\tilde x-\tilde w|^2}{4}}}{\left|(\tilde w_r,\tilde w_z)-\frac{\xi(s)}{\sqrt{t-s}}\right|^2} d\tilde wds\\
\lesssim&~ \int_0^t \frac{\la^{\nu_2}_*(s)}{t-s} \frac{t-s}{\la_*^2(s)R^2(s)} ds\\
=&~ \int_0^t \frac{|\log T|^{\nu_2-2+2\beta}(T-s)^{\nu_2-2+2\beta}}{|\log (T-s)|^{2\nu_2-4+4\beta}} ds\\
\lesssim&~ \la_*^{\nu_2-1}(0) R^{-2}(0)|\log T|
\end{aligned}
\end{equation*}
as desired.

%Then we have
%\begin{equation}
%\begin{aligned}
%\int_0^t \frac{\la^{\nu_2}_*(s)}{t-s} \int_{A_1} \frac{e^{-\frac{|\tilde x-\tilde w|^2}{4}}}{\left|(\tilde w_r,\tilde w_z)-\frac{\xi(s)}{\sqrt{t-s}}\right|^2} d\tilde wds \lesssim&~ \int_0^t \frac{\la^{\nu_2}_*(s)}{t-s} \frac{t-s}{\la_*^2(s)R^2(s)}\frac{(T-s)^{n\delta_3}}{(t-s)^{n\delta_4}} ds\\
%\lesssim&~ \la_*^{\nu_2-1+n\delta_3-n\delta_4}(0)R^{-2}(0)|\log T|^{n\delta_3-n\delta_4-1}
%\end{aligned}
%\end{equation}
%and
%\begin{equation}
%\begin{aligned}
%\int_0^t \frac{\la^{\nu_2}_*(s)}{t-s} \int_{A_2} \frac{e^{-\frac{|\tilde x-\tilde w|^2}{4}}}{\left|(\tilde w_r,\tilde w_z)-\frac{\xi(s)}{\sqrt{t-s}}\right|^2} d\tilde wds \lesssim&~ \int_0^t \frac{\la^{\nu_2}_*(s)}{t-s} \left(\frac{(t-s)^{\delta_4}}{(T-s)^{\delta_3}}\right)^{\delta_5}\left(\sqrt{\frac{T-s}{t-s}}\right)^{n-2} ds\\
%\lesssim&~\la_*^{\nu_2-\delta_3 \delta_5+\delta_4 \delta_5}(0)|\log T|^{\delta_4\delta_5-\delta_3\delta_5},
%\end{aligned}
%\end{equation}
%where we have used
%$$
%\int_{\frac{\la_*(s)R(s)}{\sqrt{t-s}}\leq \left|(\tilde w_r,\tilde w_z)-\frac{\xi(s)}{\sqrt{t-s}}\right|\leq \frac{2\delta \sqrt{T-s}}{\sqrt{t-s}}} \frac{1}{\left|(\tilde w_r,\tilde w_z)-\frac{\xi(s)}{\sqrt{t-s}}\right|^2} d\tilde w ds \lesssim \left(\sqrt{\frac{T-s}{t-s}}\right)^{n-2}.
%$$

Now we prove the gradient estimate \eqref{outergradient-rhs2}. By the heat kernel, we have
\begin{equation*}
\begin{aligned}
|\nabla \psi(x,t)|\lesssim& ~\int_0^t \frac{\la_*^{\nu_2}(s)}{(t-s)^{\frac{n+2}{2}}}\int_{\la_*(s)R(s)\leq \left|(w_r,w_z)-\xi(s)\right|\leq 2\delta\sqrt{T-s}} \frac{e^{-\frac{|x-w|^2}{4(t-s)}}|x-w|}{|(w_r,w_z)-\xi(s)|^2} dw ds\\
=&~\int_0^t \frac{\la_*^{\nu_2}(s)}{(t-s)^{3/2}}\int_{\frac{\la_*(s)R(s)}{\sqrt{t-s}}\leq\left|(\tilde w_r,\tilde w_z)-\frac{\xi(s)}{\sqrt{t-s}}\right|\leq \frac{2\delta\sqrt{T-s}}{\sqrt{t-s}}} \frac{e^{-\frac{|\tilde x-\tilde w|^2}{4}}|\tilde x-\tilde w|}{\left|(\tilde w_r,\tilde w_z)-\frac{\xi(s)}{\sqrt{t-s}}\right|^2} d\tilde w ds\\
\lesssim&~\int_0^t \frac{\la_*^{\nu_2}(s)}{(t-s)^{3/2}}\frac{t-s}{\la_*^2(s)R^2(s)} ds\\
\lesssim&~\la_*^{\nu_2-2}(t)R^{-2}(t)\sqrt{t}.
\end{aligned}
\end{equation*}

Next we prove the H\"older estimate \eqref{outerholder-rhs2}. We decompose
$$|\psi(x,t_2)-\psi(x,t_1)|\leq K_{11}+K_{12}+K_{13}$$
with
$$K_{11}=\int_0^{t_1-(t_2-t_1)}\int_{\R^n} |G(x-w,t_2-s)-G(x-w,t_1-s)|f(w,s)dwds,$$
$$K_{12}=\int_{t_1-(t_2-t_1)}^{t_1}\int_{\R^n} |G(x-w,t_2-s)-G(x-w,t_1-s)|f(w,s)dwds,$$
and
$$K_{13}=\int_{t_1}^{t_2}\int_{\R^n} G(x-w,t_2-s)f(w,s)dwds,$$
where $G(x,t)$ is the heat kernel \eqref{def-heatkernel}. For $K_{13}$, we have
\begin{equation*}
\begin{aligned}
|K_{13}|\lesssim & \int_{t_1}^{t_2} \int_{\la_*(s)R(s)\leq |(w_r,w_z)-\xi(s)|\leq 2\delta \sqrt{T-s}}\frac{e^{-\frac{|x-w|^2}{4(t_2-s)}}}{(t_2-s)^{n/2}}\frac{\la_*^{\nu_2}(s)}{|(w_r,w_z)-\xi(s)|^2} dw ds\\
=&\int_{t_1}^{t_2} \int_{\frac{\la_*(s)R(s)}{\sqrt{t_2-s}}\leq \left|(\tilde w_r,\tilde w_z)-\frac{\xi(s)}{\sqrt{t_2-s}}\right|\leq \frac{2\delta \sqrt{T-s}}{\sqrt{t_2-s}}}\frac{e^{-\frac{|\tilde x-\tilde w|^2}{4}}}{t_2-s}\frac{\la_*^{\nu_2}(s)}{\left|(\tilde w_r,\tilde w_z)-\frac{\xi(s)}{\sqrt{t_2-s}}\right|^2} d\tilde w ds\\
=& \int_{t_1}^{t_2}\left( \int_{A_1}+\int_{A_2}\right)\frac{e^{-\frac{|\tilde x-\tilde w|^2}{4}}}{t_2-s}\frac{\la_*^{\nu_2}(s)}{\left|(\tilde w_r,\tilde w_z)-\frac{\xi(s)}{\sqrt{t_2-s}}\right|^2} d\tilde w ds,\\
\end{aligned}
\end{equation*}
where
\begin{equation*}%\label{def-domA1}
A_1=\left\{\tilde w\in\R^n:~\frac{\la_*(s)R(s)}{\sqrt{t_2-s}}\leq \left|(\tilde w_r,\tilde w_z)-\frac{\xi(s)}{\sqrt{t_2-s}}\right|\leq \frac{2\delta \sqrt{T-s}}{\sqrt{t_2-s}},~|\tilde x-\tilde w|< \frac{(T-s)^{\delta_3}}{(t_2-s)^{\delta_4}}\right\}
\end{equation*}
and
\begin{equation*}%\label{def-domA2}
A_2=\left\{\tilde w\in\R^n:~\frac{\la_*(s)R(s)}{\sqrt{t_2-s}}\leq \left|(\tilde w_r,\tilde w_z)-\frac{\xi(s)}{\sqrt{t_2-s}}\right|\leq \frac{2\delta \sqrt{T-s}}{\sqrt{t_2-s}},~|\tilde x-\tilde w| \geq \frac{(T-s)^{\delta_3}}{(t_2-s)^{\delta_4}}\right\}.
\end{equation*}
Then, one has
\begin{equation*}
\begin{aligned}
&\quad\int_{t_1}^{t_2}\int_{A_1}\frac{e^{-\frac{|\tilde x-\tilde w|^2}{4}}}{t_2-s}\frac{\la_*^{\nu_2}(s)}{\left|(\tilde w_r,\tilde w_z)-\frac{\xi(s)}{\sqrt{t_2-s}}\right|^2} d\tilde w ds\\
&\lesssim \int_{t_1}^{t_2} \frac{\la_*^{\nu_2}(s)}{t_2-s} \frac{t_2-s}{\la_*^2(s)R^2(s)}\frac{(T-s)^{n\delta_3}}{(t_2-s)^{n\delta_4}} ds\\
&\lesssim \la_*^{\nu_2-2+n\delta_3}(t_2) R^{-2}(t_2) (t_2-t_1)^{1-n\delta_4}
\end{aligned}
\end{equation*}
and
\begin{equation*}
\begin{aligned}
&\quad\int_{t_1}^{t_2}\int_{A_2}\frac{e^{-\frac{|\tilde x-\tilde w|^2}{4}}}{t_2-s}\frac{\la_*^{\nu_2}(s)}{\left|(\tilde w_r,\tilde w_z)-\frac{\xi(s)}{\sqrt{t_2-s}}\right|^2} d\tilde w ds\\
&\lesssim \int_{t_1}^{t_2} \frac{\la_*^{\nu_2}(s)}{t_2-s} \left(\frac{(t_2-s)^{\delta_4}}{(T-s)^{\delta_3}}\right)^{\delta_5}\left(\sqrt{\frac{T-s}{t_2-s}}\right)^{n-2} ds\\
&\lesssim \la_*^{\nu_2-\delta_3\delta_5+\frac{n-2}{2}}(t_2) (t_2-t_1)^{\delta_4\delta_5-\frac{n-2}{2}},
\end{aligned}
\end{equation*}
where we have used
$$
\int_{\frac{\la_*(s)R(s)}{\sqrt{t_2-s}}\leq \left|(\tilde w_r,\tilde w_z)-\frac{\xi(s)}{\sqrt{t_2-s}}\right|\leq \frac{2\delta \sqrt{T-s}}{\sqrt{t_2-s}}} \frac{1}{\left|(\tilde w_r,\tilde w_z)-\frac{\xi(s)}{\sqrt{t_2-s}}\right|^2} d\tilde w ds \lesssim \left(\sqrt{\frac{T-s}{t_2-s}}\right)^{n-2}.
$$

Similar to the estimate of $K_{13}$, we evaluate
\begin{equation*}
\begin{aligned}
|K_{11}|\lesssim &~ (t_2-t_1)\int_0^1 \int_0^{t_1-(t_2-t_1)}\int_{\la_*(s)R(s)\leq |(w_r,w_z)-\xi(s)|\leq 2\delta \sqrt{T-s}} e^{-\frac{|x-w|^2}{4(t_v-s)}}\left(\frac{|x-w|^2}{(t_v-s)^{\frac{n+4}{2}}}\right.\\
&\qquad\qquad\qquad\qquad\left.+\frac{1}{(t_v-s)^{\frac{n+2}{2}}}\right)\frac{\la_*^{\nu_2}(s)}{|(w_r,w_z)-\xi(s)|^2} dw ds dv\\
=&~ (t_2-t_1)\int_0^1 \int_0^{t_1-(t_2-t_1)}\int_{\frac{\la_*(s)R(s)}{\sqrt{t_v-s}}\leq \left|(\tilde w_r,\tilde w_z)-\frac{\xi(s)}{\sqrt{t_v-s}}\right|\leq \frac{2\delta \sqrt{T-s}}{\sqrt{t_v-s}}} \frac{e^{-\frac{|\tilde x-\tilde w|^2}{4}}}{(t_v-s)^2}\left(|\tilde x-\tilde w|^2\right.\\
&\qquad\qquad\qquad\qquad\left.+1\right)\frac{\la_*^{\nu_2}(s)}{\left|(\tilde w_r,\tilde w_z)-\frac{\xi(s)}{\sqrt{t_v-s}}\right|^2} d\tilde w ds dv\\
\lesssim &~ \la_*^{\nu_2-2+n\delta_3}(t_2) R^{-2}(t_2) (t_2-t_1)^{1-n\delta_4}+\la_*^{\nu_2-\delta_3\delta_5+\frac{n-2}{2}}(t_2) (t_2-t_1)^{\delta_4\delta_5-\frac{n-2}{2}}.
\end{aligned}
\end{equation*}
Estimate of $K_{12}$ can be carried out similarly. Therefore, collecting the estimates above, we conclude that
\begin{equation*}
|\psi(x,t_2)-\psi(x,t_1)| \lesssim \la_*^{\nu_2-2+n\delta_3}(t_2) R^{-2}(t_2) (t_2-t_1)^{1-n\delta_4}+\la_*^{\nu_2-\delta_3\delta_5+\frac{n-2}{2}}(t_2) (t_2-t_1)^{\delta_4\delta_5-\frac{n-2}{2}}.
\end{equation*}
Taking $\delta_3=\delta_4$ and $1-n\delta_4=\delta_4\delta_5-\frac{n-2}{2}=\gamma$ with $0<\gamma<1$, we obtain
$$
|\psi(x,t_2)-\psi(x,t_1)| \lesssim \la_*^{\nu_2-1-\gamma}(t_2) R^{-2}(t_2) (t_2-t_1)^{\gamma}
$$
as desired.

%%%%%%%%%%%%%%%%%%%%%%%%%%%%%%%%%%%%%%%%%%%%%%%%%%%%%%%%

Estimates \eqref{outerT-rhs2} and \eqref{outergradientT-rhs2} follow in a similar manner as that of Lemma \ref{lemma-rhs1}. We omit the details.
\end{proof}

\begin{proof}[Proof of Lemma \ref{lemma-rhs3}]
By the heat kernel, we directly get
$$|\psi(x,t)|\lesssim \int_0^t\int_{\R^n} \frac{e^{-\frac{|x-w|^2}{4(t-s)}}}{(t-s)^{n/2}} dwds\lesssim t.$$
To prove the gradient estimate, we write
\begin{equation*}
\begin{aligned}
|\nabla \psi(x,t)|\lesssim &\int_0^t \frac{1}{(t-s)^{\frac{n+2}{2}}}\int_{\R^n} e^{-\frac{|x-w|^2}{4(t-s)}}|x-w| dwds\\
\lesssim & \int_0^t \frac{1}{(t-s)^{1/2}}\int_{\R^n} e^{-\frac{|\tilde x-\tilde w|^2}{4}}|\tilde x-\tilde w| d\tilde wds\\
\lesssim & \int_0^t \frac{1}{(t-s)^{1/2}} ds\lesssim t^{1/2}.
\end{aligned}
\end{equation*}
The proofs of \eqref{outerT-rhs3} and \eqref{outergradientT-rhs3} can be carried out similarly. So we omit the details.

In order to prove the H\"older estimate \eqref{outerholder-rhs3}, we decompose
$$|\psi(x,t_2)-\psi(x,t_1)|\leq J_{31}+J_{32}+J_{33}$$
with
$$J_{31}=\int_0^{t_1-(t_2-t_1)}\int_{\R^n} |G(x-w,t_2-s)-G(x-w,t_1-s)|f(w,s)dwds,$$
$$J_{32}=\int_{t_1-(t_2-t_1)}^{t_1}\int_{\R^n} |G(x-w,t_2-s)-G(x-w,t_1-s)|f(w,s)dwds,$$
and
$$J_{33}=\int_{t_1}^{t_2}\int_{\R^n} G(x-w,t_2-s)f(w,s)dwds,$$
where $G(x,t)$ is the heat kernel \eqref{def-heatkernel}.
Here we assume that $0<t_1<t_2<T$ with $t_2<2t_1$.
For $J_{31}$, by letting $t_v=vt_2+(1-v)t_1$, we have
\begin{equation*}
\begin{aligned}
J_{31}\leq &~(t_2-t_1)\int_0^1 \int_0^{t_1-(t_2-t_1)}\int_{\R^n} |\partial_t G(x-w,t_v-s)|f(w,s)dwdsdv\\
\lesssim &~ (t_2-t_1)\int_0^1 \int_0^{t_1-(t_2-t_1)}\int_{\R^n} e^{-\frac{|x-w|^2}{4(t_v-s)}}\left(\frac{|x-w|^2}{(t_v-s)^{\frac{n+4}{2}}}+\frac{1}{(t_v-s)^{\frac{n+2}{2}}}\right)  dwdsdv.\\
\end{aligned}
\end{equation*}
Let $x_v=x(t_v-s)^{-1/2}$ and $w_v=w(t_v-s)^{-1/2}$. We then have
\begin{equation*}
\begin{aligned}
&~\int_{\R^n} e^{-\frac{|x-w|^2}{4(t_v-s)}}\left(\frac{|x-w|^2}{(t_v-s)^{\frac{n+4}{2}}}+\frac{1}{(t_v-s)^{\frac{n+2}{2}}}\right) dw\\
=&~\int_{\R^n} e^{-\frac{|x_v-w_v|^2}{4}}\left(|x_v-w_v|^2+1\right)\frac{1}{t_v-s}dw_v\\
\lesssim &~\frac{1}{t_v-s} \int_{0}^{+\infty} e^{-\frac{r^2}{4}}(1+r^2) r^{n-1}d r.
\end{aligned}
\end{equation*}
So we get
\begin{equation*}
\begin{aligned}
J_{31}\lesssim (t_2-t_1)\int_0^{t_1-(t_2-t_1)} \frac{1}{t_v-s} ds\lesssim (t_2-t_1)|\log (t_2-t_1)|.
\end{aligned}
\end{equation*}

For $J_{32}$, we evaluate
\begin{equation*}
\begin{aligned}
&~\int_{t_1-(t_2-t_1)}^{t_1}\int_{\R^n} |G(x-w,t_1-s)|f(w,s)dwds\\
\lesssim&~ \int_{t_1-(t_2-t_1)}^{t_1}\int_{\R^n}e^{-\frac{|\tilde x-\tilde w|^2}{4}}d \tilde w ds\\
\lesssim&~ t_2-t_1
\end{aligned}
\end{equation*}
and similarly
$$\int_{t_1-(t_2-t_1)}^{t_1}\int_{\R^n} |G(x-w,t_2-s)|f(w,s)dwds\lesssim t_2-t_1,$$
where we have changed variables $\tilde x=x(t_1-s)^{-1/2}$ and $\tilde w=w(t_1-s)^{-1/2}$. Thus we conclude that
$$J_{32}\lesssim t_2-t_1.$$
Finally, for $J_{33}$
\begin{equation*}
\begin{aligned}
J_{33}\lesssim&~ \int_{t_1}^{t_2}  \int_{\R^n}e^{-\frac{|\tilde x-\tilde w|^2}{4}}d \tilde w ds\\
\lesssim&~ t_2-t_1
\end{aligned}
\end{equation*}
follows from the same argument as before. This completes the proof of \eqref{outerholder-rhs3} by collecting the estimates of $J_{31}$, $J_{32}$ and $J_{33}$.
\end{proof}

%%%%%%%%%%%%%%%%%%%%%%%%%%%%%%%%%%%%%%%%%%%%%%%%%%%%%%%%%%%%

\bigskip

\section*{Acknowledgements}

\bigskip

M.~del Pino has been supported by a UK Royal Society Research Professorship and Grant PAI AFB-170001, Chile. M. Musso has been partly supported by Fondecyt grant 1160135, Chile. The  research  of J.~Wei is partially supported by NSERC of Canada.

\medskip

%%%%%%%%%%%%%%%%%%%%
%\newpage

%\bibliography{localbib}
%\bibliography{mrabbrev,localbib}

%\bibliographystyle{amsalpha-lmp}

%\providecommand{\MRhref}[2]{%
 % \href{http://www.ams.org/mathscinet-getitem?mr=#1}{#2}
%}
%\providecommand{\href}[2]{#2}

%\bibliography{RefDatabase}{}
%\bibliographystyle{plain}

\end{document}